\documentclass[11 pt, oneside]{amsart}
  \usepackage[utf8]{inputenc}
  \usepackage[margin=3cm]{geometry}
\makeatletter
\def\l@subsection{\@tocline{2}{0pt}{2.5pc}{5pc}{}}
\def\l@subsubsection{\@tocline{2}{0pt}{5pc}{7.5pc}{}}
\makeatother

  \usepackage{stmaryrd}
\usepackage{amsmath,amssymb,amsthm, mathrsfs, mathtools, bbold, mathbbol}
\usepackage{amscd,mathtools}
  \usepackage{geometry}
  \usepackage{graphicx,epstopdf}
  \usepackage{color,xcolor}
  \usepackage{enumerate}
  \usepackage{xspace}
  \usepackage{tipa}
  \usepackage[user,titleref]{zref}
  \usepackage{epsfig}

\usepackage{tikz}
\usetikzlibrary{shapes.geometric}
\usetikzlibrary{angles}

\usepackage{caption, subcaption}
\usepackage{tikz-cd}
\usetikzlibrary{calc,decorations.pathreplacing}
\usepackage{tikz}
\usetikzlibrary{calc}
\usetikzlibrary{arrows}
\usetikzlibrary{decorations.pathreplacing}
\usetikzlibrary{intersections}

\usepackage{mathrsfs}
\usetikzlibrary{matrix}
\usetikzlibrary{positioning}
\DeclareMathAlphabet{\pazocal}{OMS}{zplm}{m}{n}
\tikzset{>=stealth}

  \usepackage{hyperref}

\usepackage{float}

  \newcommand{\calA}{\mathcal{A}}
  \newcommand{\calB}{\mathcal{B}}

  \newcommand{\calE}{\mathcal{E}}

  \newcommand{\calN}{\mathcal{N}}
  \newcommand{\calO}{\mathcal{O}}
  \newcommand{\calP}{\mathcal{P}}

  \newcommand{\calS}{\mathcal{S}}
  
  \newcommand{\calU}{\mathcal{U}}
    \newcommand{\calV}{\mathcal{V}}
  \newcommand{\calW}{\mathcal{W}}
  
  \newcommand{\calY}{\mathcal{Y}}
  \newcommand{\calZ}{\mathcal{Z}}
  
    \newcommand{\Uout}{ \calU_{\rm \, Out}}

  \newcommand{\EE}{\mathbb{E}}

  \newcommand{\NN}{\mathbb{N}}

  \newcommand{\RR}{\mathbb{R}}

  \newcommand{\ZZ}{\mathbb{Z}}

  \newcommand{\bfa}{\textbf{a}}
  \newcommand{\bfb}{\textbf{b}}
  \newcommand{\bfc}{\textbf{c}}

  \newcommand{\bfz}{\textbf{z}}  

  \newcommand{\gothic}{\mathfrak}
  \newcommand{\Ga}{{\gothic a}}
  \newcommand{\Gb}{{\gothic b}}
  \newcommand{\Gc}{{\gothic c}}
  
  \newcommand{\Ge}{{\gothic e}}

  \newcommand{\go}{{\gothic o}}

  \newtheorem{theorem}{Theorem}[section]
  \newtheorem{proposition}[theorem]{Proposition}
  \newtheorem{corollary}[theorem]{Corollary}
  \newtheorem{lemma}[theorem]{Lemma}

  \newtheorem{question}[theorem]{Question}
  \newtheorem{claim}[theorem]{Claim}
  
  \newtheorem*{claim*}{Claim}

  \newtheorem{introthm}{Theorem}
  \newtheorem{introcor}[introthm]{Corollary}
  \newtheorem{introprop}[introthm]{Proposition} 
  \newtheorem{introques}[introthm]{Question}

  \theoremstyle{definition}
  \newtheorem{definition}[theorem]{Definition}
  
  \newtheorem{example}[theorem]{Example}
  
  \newtheorem*{question*}{Question}
  \newtheorem*{answer*}{Answer}
  \newtheorem*{application*}{Application}
  \newtheorem{notation}[theorem]{Notation}

  \theoremstyle{remark}
  \newtheorem{remark}[theorem]{Remark}
  \newtheorem*{remark*}{Remark}
  
  \newcommand{\secref}[1]{Section~\ref{#1}}
  \newcommand{\thmref}[1]{Theorem~\ref{#1}}
  \newcommand{\corref}[1]{Corollary~\ref{#1}}
  \newcommand{\lemref}[1]{Lemma~\ref{#1}}
  \newcommand{\propref}[1]{Proposition~\ref{#1}}

  \newcommand{\figref}[1]{Figure~\ref{#1}}
  \newcommand{\defref}[1]{Definition~\ref{#1}}
  
  \newcommand{\eqnref}[1]{Equation~\eqref{#1}}

    \DeclareMathOperator{\Sat}{Sat}

  \newcommand{\pka}{\partial_{\kappa}}

  \newcommand{\Teich}{{Teichm\"uller }} 
  \newcommand{\sQ}{{\sf Q}}

  \renewcommand{\bfa}{{\sf a}}

  \newcommand{\qq}{{\sf q}}   
  \newcommand{\rr}{{\sf r}}

   \newcommand{\eventual}{{\stackrel{e}{=}}}

\DeclareMathOperator{\Cay}{Cay}

\DeclareMathOperator{\deep}{deep}
\DeclareMathOperator{\trans}{trans}

  \newcommand{\param}{{\mathchoice{\mkern1mu\mbox{\raise2.2pt\hbox{$
  \centerdot$}}
  \mkern1mu}{\mkern1mu\mbox{\raise2.2pt\hbox{$\centerdot$}}\mkern1mu}{
  \mkern1.5mu\centerdot\mkern1.5mu}{\mkern1.5mu\centerdot\mkern1.5mu}}}

\DeclarePairedDelimiterX{\norm}[1]{\lvert}{\rvert}{#1}
\DeclarePairedDelimiterX{\Norm}[1]{\lVert}{\rVert}{#1}

  \newcommand{\st}{\mathbin{\mid}} 
  \newcommand{\ST}{\mathbin{\Big|}} 
  \newcommand{\from}{\colon\thinspace}

\newcommand{\CAT}{\ensuremath{\operatorname{CAT}(0)}\xspace}

\begin{document}

  \title[The quasi-redirecting Boundary]
  {The quasi-redirecting Boundary}
  

 \author{Yulan Qing}
 \address{Department of Mathematics,  University of Tennessee at Knoxville, Knoxville, TN, USA}
 \email{yqing@utk.edu}
 
 \author{Kasra Rafi}
 \address{Department of Mathematics, University of Toronto, Toronto, ON, Canada}
\email{rafi@math.toronto.edu}
 


\begin{abstract} 
We generalize the notion of Gromov boundary to a larger class of metric spaces beyond
Gromov hyperbolic spaces. Points in this boundary are classes of quasi-geodesic rays
and the space is equipped with a topology that is naturally invariant under quasi-isometries. 
It turns out that this boundary is compatible with other notions of boundary in many ways; 
it contains the sublinearly Morse boundary as a topological subspace and it matches 
the Bowditch boundary of relative hyperbolic spaces when the peripheral subgroups have no intrinsic 
hyperbolicity.  We also give a complete description of the boundary of the Croke-Kleiner group
where the quasi-redirecting boundary reveals a new class of QI-invariant, Morse-like quasi-geodesics.
\end{abstract}

\maketitle

\tableofcontents

\section{Introduction}
Our goal in this paper is to organize and understand the space of quasi-geodesic rays in a given metric space. 
A quasi-geodesic ray represents a possible direction towards infinity hence the space of 
quasi-geodesic rays could be thought of as the boundary at infinity of a metric space $X$. 
However, different quasi-geodesics may represent the same direction. This is in direct analogy with the 
Gromov boundary defined for a Gromov hyperbolic metric space where two quasi-geodesics represent 
the same point in the Gromov boundary when they fellow travel each other. 
Our approach is to start from first principles and choose definitions that are intuitively natural  and immediately
invariant under a quasi-isometry. More precisely, we would like a notion of boundary where
\begin{enumerate} 
\item Points in the boundary are equivalence classes of quasi-geodesic rays. 
\item The definition of an equivalence class and the topology rely only on the coarse geometry of $X$ 
and hence are invariant under quasi-isometry. 
\item The boundary is as large as possible. 
\end{enumerate} 

To start, we need to decide when two quasi-geodesic rays $\alpha$ and $\beta$ represent the same 
direction in the metric space $X$. Our intuitive answer is that, if there are quasi-geodesic rays  
with uniform constants that travel along $\alpha$ for arbitrary distances and then change course 
and eventually coincide with $\beta$ then traveling in the direction of $\alpha$ does not move 
one away from the direction defined by $\beta$ and hence $\alpha$ and $\beta$ do not represent distinct 
directions. In this case, we say $\alpha$ can be quasi-redirected to $\beta$ and write $\alpha \preceq \beta$
(see Definitions \ref{Def:Redirection} and \ref{Def:P(X)} for precise definitions). 
However, this turns out to not be symmetric in general and $\alpha \preceq \beta$ does not always imply 
$\beta \preceq \alpha$. We let $P(X)$ be a set of equivalence classes of this relation. Then $\preceq$
induces a partial order in $P(X)$. 

\begin{introprop} 
A quasi-isometry between metric spaces $X$ and $Y$ induces a bijection from $P(X)$ to $P(Y)$
that preserves the partial order. 
\end{introprop} 

We think of $P(X)$ as the set of directions in $X$. One could force the relation to be symmetric. 
However, this runs counter to the idea of making the boundary as large as possible. In fact, as we shall see,  
the asymmetry highlights interesting features of shape of the metric space $X$ at infinity 
which is recorded in the set $P(X)$. 

To form a boundary at infinity, we need to put a topology on the set of all directions. However, this cannot 
be done in the setting of general proper metric spaces as they can be quite untamed.  Our main motivation 
is always to study finitely generated groups or spaces quasi-isometric to them. Thus we put some technical 
assumptions on the metric space $X$ to allow for a cone-like topology to be defined on $P(X)$. 
These are marked as Assumption 0, 1 and 2 and we make it clear throughout the paper which assumptions
are used where. Assumption 0 holds for all finitely genereted groups. To our knowledge, we do not know of a  finitely generated group whose Cayley graph does not 
satisfy Assumptions 1 and 2 and we check the validity of these assumptions for a large class of groups. 

\begin{introthm}
Let $X$ be a proper, quasi-geodesic metric space satisfying Assumptions 0, 1 and 2. Then 
the space of directions in $X$ can be organized into a topological space which we denote
by $\partial X$. A quasi-isometry from a metric space $X$ to a metric space $Y$ induces 
a homeomorphism between $\partial X$ to $\partial Y$. 
\end{introthm}

The quasi-isometry invariance allows us to write $\partial G$ for all finitely generated groups $G$. We also 
check the compatibility of this boundary with other generalization of the Gromov boundary. 
The notion of the sublinearly Morse boundary was developed in \cite{QRT1, QRT2}. This is a family 
of boundaries $\partial_\kappa X$ for every sublinear function $\kappa$.  A point in the 
sublinearly Morse boundary is a class of quasi-geodesic rays that resemble Morse geodesics 
where the Morse gauge is allowed to tend to infinity sublinearly with the radius. Two quasi-geodesics are
in the same class if they fellow travel each other sublinearly. The sublinearly Morse boundary was shown to 
be large enough to be used as a topological model for the Poisson boundary in many settings 
(see \cite{QRT1, QRT2}) as well large enough for other measures of genericity \cite{GQR22}.
We show that the quasi-redirecting boundary is an enlargement of the sublinearly Morse boundary.  
 
\begin{introthm}
Let $X$ be a proper, geodesic metric space satisfying Assumptions 0, 1 and 2
and let $\kappa$ be a sublinear function. Then, for every $\kappa$--Morse quasi-geodesic $\alpha$,
the class of quasi-geodesics that sublinearly fellow travel $\alpha$ is the same as the quasi-redirecting 
class of $\alpha$ and hence $\partial_\kappa X \subset \partial X$. In fact, $\pka X$ is a topological 
subspace of $\partial X$.
\end{introthm}

The benefit of this enlargement is that, unlike $\partial_\kappa X$,  $\partial X$ is often compact. 
Also, there are quasi-redirecting classes which exhibit Morse-like properties that are not sublinearly Morse 
(see \secref{Sec:CK}) and hence $\partial X$ encodes strictly more information than the sublinearly 
Morse boundary. The cost of the enlargement is that $\partial X$ is not always metrizable
(in our usage, compact does not imply metrizable). 

\begin{introques} 
Let $X$ be a Cayley graph of a finitely generated group. Is $\partial X$ always defined? Is $\partial X$ always compact? 
\end{introques} 

\subsection*{Relatively hyperbolic groups}
When $X$ exhibits no hyperbolicity, $\partial X$ is trivial (has only one point). 
 This includes spaces with a product structure (Proposition~\ref{product}) and the
Cayley graphs of Baumslag--Solitar groups \cite{Asha}. We refer to such groups and spaces as \emph{mono-directional}.  

The main class of examples we consider is the class of proper geodesic metric spaces that are asymptotically tree-graded with 
respect to mono-directional sets (ATM spaces).  Asymptotically tree-graded spaces were first introduced by Drutu-Sapir in \cite{DS05} and 
systematically studied also in \cite{DS05, DS08} and \cite{sis12}, among others. The idea also appears in \cite{HK05} in similar setting. 

\begin{introthm}
Let $X$ be an ATM space. Then:
\begin{enumerate}[(I)]
\item The space $X$ satisfies Assumptions 0 1, and 2, thus $\partial X$ is defined. In fact, $\preceq$ is 
a symmetric relation on $P(X)$.
\item The boundary $\partial X$ is compact, metrizable and second countable. 
\end{enumerate}
\end{introthm}

These spaces are metric analogues of relative hyperbolic groups. In fact, when $X$ is a Cayley graph of a relative 
hyperbolic group, this theorem gives an alternative description of the Bowditch boundary that is purely based on the 
coarse geometry of $X$, and does not use the algebraic structure of relative hyperbolic groups. 

\begin{introthm}
Let $G$ is a relatively hyperbolic group equipped with the word metric associated to a finite generating set 
such that the peripheral subgroups are mono-directional. 
Then $\partial G$ is homeomorphic to the Bowditch boundary of $G$.  
\end{introthm} 

Hence, when $X$ is an ATM space, we can think of $\partial X$ as the Bowditch boundary of $X$
which is not defined without a group action. Also, using the quasi-redirecting boundary, we see 
how quasi-isometries between relatively hyperbolic groups naturally induce homeomorphisms 
between the Bowditch boundaries. (This statement also follows from \cite[Theorem 4.8]{BDM09}). 

\begin{introcor}
Let $G$ and $G'$ be relatively hyperbolic groups equipped with the word metric associated to a finite generating set
such that the peripheral subgroups are mono-directional. Then any quasi-isometry $\Phi \from G \to G'$ 
induces a homeomorphism between the Bowditch boundaries of $G$ and $G'$. 
\end{introcor}

It would be interesting to know if the metrizability of $\partial X$ is a characterizing property of ATM spaces. 

\begin{introques}
Let $X$ be a geodesic metric space where assumptions 0, 1 and 2 hold. 
Assume $\partial X$ is metrizable and $X$ has a cocompact action by a finitely generated group $G$. 
Does that imply that $G$ is a relative hyperbolic group with respect to mono-directional peripheral subgroups?
\end{introques}

\subsection*{The Croke-Kleiner group}
A good example of a non-positively curved group that is not relatively hyperbolic is  the Croke-Kleiner group,
\[
G = \big \langle a, b, c, d \st [a, b], [b, c], [c,d] \big \rangle,
\]
 which we study in detail in Section~\ref{Sec:CK}. 
The Croke-Kleiner group  is a well-known obstruction to attempts to generalize the Gromov boundary to visual boundaries 
in non-hyperbolic settings  \cite{CK02}. Therefore, it is useful to analyze the group in our current 
generalization of the Gromov boundary. 
It turns out that, certain (but not all) directions in the boundary of the Bass-Serre tree behave like Morse geodesics in a weak sense. 
As expected, this set contains the sublinearly Morse directions, but in this case, it is strictly a larger set. 
The quasi-redirecting boundary is a one point compactification of a set of \emph{Morse-like} directions. 

\begin{introthm}\label{TheoremE}
Let $X$ be the universal cover of the Salvetti complex of the Croke-Kleiner group $G$.  Then $X$ satisfies 
Assumptions 0, 1 and 2, thus $\partial G = \partial X$ is defined. The relation $\preceq$ is not symmetric; 
$P(G)$ has one maximal element and other elements (the minimal elements) are not comparable. 
The set of minimal elements is a strict enlargement of the sublinearly Morse boundary. 
\end{introthm}
 
Our study of the Croke-Kleiner group offers a template to analyze quasi-redirecting boundaries 
 of irreducible right-angled Artin groups, CK-admissible groups, mapping class groups and hierarchically 
 hyperbolic groups where we expect the picture to be similar.

\subsection*{History}
Hyperbolic groups and their boundaries were first introduced by Gromov \cite{gromov}. This notion was 
generalized to many other settings where the group is not hyperbolic but it has some weaker hyperbolic-like 
properties, notably, $\mathrm{CAT}(0)$ groups \cite{gromov, Ger94}, relative hyperbolic groups 
\cite{Bowditch, Farb98}, the mapping class group \cite{MM00}, acylindrical hyperbolic groups \cite{Osin} 
and hierarchically hyperbolic groups \cite{BHS1}. The hyperbolicity in these groups is captured in various 
boundaries, namely the visual boundaries for $\CAT$ spaces \cite{gromov91}, the Bowditch boundary for 
relatively hyperbolic groups \cite{Bowditch}, Thurston boundary of \Teich spaces \cite{FLP},   Furstenberg 
boundary of the symmetric spaces \cite{Fur63}, the Floyd boundary  \cite{Floyd} and horofunction 
boundaries of geodesic metric spaces \cite{Gro81} to name a few.

There has also been many attempts to define a natural boundary that is invariant under quasi-isometries. 
In 2013, the \emph{contracting boundary} of \CAT spaces was constructed by Charney and Sultan 
\cite{contracting},  and is shown to be a first quasi-isometrically invariant geometric boundary in 
non-hyperbolic settings. The construction was generalized to proper geodesic spaces by Cordes in 
\cite{Morse} The Morse boundaries are equipped with a \emph{direct limit topology} and  are invariant under 
quasi-isometries. However, this space does not have good topological properties; for example, it is not first 
countable. Cashen-Mackay \cite{cashenmackay}, following the work of Arzhantseva-Cashen-Gruber-Hume 
\cite{Hume}, defined a different topology on the Morse boundary.
These turn out to be topological subspaces a larger space, namely, the sublinearly Morse boundaries \cite{QRT1, QRT2}. Aside from being QI-invariant and metrizable, sublinear Morse boundaries turn out to be generic in many senses. In the case of right-angled Artin groups, \cite{QRT1} also shows that $\kappa$-Morse boundaries realize Poisson boundaries for $\kappa(t)=\log t$. For mapping class groups, Kaimanovich-Masur showed that uniquely ergodic projective measured foliations with the corresponding harmonic measure can be identified with the Poisson boundary of random walks; Qing-Rafi-Tiozzo \cite{QRT2} showed that, when $\kappa=\log t$, the $\kappa$-boundary of the Cayley graph of the mapping class group can be identified with the Poisson boundary of the associated random walks. Meanwhile, genericity of a more geometric flavor is also exhibited for sublinearly Morse boundaries. In \cite{GQR22}, genericity of sublinearly Morse directions under Patterson Sullivan measure was shown to hold in a more general context of actions admitting a strongly contracting element. In fact, the results in \cite{GQR22} concerning stationary measures were recently obtained in this more general setting by Inhyeok Choi \cite{Choi}, who in place of our ergodic theoretic and boundary techniques uses a pivoting technique developed by Gou\"ezel\cite{Gou22}. Meanwhile, a Patterson-Sullivan theory on a certain quotient of the horofunction boundary for spaces admitting non-elementary group actions with contracting elements was recently obtained by Coulon \cite{Coulon} and Yang \cite{wenyuan}. Following  \cite{wenyuan}, genericity of sublinearly Morse directions on the horofunction boundary was recently shown for all proper statistically convex-cocompact actions on proper metric spaces \cite{QY24}. Furthermore, sublinearly Morse directions are invariant even under sublinear biLipschitz equivalence between metric spaces \cite{PQ}. A compact metrizable boundary was introduced by Dydac and Rashed in \cite{DR22} using $C^*$-algebra. However, this boundary 
seems to be small in many settings and does not contain the sublinearly Morse boundary as a topological subspace.

\subsection*{Acknowledgments}
The authors thank Giulio Tiozzo for helpful discussions during the conception of the project. We also thank Vivian He, Chris Hruska and 
Andrew Zimmer for useful conversations. We are grateful to Asha McMullin and Hoang Thanh Nguyen for feedbacks on the early drafts of the paper. The second named author was partially supported by NSERC Discovery grant RGPIN 05507. 

\section{Preliminaries}
In this section we recall some basic definitions and set up a few notations (see Notation~\ref{notation}). 
We also present a few old and new surgery constructions between quasi-geodesics as many arguments in 
this paper involve constructing quasi-geodesics with controlled constants. We also discuss Assumption 0 
which is the most basic assumption on all the metric space we will be considering.

\begin{definition}[Quasi Isometric embedding] \label{Def:Quasi-Isometry} 
Let $(X , d_X)$ and $(Y , d_Y)$ be metric spaces. For constants $q \geq 1$ and
$Q \geq 0$, we say a map $\Phi \from X \to Y$ is a 
$(q, Q)$--\emph{quasi-isometric embedding} if, for all $x_1, x_2 \in X$
$$
\frac{1}{q} d_X (x_1, x_2) - Q  \leq d_Y \big(\Phi (x_1), \Phi (x_2)\big) 
   \leq q \, d_X (x_1, x_2) + Q.
$$
If, in addition, every point in $Y$ lies in the $Q$--neighbourhood of the image of $\Phi$, then $\Phi$ is called 
a $(q, Q)$--quasi-isometry. This is equivalent to saying that $\Phi$ has a \emph{quasi-inverse}. That is, 
there exist constants $q', Q'>0$ and a $(q',Q')$--quasi-isometric embedding $\Psi \from Y \to X$ such that,
\[
\forall x \in X \quad d_X\big(x, \Psi\Phi(x)\big) \leq Q' \qquad\text{and}\qquad
\forall y \in Y \quad d_Y\big(y, \Phi\,\Psi(x)\big) \leq Q'.
\]
When such a map $\Phi$ exists, we say $(X, d_X)$ and $(Y, d_Y)$ are \textit{quasi-isometric}. 
\end{definition}

\begin{definition}[Quasi-Geodesics] \label{Def:Quadi-Geodesic} 
A \emph{quasi-geodesic} in a metric space $X$ is a quasi-isometric embedding $\alpha \from I \to X$ where $I \subset \RR$ is an 
(possibly infinite) interval. That is, $\alpha \from I \to X$ is a $(q,Q)$--quasi-geodesic if, for all $s, t \in I$, we have 
\[
\frac{|t-s|}{q} - Q  \leq d_X \big(\alpha(s), \alpha(t)\big)  \leq q \cdot |s-t|. 
\]
Furthermore, without loss of generality (see Lemma~\ref{Lem:Tame}), in this paper we always assume $\alpha$ is $(2q+2Q)$--Lipschitz, 
in particular, $\alpha$ is continuous. 
\end{definition}

The assumption that $\alpha$ is Lipschitz is needed so we can apply the Arzel\`a-Ascoli theorem to a sequence of quasi-geodesics 
to obtain a limiting quasi-geodesic. However, this assumption can always be achieved by increasing 
the constants of the quasi-geodesic:

\begin{lemma}[Taming of the quasi-geodesics] \label{Lem:Tame}
Let $X$ be a geodesic metric space and let $I \subset \RR$ be an interval of length bigger than 1. 
Given a $(q,Q)$--quasi-isometric embedding $\alpha \from  I \to X$, one can find a
$(q',Q')$--quasi-geodesic $\alpha' \from I \to X$, with $q'$ and $Q'$ depending only on $q$ and $Q$, 
that is $2(q+Q)$--Lipschitz and fellow travels $\alpha$. In fact, for $t \in I$, we have 
\[
d_X\big(\alpha(t), \alpha'(t)\big) \leq 2(q +Q). 
\]
\end{lemma}
\begin{proof}
We assume $I$ is compact, the proof for other cases is similar. 
Let $k$ be an integer larger than the length of $I$ and choose times $t_0 \leq t_1 \leq \dots \leq t_k$ 
such that  $\tfrac 12 \leq |t_{i+1}- t_i| \leq 1$ and $I = [t_0, t_k]$. For $i =0, \dots, k$, define 
$\alpha'(t) = \alpha(t)$. Then define $\alpha'[t_i, t_{i+1}]$ to be a geodesic segments connecting 
$\alpha(t_i)$ to $\alpha(t_{i+1})$. The length of each of these geodesic segments is at most $(q+Q)$ 
and the length of $[t_i, t_{i+1}]$ is at least $\tfrac 12$. Hence $\alpha'$ is $2(q+Q)$--Lipschitz. For every 
$t \in I$, there is $t_i$ such that $|t-t_i| \leq \frac 12$. Hence
\[
d_X(\alpha(t), \alpha'(t)) \leq d_X(\alpha(t), \alpha(t_i)) + d_X(\alpha'(t_i), \alpha'(t)) 
\leq  \frac{q}2 + Q + \frac{q+Q}2 \leq  q + 2Q. 
\]
To see the lower bound for $\alpha'$, let $t, s \in I$ and let $t_i$ and $t_j$ be such that 
$|t- t_i| \leq \frac 12$ and $|s- t_j| \leq \frac 12$. Then 
\begin{align*}
d_X(\alpha'(t), \alpha'(s)) 
& \geq d_X(\alpha'(t_i), \alpha'(t_j))  - d_X(\alpha'(t), \alpha'(t_i))  - d_X(\alpha'(s), \alpha'(t_j)) \\
&\geq \frac{|t_i - t_j|}{q} - Q - \frac{(q+Q)}2 - \frac{(q+Q)}2      \\
& \geq \frac{|s-t| -1}{q} - q-2Q.
\end{align*}
That is, the Lemma holds for 
\begin{equation*} 
q' =q+Q \qquad\text{and}\qquad  Q'= q + \frac 1q + 2Q.  \qedhere
\end{equation*} 
\end{proof}

\begin{notation} \label{notation}
To simplify notation, we use $\qq=(q, Q) \in [1, \infty) \times [0, \infty)$ to indicate a pair of  constants.
That is, we say $\Phi \from X \to Y$ is a $\qq$--quasi-isometry. We also say $\alpha$ is 
$\qq$--quasi-geodesic, which can be a ray and/or a segment depending on the context. Furthermore, we fix a base point $\go$ in the metric space $X$. By a \emph{$\qq$--ray} 
we mean a $\qq$--quasi-geodesic ray $\alpha \from [0, \infty) \to X$ such that $\alpha(0) = \go$. 
For an interval $[s, t]\subset [0, \infty)$, we denote the restriction of $\alpha$ to the time interval 
$[s,t]$ by $\alpha [s,t]$ (simplified from $\alpha([s,t])$). However, if points $x, y \in X$ on the image of 
$\alpha$ are given, we denote the sub-segment of $\alpha$ connecting $x$ to $y$ by $[x,y]_\alpha$. 
That is, if $\alpha(s) = x$ and $\alpha(t) = y$ for $s \leq t$, then $[x,y]_\alpha = \alpha[s,t]$. 

We often need to concatenate quasi-geodesics. Let $\alpha \from [s_1, s_2] \to X$
and $\beta \from [t_1, t_2] \to X$ be two quasi-geodesics such that $\alpha(s_2) = \beta(t_1)$. 
We denote the concatenation of $\alpha$ and $\beta$ by $\alpha \cup \beta$ by which we mean 
the following quasi-geodesic:
\[
\alpha \cup \beta \from [s_1, t_2- t_1 + s_2] \to X,  
\qquad 
\alpha\cup\beta(t) = 
\begin{cases} 
\alpha(t) & \text{for $ t \in [s_1, s_2]$}\\
\beta(t+t_1-s_2) & \text{for $ t \in [s_2, t_2-t_1+s_2]$}
\end{cases} .
\]

For $r>0$, let $B_r^\circ \subset X$ be the open ball of radius $r$ centered at $\go$, let $B_r$ be the 
closed ball centered at $\go$ and let $B^{c}_{r}= X - B_r^\circ$. For a $\qq$--ray $\alpha$ and $r>0$, we let 
$t_r \geq 0 $ denote the first time $\alpha$ intersects 
$B^{c}_{r}$ and $T_r \geq t_r$ be the last time $\alpha$ intersects $B_r$. 
We denote $\alpha(t_r)$ by $\alpha_{r} \in X$.  Also, let 
\[
\alpha|_r =\alpha[0,t_r] \qquad\text{and}\qquad \alpha|_{\geq r} =\alpha[T_r, \infty)
\] 
be the restrictions $\alpha$ to the intervals $[0, t_r]$ and $[T_r, \infty)$ respectively. 
That is, $\alpha|_r $ is the sub-segment of $\alpha$ connecting $\go$ to $\alpha_r$
and $\alpha|_{\geq r}$ is the portion of $\alpha$ that starts at radius $r$ but never returns to 
$B_r$. Lastly, if $p$ is a point on a $\qq$--ray $\alpha$, we also use $\alpha_{[p, \infty)}$ 
to denote the tail of $\alpha$ starting from the point $p$.

We use some short hands for pairs of constants. For example, for a given pairs $\qq$ and $\qq'$
we write $\qq \leq \qq'$ if 
\[
\qq=(q, Q), \quad \qq' = (q', Q'), \quad q \leq q' \quad \text{and}\quad Q \leq Q'.
\] 
Also, $\qq = \max(\qq_1, \qq_2)$ means
\[
\qq_1 = (q_1, Q_1),\quad \qq_2 = (q_2, Q_2), \quad q = \max(q_1, q_2)\quad 
\text{and}\quad  Q = \max(Q_1, Q_2).
\]

We also use $d(\param, \param)$ instead of $d_X(\param, \param)$ when the metric space $X$ is fixed.
For $x \in X$, $\Norm{x}$ denotes $d(\go, x)$.  Let $A \subset X$ be a set and $D>0$, then 
\[
N_D(A) : = \{ x \in X | \, \exists \, a \in A \text{ where } d(x, a) \leq D\}
\] 
\end{notation} 

\subsection*{Preliminary assumption}
General metric spaces could be very wild and difficult to work with. Hence, we make some assumptions 
about the space $X$ which we show they holds for our main examples (see \secref{Sec:ATM} and 
\secref{Sec:CK}). Some of the statements in this paper can be stated in a more general setting, but these 
assumptions simplify the exposition and exclude certain exotic examples.

\subsection*{Assumption 0} (No dead ends) The metric space $X$ is always assumed to be a proper, 
geodesic metric space. Furthermore, there exist a pair of constants $\qq_0$ such that every point $x \in X$ 
lies on an infinite $\qq_0$--ray.  

Recall that every proper quasi-geodesic metric space is quasi-isometric to a proper geodesic metric space 
(see for example \cite[Proposition 5.3.9]{Clara}) via a process similar to \lemref{Lem:Tame}. So the first condition 
in the Assumption 0 is not a strong assumption.  The second condition in Assumption 0 holds for the spaces
we are most concerned about, namely, any space quasi-isometric to a finitely generated group. 
Note that, many groups such that lamplighter groups, have dead ends in the sense that not every point
lies on an infinite geodesic ray. In fact it is shown that there exists wreath products with unbounded dead-end length with respect to certain standard generating sets\cite{Cleary}.  However we prove now that all spaces quasi-isometric to a finitely generated group satisfies Assumption 0. This result is proven independently in \cite[Theorem 3.3]{thm33}.

\begin{lemma} \label{Lem:cocompact} 
Let $X$ be a proper geodesic metric space with a cocompact action by a finitely generated group $G$. Then $X$ satisfies Assumption 0.
\end{lemma}

\begin{proof}
If $\Phi \from X \to Y$ is a quasi-isometry between two proper geodesic metric spaces, then 
Assumption 0 holds for $X$ if an only if it holds for $Y$. This is because quasi-geodesics in 
$X$ are mapped to quasi-isometric embeddings of intervals to $Y$ which can be tamed using 
\lemref{Lem:Tame}. Hence, it is enough to prove the lemma for the case when $X$ is the Cayley graph 
of $G$ with respect to a finite generating set. 

We first argue that $X$ contains a bi-infinite geodesic ray $\gamma$. Pick a sequence of 
point $x_n$ such that $\Norm{x_n} \to \infty$, let $\gamma_n$ be a geodesic in $X$
connecting $\go$ to $x_n$ and let $y_n$ be a point on $\gamma_n$ such that both 
$\Norm{y_n} \to \infty$ and $d(y_n, x_n) \to \infty$. Note that we can think of $y_n$
as an element of $G$. Since $X$ is proper, the sequence of geodesic segments 
$y_n^{-1}(\gamma_n)$ converges, up to taking a subsequence, to a bi-infinite geodesic 
$\gamma$ passing through $\go$. 

Now let $x$ be a point in $X$ and let $\gamma_x = x \cdot \gamma$ be a bi-infinite geodesic passing
thorough $x$. Let $z$ be a point on $\gamma_x$ that is closest to $\go$. Then $z$ divides 
$\gamma_x$ into two half-infinite geodesics $\gamma_x^+$ and $\gamma_x^-$ starting at $z$.  
Let $\gamma_x^+$ be the half-infinite geodesic that contains $x$. By Part (II) of \lemref{surgery}, the concatenations 
$[\go, z] \cup \gamma_x^+$ is a $(3, 0)$--quasi-geodesic ray emanating from $\go$
passing through $x$. 
\end{proof}

\subsection{Surgeries between quasi-geodesics}
In this section we present several methods to produce a quasi-geodesic as a concatenation 
of other geodesics or quasi-geodesics. The statements are intuitively clear and the proofs are elementary. 
So, this subsection could be skipped on the first reading of the paper. 
First, we recall a few surgery lemmas from \cite{QRT1} and \cite{QRT2}. 

\begin{lemma}\label{surgery}
Let $X$ be a metric space satisfying Assumption 0. 
The following statements are \cite[Lemma 2.5, Lemma 4.3]{QRT1} and \cite[Lemma 3.7]{QRT2} 
respectively. 
 
\begin{enumerate}[(I)]
\item\label{concate}(Nearest-point projection)
Consider a point $x \in X$ and a $(q, Q)$--quasi-geodesic segment $\beta$ 
connecting a point $z \in X$ to a point $w \in X$. Let $y$ be a closest point in $\beta$
to $x$. Then 
\[
\gamma = [x, y] \cup [y, z]_\beta
\] 
is a $(3q, Q)$--quasi-geodesic.

\begin{figure}[h!]
\begin{tikzpicture}[scale=0.9]
 \tikzstyle{vertex} =[circle,draw,fill=black,thick, inner sep=0pt,minimum size=.5 mm]
[thick, 
    scale=1,
    vertex/.style={circle,draw,fill=black,thick,
                   inner sep=0pt,minimum size= .5 mm},
                  
      trans/.style={thick,->, shorten >=6pt,shorten <=6pt,>=stealth},
   ]

  \node[vertex] (z) at (0,0)[label=left:$z$] {}; 
  \node[vertex] (w) at (8,0) [label=right:$w$]  {}; 
  \node[vertex] (x) at (4,1.5) [label=right:$x$]  {};     
  \node[vertex] (y) at (4,0) [label=below:$y$]  {};     
  \node (a) at (.9,.8) [label=right:$\beta$]  {};    
  \draw[thick, dashed]  (x)--(4, 0){};
 
  \pgfsetlinewidth{1pt}
  \pgfsetplottension{.75}
  \pgfplothandlercurveto
  \pgfplotstreamstart
  \pgfplotstreampoint{\pgfpoint{0cm}{0cm}}  
  \pgfplotstreampoint{\pgfpoint{1.4cm}{-.6cm}}   
  \pgfplotstreampoint{\pgfpoint{1.3cm}{.2cm}}
  \pgfplotstreampoint{\pgfpoint{3cm}{-.4cm}}
  \pgfplotstreampoint{\pgfpoint{4cm}{0cm}}
  \pgfplotstreampoint{\pgfpoint{5cm}{-.2cm}}
  \pgfplotstreampoint{\pgfpoint{6cm}{.3cm}}
  \pgfplotstreampoint{\pgfpoint{7cm}{-.7cm}}
  \pgfplotstreampoint{\pgfpoint{8cm}{0cm}}
  \pgfplotstreamend
  \pgfusepath{stroke} 
  \end{tikzpicture}
  
\caption{The concatenation of the geodesic segment $[x,y]$ 
and the quasi-geodesic segment $[y,z]_\beta$ is a quasi-geodesic.}
\label{Fig:omega} 
\end{figure}
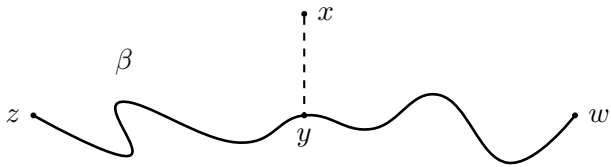

\item\label{redirect11} (Quasi-geodesic ray to geodesic ray) Let $\beta$ be a geodesic ray and 
$\gamma$ be a $(q, Q)$--ray. 
For $r>0$, assume that $d_X(\beta_\rr, \gamma)\leq r/2$. Then, there exists a
$(9q,Q)$--quasi-geodesic $\gamma'$ where $\gamma'(t) =\beta(t)$ for large values of $t$ and 
\[
\gamma|_{r/2} = \gamma'|_{r/2}. 
\]

\end{enumerate}
\end{lemma}

\begin{lemma}\label{onestep}(Fellow travelling)
Let $X$ be a metric space satisfying Assumption 0. Let $\qq$--rays $\alpha, \beta$, $t_0>0$ and $C>0$
be such that, for all $t \leq t_0$, we have 
\[ 
d(\alpha(t), \beta(t)) \leq C.
\]
Then there exists a $(q, Q+C)$--quasi-geodesic ray $\beta'$ such that 
\[
\beta'|_{t_0} =\beta|_{t_0} \qquad\text{and}\qquad \beta'|_{(t_0+1, \infty)} =\alpha|_{(t_0, \infty)} .
\]
\end{lemma}
\begin{proof}
Let $\gamma \from [0,C] \to X$ be a geodesic segment connecting $\beta(t_0)$ and 
$\alpha(t_0)$. Define $\beta'$ as:
\[
\beta'(t) = \begin{cases}
\beta(t) &\text{for } t \in[0, t_0]\\
\gamma(t-t_0) &\text{for } t \in[t_0, t_0+C]\\
\alpha(t-C) &\text{for } t \geq t_0+C
\end{cases}.
\]
We claim that $\beta'$ is a $(q, Q+ C)$--quasi-geodesic ray. Given two points $\beta'(t_1)$ and 
$\beta'(t_2)$. First we consider the case when $t_1 < t_0$ and $t_2 \geq t_0+C$.  By assumption 
$\beta'(t_1) = \beta(t_1)$ and
\[ 
d(\beta(t_1), \alpha(t_1)) \leq C.
\]
Thus 
\[
d(\beta'(t_1), \beta'(t_2)) \leq  
  d(\alpha(t_1), \alpha(t_2-C))+C \leq q (t_2-C -t_1) + Q +C \leq q (t_2 -t_1) + Q.
\]

On the other hand we have
\begin{align*}
d(\beta'(t_1), \beta'(t_2)) &\geq  d(\alpha(t_1), \alpha(t_2-C))-C \\
                                      &\geq \frac 1q  (t_2-C -t_1)  -Q -C.
\end{align*} 
Therefore, $\beta'$ is a $(q, Q+C)$--quasi-geodesic redirecting $\beta$ to $\alpha$. 

Another case to consider is when $t_1 \leq t_0$ and $t_0 \leq t_2 \leq t_0+C$. In this case \[d(\beta(t_0), \beta(t_2))<C.\] Thus we have
\begin{align*}
d(\beta(t_1), \beta(t_2)) &\leq d(\beta(t_1), \beta(t_0))+C \\
                                      &\leq q |t_0 -t_1| + Q +C \\
                                      & \leq q|t_2-C -t_1| + Q +C \leq q|t_2-t_1| + Q.
\end{align*}
On the other side we have 
\begin{align*}
                                      d(\beta(t_1), \beta(t_2)) &\geq  d(\beta(t_1), \beta(t_0))-C \\
                                      &\geq \frac 1q  (t_2-C -t_1)  -Q -C.
\end{align*}                                      
The case where $t_0 \leq t_1 \leq t_0+C$ and $t_2 \geq t_0+C$ is analogous. Other cases are trivial.
\end{proof}

\begin{lemma}[Pass through a nearby point]\label{twostep}
Let $X$ be a metric space satisfying Assumption 0. 
Let $\alpha$ be a $(q,Q)$--ray, $x \in X$ and let $t_0>0$ be such that 
\[ 
\ell : =d(x, \alpha(t_0)) \leq 1.
\]
Then there exists a $(q, Q+3)$--quasi-geodesic ray $\alpha'$ such that $x = \alpha'(t_0+1)$ and 
\[
\alpha'|_{t_0} =\alpha|_{t_0} \qquad\text{and}\qquad \alpha'|_{(t_0+2, \infty)} =\alpha|_{(t_0, \infty)} .
\]
\end{lemma}
\begin{proof}
Let $\gamma \from [0,1] \to X$ be a geodesic segment connecting $\alpha(t)$ and $x$
parametrized with constant speed.  Define $\alpha'$ as:
\[
\alpha'(t) = \begin{cases}
\alpha(t) &\text{for } t \in[0, t_0],\\
\gamma(t-t_0)&\text{for } t \in[t_0, t_0+1],\\
\gamma(t-t_0-1)&\text{for } t \in[t_0+1, t_0+2],\\
\alpha(t-2\ell) &\text{for } t \geq t_0+2.
\end{cases}
\]
We claim that $\alpha'$ is a $(q, Q+3)$-quasi-geodesic ray. Given two points $\alpha'(t_1)$ and $\alpha'(t_2)$. First we consider the case 
when $t_1 < t_0$ and $t_2 \geq t_0+\ell$.  By assumption 
$\alpha'(t_1) = \alpha(t_1)$ and $\alpha'(t_2) = \alpha(t_2-2)$. Thus 
\[
d(\alpha'(t_1), \alpha'(t_2)) = d(\alpha(t_1), \alpha(t_2-2)) \leq q (t_2-2 -t_1) + Q  \leq q (t_2-t_1) + Q.
\]
On the other hand we have
\begin{align*}
d(\alpha'(t_1), \alpha'(t_2)) &\geq  d(\alpha(t_1), \alpha(t_2-2)) \\
                                      &\geq \frac 1q  (t_2-2 -t_1)  -Q \\
                                      &=\frac 1q  (t_2 -t_1)  -Q -\frac{2}{q} 
                                      \geq \frac 1q  (t_2 -t_1)  -(Q+2).
\end{align*} 
Another case to consider is when $t_1 \leq t_0$ and $t_0 \leq t_2 \leq t_0+2$. In this case 
\begin{align*}
d(\alpha'(t_1), \alpha'(t_2)) &\leq d(\alpha(t_1), \alpha(t_0)) + \ell \\
                                      &\leq q |t_0 -t_1| + Q +  \ell \leq q |t_2 -t_1| + (Q +2).
                                      \end{align*}
On the other side we have 
\begin{align*}
 d(\alpha'(t_1), \alpha'(t_2)) &\geq  d(\alpha(t_1), \alpha(t_0))  -\ell\\
                                     &\geq \frac 1q |t_0-t_1| -Q-\ell\\
                                     &\geq \frac 1q(t_2-t_1 - 2)-Q-\ell  \geq \frac 1q(t_2-t_1)-(Q+3).
\end{align*}                                      
The case where $t_0 \leq t_1 \leq t_0+ 2\ell$ and $t_2 \geq t_0+ 2 \ell $ is analogous. Other cases are trivial.
\end{proof}

\section{Equivalence classes of rays up to quasi-redirection} \label{Sec:Partial_order}
As previously stated we assume throughout that Assumption 0 holds. In this section, we define a preorder 
$\preceq$ on the set of quasi-geodesic rays. The set of equivalence classes associated to $\preceq$ form 
a partially ordered set $P(X)$. Elements of $P(X)$ will later serve as points in our boundary $\partial X$ 
(see \secref{Sec:boundary}). We will also show in \secref{Sec:boundary} that equivalent classes
associated to sublinearly Morse quasi-geodesics rays are the minimal elements with respect to this partial 
order.

Roughly speaking, for quasi-geodesic rays $\alpha$ and $\beta$, $\alpha \preceq \beta$
if $\alpha$ can be \emph{quasi-redirected} to $\beta$, that is, if there is a family of quasi-geodesic 
rays with uniform constants that coincide with $\alpha$ in the beginning for an arbitrarily long time 
but eventually coincide with $\beta$. 

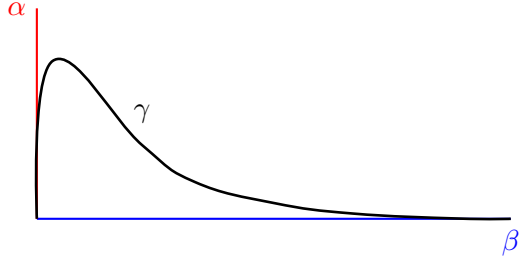
\begin{figure}[h!]
\centering
\begin{tikzpicture}[scale=0.7] 
\draw[blue, thick] (0,0) -- (9,0) node[below] {$\beta$};
\draw[red, thick] (0,0) -- (0,4) node[left] {$\alpha$};
\node at (2,2) {$\gamma$};

 \pgfsetlinewidth{1pt}
 \pgfsetplottension{.75}
 \pgfplothandlercurveto
 \pgfplotstreamstart
 \pgfplotstreampoint{\pgfpoint{0 cm}{0cm}}  
 \pgfplotstreampoint{\pgfpoint{.3 cm}{3 cm}}   
 \pgfplotstreampoint{\pgfpoint{2 cm}{1.4 cm}}
 \pgfplotstreampoint{\pgfpoint{3 cm}{.7 cm}}
 \pgfplotstreampoint{\pgfpoint{4.5 cm}{.3 cm}}
 \pgfplotstreampoint{\pgfpoint{6 cm}{.1 cm}} 
 \pgfplotstreampoint{\pgfpoint{8 cm}{.005 cm}} 
 \pgfplotstreampoint{\pgfpoint{9 cm}{0 cm}} 
 \pgfplotstreamend
 \pgfusepath{stroke} 
 \end{tikzpicture}
\caption{The ray $\alpha$ can be quasi-redirected to $\beta$ at radius $r$.}
\end{figure}

\begin{definition}\label{Def:Redirection}
Let $\alpha, \beta$ and $\gamma$ be quasi-geodesic rays. We say $\beta$ \emph{eventually coincide with 
$\gamma$} (and write $\gamma \eventual \beta$) if there are times $t_\gamma>0$
and $t_\beta$ (which maybe negative) such that, 
for $t\geq t_\gamma$, we have 
\[
\gamma(t) =\beta(t+t_\beta).
\]
For $r>0$, we say $\gamma$ \emph{quasi-redirects $\alpha$ to $\beta$ at radius $r$} if
\[
\gamma|_r = \alpha|_r \qquad\text{and} \qquad \beta \eventual \gamma. 
\]
If $\gamma$ is a $\qq$--ray, we say \emph{$\alpha$ can be $\qq$--redirected to $\beta$
at radius $r$}. We refer to $(t_\beta+t_\gamma)$ as the \emph{landing time}. 	
We say $\alpha \preceq \beta$, if there is $\qq \in [1, \infty) \times[0,\infty)$ such that, 
for every $r >0$, $\alpha$ can be $\qq$--redirected to $\beta$ at radius $r$. \end{definition}

\begin{lemma}[Quasi-redirection is transitive] \label{L:Transitive}
Let $\alpha, \beta, \gamma$ be  quasi-geodesic rays. If $\alpha$ can be $(q_1, Q_1)$--redirected to 
$\beta$ at every radius $r>0$ and $\beta$ can be $(q_2, Q_2)$--redirected to $\gamma$ at every
radius $r>0$, then $\alpha$ can be $(q_3, Q_3)$--redirects to $\gamma$ at every radius $r>0$ where 
\[
q_3 = \max \big\{ q_2+ 1, q_1 \big\}, \, \text{ and }  \, Q_3= \max \big\{ Q_1, Q_2 \big\}.
\]
Hence, the relation $\preceq$ is transitive, that is, if $\alpha \preceq \beta$ and $\beta \preceq \gamma$
then $\alpha \preceq \gamma$.
\end{lemma}

\begin{proof}
Consider $r > 0$ and let $t_1$ be the first time such that the $\Norm{\alpha(t_1)} = r$. 
Let $\zeta_1$ be a  $(q_1, Q_1)$--quasi-geodesic ray quasi-redirecting $\alpha$
to $\beta$ at radius $r$. Let $t_2 >0$ and $s_1 \in \RR$ be such that for all $t \in [t_2, \infty)$, 
\[\zeta_1 (t) = \beta(t+s_1).\]
Let $t_3>0$ be large enough such that
\begin{equation} \label{Eq:first-t_3} 
\left(\frac{1}{q_2}- \frac{1}{q_3}\right) t_3 
  \geq q_1 t_2 - Q_3 +Q_1 + Q_2 + \frac{|s_1|}{q_2},
\end{equation} 
and 
\begin{equation} \label{Eq:second-t_3} 
(q_3 - q_2) \, t_3 \geq q_1 t_2 - Q_3 +Q_1 + Q_2 + q_2 |s_1|. 
\end{equation} 
Let $r' : = \Norm{\beta(t_3)}$ and let $\zeta_2$ be  a  $(q_2, Q_2)$--quasi-geodesic ray redirecting
 $\beta$ to $\gamma$ at radius $r'$.  Let $t_4>0$ and $s_2 \in \RR$ be such that for all 
 $t \in [t_4, \infty)$, 
\[\zeta_2 (t) = \gamma(t+s_2).\]
Now let $\zeta$ be a ray defined as follows: 
\[
\zeta \from \RR_+ \to X, \qquad 
\zeta(t) = \begin{cases}
\zeta_1(t) &\text{for } t \in[0, t_3],\\
\zeta_2(t+s_1) &\text{for } t \in [t_2, \infty).
\end{cases}
\]
Note that the two intervals $[0, t_3]$ and $ [t_2, \infty)$ overlap. However, for $t \in [t_2, t_3]$,
we have 
\[ \zeta(t) = \zeta_1(t) = \beta(t+s_1) = \zeta_2(t+s_1).\]
We claim that $\zeta$ is a $(q_3, Q_3)$--quasi-geodesic ray. 

\begin{figure}[h!]
\begin{tikzpicture}
\tikzstyle{vertex} =[circle,draw,fill=black,thick, inner sep=0pt,minimum size=2pt] 

   \draw[thick] (0,0) -- (10,0) node[right]{$\gamma$};
   \draw[thick] (0,0) -- (8,2) node[right]{$\beta$};
    draw[thick] (0,0) -- (3,3) node[right]{$\alpha$};

  \draw [red, thick, rounded corners] (0,0) -- (1.1 , 0.9)
           -- (2.15,0.4) -- (5.1,1.1) -- (7.6,-0.1) -- (10, -0.1) node[below] {$\zeta$} ; 
  \draw [blue,thick, rounded corners] (0 ,0) -- (1.1,1) -- (2,0.6) -- (7.95, 2.1) node[above] {$\zeta_1$};

  \colorlet{darkgreen}{green!60!black}      
  \draw [darkgreen ,thick, rounded corners] (0 ,0) -- (5.1,1.2) -- (7.5, 0.1) --(10, 0.1) 
      node[above] {$\zeta_2$};
\end{tikzpicture}
\caption{The ray $\zeta$, which is constructed from ray $\zeta_1$ and $\zeta_2$, quasi-redirects
$\alpha$ to $\gamma$.}
\end{figure}
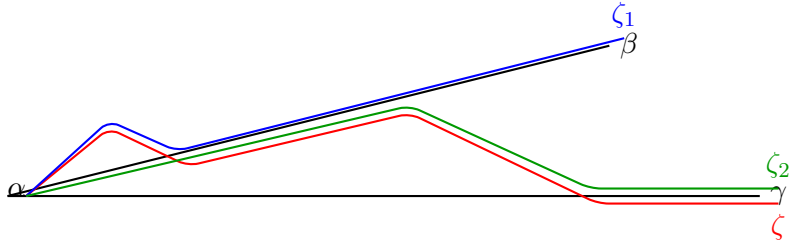

Let $x, y$ be points along $\zeta$ where $x = \zeta(t_x) $ and $y = \zeta(t_y)$. 
There are several cases. If $t_x, t_y \leq t_3$, then $\zeta(t_x) = \zeta_1(t_x)$ and 
$\zeta(t_y) = \zeta_1(t_y)$, and hence
\[ 
 \frac{1}{q_1} (t_y - t_x) -Q_1\leq d(x, y) 
    = d\big(\zeta_1(t_x), \zeta_1(t_y)\big) \leq q_1 (t_y - t_x) + Q_1.
\]
But $q_1 \leq q_3$ and $Q_1 \leq Q_3$, hence the claim holds for these times. 

Likewise if $t_x, t_y \geq t_2$, then  $\zeta(t_x) = \zeta_2(t_x+s_1)$, $\zeta(t_y) = \zeta_2(t_y+s_1)$
and hence
\[ 
 \frac{1}{q_2} (t_y - t_x) -Q_2\leq d(x, y) = d(\zeta_2(t_x+s_1), \zeta_2(t_y + s_1)) \leq q_2 (t_y - t_x) + Q_2.
\]
But $q_2 \leq q_3$ and $Q_2 \leq Q_3$, hence the claim also holds for these times.

It remains to consider the case where $t_x \in [0, t_2]$ and $t_y \in [t_3, \infty)$. We have 
\begin{equation} \label{Eq:Norms}
\frac 1{q_1} t_x -Q_1 \leq \Norm x \leq q_1 t_x +Q_1 
\quad \text{ and } \quad 
\frac {1}{q_2} (t_y +s_1)-Q_2 \leq \Norm y \leq q_2 (t_y +s_1)+Q_2.
\end{equation}
Therefore, 
\begin{align*}
d(x, y) & \geq \Norm y - \Norm x \quad \tag{triangle inequality}\\
           & \geq \frac {1}{q_2} (t_y +s_1)-Q_2- q_1 t_x -Q_1   
             \tag{Equation \eqref{Eq:Norms}} \\
            &\geq \frac{1}{q_3}  t_y  -Q_3, \tag{Equation~\eqref{Eq:first-t_3}}
\end{align*}
and 
\begin{align*}
d(x, y) & \leq \Norm y + \Norm x \quad \tag{triangle inequality}\\
           & \geq q_2 (t_y +s_1)+Q_2+ q_1 t_x +Q_1   
             \tag{Equation \eqref{Eq:Norms}} \\
            &\geq q_3  \, t_y  +Q_3. \tag{Equation~\eqref{Eq:second-t_3}}
\end{align*}
That is, $\zeta$ is a $(q_3, Q_3)$--quasi-geodesic ray. The argument holds for any $r>0$.
Hence $\alpha$ can be  $(q_3, Q_3)$--redirected to $\gamma$ at every radius $r > 0$.
\end{proof}

Since we also have $\alpha \preceq \alpha$ for every quasi-geodesic ray, $\preceq$ is a preorder
on the set of quasi-geodesic rays. 

\begin{definition} \label{Def:P(X)} 
Define $\alpha \simeq \beta$ if and only if $\alpha  \preceq \beta$ and $\beta  \preceq \alpha$. Then  
$\simeq$ is an equivalence relation on the space of all quasi-geodesic rays in $X$. Let $P(X)$
denote the set of all equivalence classes of quasi-geodesic rays under $\simeq$. 
For a quasi-geodesic ray $\alpha$, let $[\alpha] \in P(X)$ denote the equivalence class containing 
$\alpha$. We extend $\preceq$ to $P(X)$ by defining $[\alpha] \preceq [\beta]$ if 
$\alpha \preceq \beta$. Note that this does not depend on the representative chosen 
in the given class. The relation $\preceq$ is a partial order on elements of $P(X)$.
\end{definition}

We now check that the partially ordered set $P(X)$ is invariant under a quasi-isometry. 

\begin{proposition}\label{PXisQIinvariant}
Let $X, Y$ be proper geodesic metric spaces and let $\Phi \from X \to Y$ be a $(k, K)$-quasi-isometry
sending the base point $\go_X \in X$ to the base point $\go_Y \in Y$. 
Then there is a well-defined induced map 
\[
\Phi^* \from P(X) \to P(Y)
\qquad\text{where} \qquad 
\Phi^*([\alpha]) = [ \Phi \circ \alpha]. 
\]
Furthermore, $\Phi^*$ preserves the partial order on $P(X)$ and $P(Y)$. 
\end{proposition}

\begin{proof}
It suffices to argue that the relation $\preceq$ is preserved by $\Phi$. That is, for quasi-geodesic rays 
$\alpha$ and $\beta$ in $X$ where $\alpha$ can be quasi-redirected to $\beta$, we need to show that 
$\Phi \circ \alpha$ can be quasi-redirected to $\Phi \circ \beta$. Consider a pair of constants $\qq$ and 
family of $\qq$--rays $\gamma_r$ ($r>0$) that respectively $\qq$--redirect $\alpha$ to $\beta$ at radius 
$r$. Since $\alpha|_r = \gamma_r|_r$ we have 
\[
(\Phi \circ \alpha)|_{r'} = (\Phi \circ \gamma_r)|_{r'}\qquad\text{for}\qquad r' \geq \frac rk - K,
\]
and since $\alpha \eventual \gamma_r$ we have \[\Phi \circ \beta \eventual \Phi \circ \gamma_r.\] 
Also, $r' \to \infty$ and $r \to \infty$ and, since $\gamma_r$ are uniform quasi-geodesics, 
$\Phi \circ \gamma_r$ are uniform quasi-geodesics as well (note that we need to use \lemref{Lem:Tame} 
to tame $\Phi \circ \alpha$, $\Phi \circ \beta$ and $\Phi \circ \gamma_r$, but if portions of these 
quasi-geodesics coincide before taming, they will also coincide after taming). This finishes the proof. 
\end{proof}

It is desirable to have a geodesic representative in each quasi-redirecting class, however, this is not 
the case in general (see Example~\ref{Ex:No-Geodesic}). But a weaker version holds which will be useful later. 

\begin{lemma}\label{limit-geo}
Let $X$ be a proper, geodesic, metric space and let $\alpha$ be a $\qq$--ray. Then there exists a geodesic ray $\alpha_0 \preceq \alpha$.
\end{lemma}

\begin{proof}
Choose a sequence $r_i \to \infty$ and let $x_i$ be starting point of the 
quasi-geodesic $\alpha|_{\geq r_i}$. Then $x_i$ is also the closest point in $\alpha|_{\geq r_i}$
to $\go$. Let
\[
\alpha_i = [\go, x_i] \cup \alpha|_{\geq r_i}. 
\]
By Part (I) of \lemref{surgery}, $\alpha_i$ is a $(3q, Q)$--quasi-geodesic ray. Up to taking a subsequence, 
the geodesic segments $[\go, x_i]$ converge to a geodesic ray $\alpha_0$. 
That is, for every $r>0$, assume $i$ is large enough we have $[\go, x_i]$ stays within distance
$1$ of $\alpha_0$ up to radius $r$. By Lemma~\ref{onestep} (setting $C=1$), 
$\alpha_0$ can be  $(3q, Q+1)$--redirected to $\alpha_i$ at radius $r$. 
But the tail of $\alpha_i$ is the same as the tail of $\alpha$. Thus $\alpha_0$ can be 
$(3q, Q+1)$--redirected to $\alpha$ at radius $r$ for every $r>0$. 
Thus, $\alpha_0 \preceq \alpha$ with quasi-redirecting constants $(3q, Q+1)$.
\end{proof}

\subsection*{Fundamental assumptions on redirecting} \label{Sec:Assumptions} 
To continue, we need to make more assumptions about the metric space $X$. General metric spaces
can be very wild with large holes in the middle. Later in the paper, we will show that for a 
large classes of groups, the Cayley graphs satisfy these assumptions. It would be interesting to know
if this holds for all finitely generated group and whether these assumptions follow from a simpler,
more geometric assumption on the metric space. 

\subsection*{Assumption 1} (Quasi-geodesic representative)\label{Q1}
There is $\qq_0$ (by making it larger, we can assume it is the same at $\qq_0$ in 
Assumption~0) such that every equivalence class $\bfa \in P(X)$ contains a 
$\qq_0$--ray. We fix such a $\qq_0$--ray, denote it by $\alpha_0 \in \bfa$ and refer to it 
as the central element of the class $\bfa$.

\subsection*{Assumption 2} (Uniform redirecting function)\label{Q2}
For every $\bfa \in P(X)$, there is a function 
\[
f_\bfa \from [1, \infty) \times [0, \infty) \to [1, \infty) \times [0, \infty),
\] 
called the redirecting function of
the class $\bfa$, such that if $\bfb \prec \bfa$ then any $\qq$--ray $\beta \in \bfb$ can be
$f_\bfa(\qq)$--redirected to $\alpha_0$. 

Note that the function $f_\bfa$ may depend on the choice of the central element. But such 
functions exist for every quasi-geodesic ray. That is:

\begin{lemma} \label{Lem:f-alpha}
Let $X$ be a space where Assumptions 0, 1 and 2 hold. 
For every quasi-geodesic ray $\alpha$, there is a function
\[
f_\alpha \from [1, \infty) \times [0, \infty) \to [1, \infty) \times [0, \infty),
\] 
such that if $\beta \prec \alpha$ for a $\qq$--ray $\beta$, the $\beta$ can be 
$f_\alpha(\qq)$--redirected to $\alpha$. 
\end{lemma} 

\begin{proof} 
Let $\alpha_0$ be the central element in the class $[\alpha]$. Assume $\alpha_0$ can be $(q_1, Q_1)$--redirected 
to $\alpha$. By Assumption~2, $\beta$ can be $f_\bfa(\qq)$--redirected to $\alpha_0$. Now, \lemref{L:Transitive} 
implies that for 
\[
f_\alpha(\qq) = \max \big( f_\bfa(\qq), (q_1+1, Q_1) \big) 
\]
the $\qq$--ray $\beta$ can be $f_\alpha(\qq)$--redirected to $\alpha$. 
\end{proof} 

We use the following example to show how Assumptions 1 may fail. To see how assumptions 2 
may fail, see Example \ref{example2}. 

\begin{example} \label{Ex:No-Geodesic}
The easiest way to generate examples is via connected, locally finite metric graphs
since these are always proper geodesic metric spaces. For a simple construction of such an example, fix an integer $k>0$. 
Attach two copies of $\RR_+$ at a point $\go$ and denote them by $\alpha_0$ and $\beta$. 
Then attach the point in $\alpha_0$ that is distance $n$ from $\go$ to 
the point in $\beta$ that is distance $n^2$ along $\beta$ with a segment of length 
$n^2/k$. Denote the resulting metric space by $X_k$. 

Then $\alpha_0$ is a geodesic in $X_k$ but $\beta$ is only a $(k,0)$--quasi-geodesic
since the paths that go along $\alpha_0$ a distance $n$ and then switch to $\beta$ give
shortcuts for points in $\beta$. That is, every point in $\beta$ lies on a quasi-geodesic ray
but not on a geodesic rays. 

Furthermore, $\alpha_0 \preceq \beta$ but $\beta \not \preceq \alpha_0$.  Hence
$P(X_k) = \{ [\alpha_0] , [\beta]\}$ and the relation is not symmetric. We also notice
that $[\beta_0]$ does not contain a geodesic ray even though $X_k$ is a geodesic metric space. 

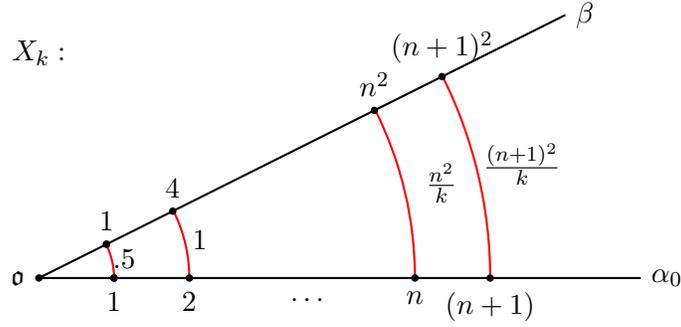
\begin{figure}[ht]
\begin{tikzpicture}
[point/.style={circle,fill,inner sep=1.3pt} ] \node (N) at (1.5,0){}; \node (N') at (1.3,0.8){}; \node (O) at (0,0){}; 

 \tikzstyle{vertex} =[circle,draw,fill=black,thick, inner sep=0pt,minimum size=2pt] 

 \node[vertex]  at (0, 0) [label=-180:$\go$] {}; 

\node at (0, 3){$X_k:$};

\node[vertex] (A) at (1,0) {}; 
\node[vertex] (A') at (.9,.45) {};

\pic [draw=red, thick, angle radius=10mm, "$\ .5$", angle eccentricity=1.1]{angle=A--O--A'}; 
 \pic [draw=red, thick, angle radius=20mm, "$1$", angle eccentricity=1.1]{angle=A--O--A'}; 
 \pic [draw=red, thick, angle radius=50mm, "$\frac{n^2}{k}$", angle eccentricity=1.1]{angle=A--O--A'}; 
 \pic [draw=red, thick, angle radius=60mm, "$\frac {(n+1)^2}{k}$", angle eccentricity=1.1]{angle=A--O--A'}; 

\draw[thick] (0,0) -- (8,0) node[right]{$\alpha_0$}; 
\draw[thick] (0,0) -- (7,3.5) node[right]{$\beta$};  
 
\node[vertex] at (1,0) [label=-90:$1$] {}; 
\node[vertex]  at (.9,.45) [label=90:$1$] {}; 
\node[vertex]  at (2,0) [label=-90:$2$] {}; 
\node[vertex]  at (1.78,.89) [label=90:$4$] {}; 
\node[vertex] at (5,0) [label=-90:$n$] {}; 
\node[vertex]  at (4.46,2.23) [label=90:$n^2$] {}; 
\node[vertex] at (6,0) [label=-90:$(n+1)$] {}; 
\node[vertex]  at (5.36,2.68) [label=90:$(n+1)^2$] {}; 
 
\node at (3.6, -0.3){$\cdots$};

\end{tikzpicture}
\caption{The space $X_k$ is a proper geodesic metric space. However, $[\beta]$ does not 
have a geodesic representative.}
\end{figure}

We can also use $X_k$ to see how Assumption 1 can fail. Namely, consider $\RR^2$ equipped with 
the Euclidean metric. For every $k>0$ attach a copy of $X_k$ to $\RR^2$ along $\alpha_0$
in a way that the resulting space is still proper. For example, attach $\alpha_0$ in $X_k$ to 
the line starting at $(1,k)$ with slope $k$. Then the geodesic $\beta_k$ connecting $(0,0)$
to $(1,k)$ and then traveling along $\beta$ in $X_k$ is only a $(k,0)$--quasi-geodesic.
In fact $[\beta_k]$ does not have a $(q, Q)$--representative for $q < k$. Hence, 
the Assumption 1 does not hold for any $\qq_0$. 
\end{example}

In \propref{Prop:Landing}, we establish some consequences of Assumption 1 and Assumption 2.  
First we need a lemma that follows only from Assumption 0. 

\begin{lemma}[Down sets in $P(X)$ are closed under point-wise convergence] \label{Lem:limit}
Let $X$ be a metric space where Assumption 0 holds, let $\bfa \in P(X)$ and let $\alpha_0$
be the central element of $\bfa$. 
Let $\alpha_n \in \bfa$ be a sequence of $\qq$--rays such that  $\alpha_n \to \beta$ point wise. 
That is, for all $t>0$, $\alpha_n(t) \to \beta(t)$. 
Then $\beta$ is a $\qq$--ray and $\beta$ can be $(q, Q+1)$--redirected to $\alpha_0$. 
\end{lemma} 
\begin{proof}
For $r>0$, let $t_r$ be the first time where $\Norm{\alpha(t_r)}= r$. 
Pick $n>0$ large enough such that, for all $t \leq t_r$, we have 
\[ 
d(\alpha_n(t), \beta(t)) \leq 1 
\]
Let $\gamma \from [0,1] \to X$ be a geodesic segment connecting $\beta(t_r)$ and 
$\alpha_n(t_r)$. Define $\beta_r$ as:
\[
\beta_r(t) = \begin{cases}
\beta(t) &\text{for } t \in[0, t_r],\\
\gamma(t-t_r) &\text{for } t \in[t_r, t_r+1],\\
\alpha_n(t-1) &\text{for } t \geq t_r+1
\end{cases}
\]
By Lemma~\ref{onestep}, $\beta_r$ is a $(q, Q+ 1)$-quasi-geodesic ray that redirects $\beta$ to $\alpha_n \sim \alpha_0$. Thus by transitivity there exists $\qq'$ that redirects $\beta$
to $\alpha_0$ at each radius $r$. Thus $\beta \preceq \alpha_0$.
\end{proof}

\begin{proposition} \label{Prop:Landing}
Let $X$ be a metric space where Assumption 0, 1 and 2 hold. 
For every $\bfa \in P(X)$, and $\qq, \qq' \in [0, \infty) \to [1, \infty) \times [0, \infty)$ and $r>0$, there are 
constants $\ell_\bfa(\qq, r) > 0$ and $R_\bfa( f_\bfa(\qq) + (0,1), \qq', r)>0$ such that 
the followings hold: 
\begin{enumerate}[(I)] 
\item (Uniform landing function for each class)  
If $\bfb \prec \bfa$, then every $\qq$--ray $\beta \in \bfb$ can be  $ (f_\bfa(\qq) + (0,1))$--redirected 
to $\alpha_0$ with the landing time at most $\ell_\bfa(\qq, r)$ (see \defref{Def:Redirection}). 

\item (Redirecting at large distance implies uniform redirecting at small distant)
If $\beta$ is a $\qq$--ray that $\qq'$--redirects to $\alpha_0$ at radius $R \geq R_\bfa( f_\bfa(\qq) + (0,1), \qq', r)$, then 
$\beta$ can be $(f_\bfa(\qq) + (0,1))$--redirected to $\alpha_0$ at radius $r$.

\item (Taming of the tail) 
If  $\beta \in \bfb \prec \bfa$ is a $\qq'$--ray where $\beta|_R$ is a $\qq$--quasi-geodesic 
segment for $R \geq R_\bfa( f_\bfa(\qq) + (0,1), \qq', r)$ then there is a $( f_\bfa(\qq) + (0,1))$--ray $\alpha \in \bfa$ 
such that  $\beta|_r=\alpha|_r$. 
\end{enumerate} 
\end{proposition} 

\begin{proof} 
We start with the proof of the first assertion. 
Assume, for contradiction, that there exists $r>0$ and a sequence of $\qq$--rays $\alpha_n \prec \alpha_0$ 
and times $t_n \to \infty$ such that there does not exists a $ f_\bfa(\qq) + (0,1)$--quasi-redirection 
of $\alpha_n$ to $\alpha_0$ at radius $r$ that lands on $\alpha_0$ before $\alpha_0(t_n)$. 
Since $X$ is proper, up to taking a subsequence, we can assume that the sequence of rays 
$\alpha_n$ converges to some ray $\beta$. By Lemma~\ref{Lem:limit}, $\beta$ is a $\qq$--ray with 
$\beta \prec \alpha_0$. By Assumption 2, there exists a $f_\bfa(\qq)$--ray $\alpha$ that redirects 
$\beta$ to $\alpha_0$ at radius $r$; let $t_\alpha$ be the landing time. 
Take $n$ large enough such that $\alpha_n$ is within distance $1$ of $\beta|_r = \alpha|_r$
and that $t_n > t_\alpha$. Let $t_r$ be the time when $\alpha_r=\alpha(t_r)$ 
and apply Lemma~\ref{onestep} to $\alpha_n$ and $\alpha$ at the time $t_r$ to construct 
a $(f_\bfa(\qq) + (0,1))$--rays $\alpha'_n$ that quasi-redirects $\alpha_n$ to 
$\alpha_0$. Then $\alpha_n'$ lands on $\alpha_0$ before $t_n$. 
This contradicts our assumptions and hence proves the first claim.

To see the second assertion we assume for contradiction that there exists $r>0$, 
$\qq' \in [1, \infty) \times [0, \infty)$, a sequence of radii $R_n \to \infty$ and a sequence of $\qq$--rays 
$\beta_n$ such that 
\begin{itemize}
\item[(S1)] $\beta_n$ can be $\qq'$--redirected to $\alpha_0$ at a radius $R_n$, but 
\item[(S2)] $\beta_n$ cannot be $(f_\bfa(\qq) + (0,1))$--redirected to $\alpha_0$ at radius $r$. 
\end{itemize} 
After taking a subsequence, we can assume that there exists a $\gamma$ where $\beta_n \to \gamma$ 
point-wise. By Lemma~\ref{Lem:limit}, $\gamma$ is a $\qq$--ray and it can be 
$(q', Q'+ 1)$--redirected to $\alpha_0$. Therefore, $\gamma$ can be 
in fact be $f_\bfa(\qq)$ redirected to $\alpha_0$ by Assumption 2. Let $n_r$ be large enough so that 
$\beta_{n_r}$ is within distance $1$ of $\gamma$ up to radius $r$. Applying Lemma~\ref{onestep}
to $\beta_{n_r}$ and $\gamma$ at radius $r$, we produce $(f_\bfa(\qq) + (0,1))$--rays $\gamma_r$ 
redirecting $\beta_{n_r}$ to $\alpha_0$ at radius $r$. This contradicts (S2) and thus proves the second 
assertion.

The proof of the third assertion is nearly identical to above. We assume, for contradiction, that there exists 
$r>0$, $\qq'\in [1, \infty) \times [0, \infty)$, a sequence of radii $R_n \to \infty$ and a sequence of 
$\qq'$--rays $\beta_n$ such that 
\begin{itemize}
\item[(S3)] $\beta_n|_{R_n}$ is a $\qq$--quasi-geodesic segment, but
\item[(S4)] $\beta_n$ cannot be $ f_\bfa(\qq) + (0,1)$--redirected to $\alpha_0$ after radius $r$. 
\end{itemize} 
After taking a subsequence, we can assume that there exists a $\gamma$ where $\beta_n \to \gamma$ point-wise. 
As before, $\gamma$ is a $\qq$--ray that can be $(q', Q'+ 1)$--redirected to $\alpha_0$. Therefore, $\gamma$ can
in fact be $f_\bfa(\qq)$ redirected to $\alpha_0$ by Assumption 2. We can argue identical to above to 
get a contradiction to (S4). This proves the third assertion. 
\end{proof} 

\section{Mono-directional spaces} \label{Sec:Single}
The boundary we are defining is meant to generalize the Gromov boundary of a hyperbolic space 
and it captures the hyperbolicity in a metric space. Hence, when the space $X$ has no hyperbolic directions, 
the space of direction $P(X)$ has only one point. In this section we concentrate on spaces without hyperbolic 
directions in this sense.

\begin{definition} 
Let $X$ be a metric space satisfying the Assumption 0, 1 and 2. We say $X$ is \emph{mono-directional} if $P(X)$ 
has only one point. That is, for every pair of quasi-geodesic rays $\alpha$ and $\beta$, we have 
$\alpha \prec \beta$. 
\end{definition} 

The first classes of mono-directional spaces we consider are products of unbounded metric spaces.

\begin{proposition}\label{product}
Let $ X= A \times B$, where $A$ and $B$ are proper geodesic metric spaces satisfying the 
Assumption 0, equipped with the $L^\infty$--metric. Then $P(X)$ is a point. 
\end{proposition}

Note that since $P(X)$ is invariant under quasi-isometries, the proposition also holds if 
we equip $X$ with the $L^p$--metric, $p \geq 1$. 

\begin{proof}
Consider a pair of $\qq$--rays $\zeta$ and $\xi$. For every $r>0$ let $R = 4r$. Let 
$(a_1, b_1)=\zeta(t_r)$ be coordinates of the first time $\zeta$ hits the sphere of radius $r$ in 
$X$ and let $(a_2, b_2)=\xi(t_R)$ be the last time $\xi$ hit the sphere of radius $R$ in $X$. 
Either $\Norm{a_1} = r$ or $\Norm{b_1} = r$. We assume without loss of generality that 
$\Norm{a_1} = r$. Similarly, either $\Norm{a_2} = R$ or $\Norm{b_2} = R$. We assume 
$\Norm{a_2} = R$ which is the more complicated case. 

Consider a $\qq_0$--geodesic ray $\alpha_1$ in $A$ passing through $a_1$ and let $a_1'$ be a 
points along $\alpha_2$ with $\Norm{a_1'} = 2r$. Similarly, consider a $\qq_0$--geodesic ray $\alpha_2$ 
in $A$ passing through $a_2$ and let $a_2'$ be a points along $\alpha_2$ with $\Norm{a_2'} = 3r$. 
Note that $a_2'$ lies between $\go$ and $\alpha_2$ could just be a quasi-geodesic segment. Finally, 
let $\alpha$ be a $\qq_0$-segment connecting $a_1'$ to $a_2'$. 
Let $b$ be any point in $B$ with $\Norm{b_2}=2r$. Let $\beta_1$ be a $\qq_0$--geodesic segment 
connecting $b_1$ to $b$ and $\beta_2$ be a $\qq_0$--geodesic segment
connecting $b$ to $b_2$. Let
\begin{align*}
\alpha_1 &\from [t_{a_1}, t_{a_1} + s_{a_1}] \to A
&\text{be the parametrization of the segment }\qquad &[a_1, a_1']_{\alpha_1}, \\
\alpha &\from [t_{a}, t_{a}+ s_{a}] \to A
&\text{be the parametrization of the segment }\qquad &[a_1', a_2']_{\alpha},  \\
\alpha_2 &\from [t_{a_2}, t_{a_2}+ s_{a_2}] \to A
&\text{be the parametrization of the segment }\qquad &[a_2', a_2]_{\alpha_2},  \\
\beta_1 &\from [t_{b_1}, t_{b_1}+ s_{b_1}] \to B
&\text{be the parametrization of the segment }\qquad &[b_1, b]_{\beta_1},\\
\intertext{and} 
\beta_2 &\from [t_{b_2}, t_{b+2}+ s_{b_2}] \to B
&\text{be the parametrization of the segment }\qquad &[b, b_2]_{\beta_2}.
\end{align*} 
Let 
\begin{align*}
t_1&= t_r & t_2&= t_1 + s_{a_1} & t_3&= t_2 + s_{b_1}\\
t_4&= t_3 +  s_{a} &t_5& = t_4 + s_{b_2} & t_6&= t_5 +s_{a_2}
\end{align*}
and define, 
\[
\gamma(t) = \begin{cases}
\zeta(t)                                                         &\text{for } t \in[0, t_1],\\
\big(\alpha_1(t-t_1 + t_{a_1}), b_1\big)       &\text{for } t \in [t_1, t_2],\\
\big(\alpha_1', \beta_1(t-t_2+t_{b_1}) \big) &\text{for } t \in [t_2, t_3],\\
\big(\alpha(t-t_3+t_a), b\big)                        &\text{for } t \in [t_3, t_4],\\
\big(a_2', \beta_2(t-t_4+t_{b_2}) \big)        &\text{for } t \in [t_4, t_5],\\
\big(\alpha_2(t-t_5+t_{a_2}), b_2 \big)        &\text{for } t \in [t_5, t_6],\\
\xi(t) &\text{for } t \geq t_6.
\end{cases}
\]
Then $\gamma$ is a quasi-geodesic where the constant depend only on $\qq$ and $\qq_0$. 
In fact, setting $t_0=0$ and $t_7=\infty$, we have the restriction of $\gamma$ to each interval 
$[t_{i-1}, t_{i+1}]$ ($i=1, \dots, 6$) is a uniform
quasi-geodesic by Part (II) of \lemref{surgery} since  $\gamma(t_i)$ is a closest point from any point in 
$\gamma|_{[ t_i, t_{i+1}]}$ to the segment $\gamma|_{[ t_{i-1}, t_i]}$.
And, for $|i-j| \geq 2$, any point in the segment 
$\gamma|_{[t_i, t_{i+1}]}$ and any point in the segment $\gamma|_{[t_j, t_{j+1}]}$
are at least $r$ apart. Since the length of all these intervals are comparable to $r$ and 
each one is a quasi-geodesic, $\gamma$ is a also quasi-geodesic where the constants do not depend 
on $r$. Therefore, $\zeta \prec \xi$. The proof in the case $\Norm{b_2} = R$ is similar. 
\end{proof}

\begin{figure}

\begin{tikzpicture}[scale=0.8]
 \tikzstyle{node_style}=[inner sep=1.5pt, circle, fill=black]
    \draw[thick] (0,0) -- (8,0) node[left] {};
   \draw[thick] (0,0) -- (0,8) node[left] {};
   \draw[thick] (0,8) -- (8,8) node[left] {};
    \draw[thick] (8,0) -- (8,8) node[right] {};
     \draw[thick] (2,3) -- (6,3);
     \draw[thick] (2,3) -- (2,6);
      \draw[thick] (6,3) -- (6,6);
       \draw[thick] (2,6) -- (3,6) node[above] {${\scriptstyle(a_1,b_1)}$}-- (6,6);
    \draw[thick] (4,6) -- (4,7);
      \draw[thick] (4,7) node[above right] {${\scriptstyle(a_1,b'_1)}$}-- (7,7) node[above] {${\scriptstyle(a'_1,b'_1)}$};
      \draw[thick] (7,7) -- (7,1) node[below] {${\scriptstyle (a'_1,b'_2)}$};
      \draw[thick] (7,1) -- (4,1) node[above] {${\scriptstyle (a'_2,b'_2)}$};
      \draw[thick] (4,1) -- (4,0) node[below left] {${\scriptstyle (a_2,a_2)}$};
      \node[node_style, label=below:{}] at (4,4.5) (y) {};
      \draw [thick, rounded corners] (y.center) .. controls (4.5,5)   .. (5,4.9) .. controls (5.8,5.2) .. (5.5,5.7) .. controls (5.3,5.7) .. (4.1,5.7) .. controls (4.15,5.7) .. (4,6) ;

  \draw [thick, rounded corners] (4,0) .. controls (3.95,-0.5) .. (4.4,-0.7) .. controls (6.5,-0.9) .. (6.3,-2.5) .. controls (6.7,-2.8) .. (6.9,-3) node[above] {};
     \draw[decorate,decoration={brace,amplitude=8pt,mirror,aspect=0.25},xshift=-2pt] (0.1,4.4)--(1.1,4.4) node [below=8pt] {$4r$}--(4,4.4) ;  
        \draw[decorate,decoration={brace,amplitude=8pt,mirror},xshift=-2pt] (4.1,4.4) -- (6,4.4) node [midway,below=8pt] {$r$};  
\end{tikzpicture}

\end{figure}
\begin{example}The Baumslag-Solitar group.
For a complete calculation of the redirecting boundary of the Baumslag–Solitar group, see \cite{Asha}. 
\end{example}

Morse and sublinearly Morse quasi-geodesic rays 
in a metric space $X$ resemble geodesics in a Gromov hyperbolic space. 
In fact, we will later show in \propref{SubMorsemin} that every 
equivalence class of $\kappa$--Morse quasi-geodesic rays, and hence every equivalence class of 
Morse geodesic rays (where the equivalence relations are those specified in the construction of sublinearly 
Morse boundary and Morse boundaries) are also equivalence classes in $P(X)$. 

\begin{question}
Assume $X$ does not have any Morse geodesics. Does that imply that $P(X)$ is a single point?
\end{question}

\section{The Topology and the boundary} \label{Sec:boundary}

In this section, we build a topology on the set $X \cup P(X)$. We denote $P(X)$ equipped
with the restriction of this topology to $P(X)$ by $\partial X$. 
This boundary is strongly analogous to the $\kappa$--Morse boundary of $X$ (see \cite{QRT2})
and it should be considered as an enlargement of the $\kappa$-Morse boundary.
In fact, we will show in the \secref{Sec:kappa} that $\partial X$ contains every $\kappa$--Morse boundary 
as a topological subspace.

We define the topology on $X \cup P(X)$ by defining a system of neighbourhoods. Recall that points
in $P(X)$ are equivalence classes of quasi-geodesic rays. To unify the treatment of point in 
$X$ and $P(X)$, for every $x \in X$, we consider the set of quasi-geodesic rays that pass 
through $x$. Abusing the notation, we denote this set by $ \mathbb{x}$, that is 
\[
 \mathbb{x} = \Big\{ \text{quasi-geodesics rays passing through $x$} \Big\}. 
\]
We use the gothic letters $\Ga, \Gb, \Gc$ to denote elements of $P(X) \cup X$, that is, either a set of quasi-geodesic rays 
passing through a point $x \in X$ or an equivalence class of quasi-geodesic rays in $P(X)$. For $\bfa \in P(X)$, define
$F_\bfa \from [1,\infty) \times [0, \infty) \to [1,\infty) \times [0, \infty)$ by 
\begin{equation}\label{actualtopology}
F_\bfa(\qq) = \max \{ f_\bfa(\qq) + (1,1), (1, q+Q)\} \qquad\text{for}\qquad \qq \in   [1,\infty) \times [0, \infty).
\end{equation}
\begin{definition} \label{Def:The-In-Topology}
For $\bfa \in P(X)$ (with $\alpha_0 \in \bfa$ as a central element) and $r>0$, define   
\begin{align*}
 \calU(\bfa, r) := \Big  \{  \Gb \in P(X) \cup X \, \ST 
 &\text{ every $\qq$--ray in $\Gb$ can be $F_\bfa(\qq)$--redirected to $\alpha_0$ at radius $r$} \Big \}.   
\end{align*}
\end{definition}

\begin{remark} 
The sets $ \calU(\bfa, r)$ will define open neighborhoods around $\bfa$ in our topology. This is meant to be 
a direct analogue of the the cone topology for the Gromov boundary of hyperbolic spaces.
Roughly speaking, we think a quasi-geodesic ray $\beta$ is in a small neighborhood of 
a quasi-geodesic ray $\alpha$ if $\beta$ can be redirected to $\alpha$ at a large radius. However, one might think 
it is more natural to require that $\alpha$ should be redirected to $\beta$ at a large radius. This would suggest 
that there may exists an Out-topology which is different from our topology but equally valid. 
In \secref{Sec:Out-Topology}, we will argue that this is not true and the Out-topology is the wrong 
definition. The main problem is that \lemref{Lem:F} below will not hold for the Out-topology. 
\end{remark} 

In most arguments about a class of quasi-geodesic rays $\bfa$, it is enough to consider $\qq$--rays 
where $\qq$ is not too big. We now make this precise.

\begin{lemma} \label{Lem:Max}
For every $r>0$ there is a pair of constants $\qq_{\rm max}(r) \in [1,\infty) \times [0,\infty)$
such that if $\qq \not \leq \qq_{\rm max}$ then, for every $\bfa \in P(X)$, any $\qq$--ray $\beta$ can be 
$F_\bfa(\qq)$--redirected to $\alpha_0$ at radius $r$. 
\end{lemma}

\begin{proof} 
For $\qq=(q, Q)$ let $(q_1, Q_1)= F_\bfa(q,Q) $. From the definition of $F_\bfa$ we know that
$Q_1 \geq \max(q, Q)$ and $q_1 \geq 2$. That is, if $q>2r$ or $Q \geq 2r$, then $Q_1 \geq 2r$. Now, for a $\qq$--ray 
$\beta$, and $\bfa \in P(X)$, let $\gamma$ be the concatenation $\beta|_r$ followed by $\beta|_r$ traverses in reverse 
and then $\alpha_0$. Then $\gamma$ is a $(2, 2r)$--ray, $q_1 \geq 2$ and $Q_1 \geq 2r$. This finishes the proof. 
\end{proof} 

We verify some basic properties of $\calU(\bfa, r)$. Below $B(x,1)$ is the ball of radius $1$ around $x$. 

\begin{lemma} \label{Lemma:Ubeta}
Assume $X$ satisfies Assumptions 0, 1 and 2. Then: 
\begin{enumerate}[(I)]
\item For $\bfa, \bfb \in P(X)$, if $\bfb \preceq \bfa$ then $\bfb \in \calU(\bfa, r)$ for all $r>0$. In particular,
$\bfa \in \calU(\bfa, r)$ for all $r > 0$. 
\item For every $r_{2} \geq r_{1}>0$, we have $\calU(\bfa, r_{2}) \subseteq \calU(\bfa, r_{1})$.
\item For every $r > 0$, there is  $r_\bfa>0$ depending on $\bfa$ and $r$ such that 
\begin{enumerate}
\item For every $\bfb \in \calU(\bfa, r_\bfa)$,
there is $r_\bfb>0$ such that 
\[
 \calU(\bfb, r_\bfb \,) \subset \calU(\bfa, r).
\]
\item For every $ \mathbb{x} \in \calU(\bfa, r_\bfa) \cap X$, 
\[
 B( \mathbb{x}, 1) \subset \calU(\bfa, r).
\]
\end{enumerate}
\end{enumerate}
\end{lemma}

\begin{proof}
Parts (I) and (II) follow immediately from the definition of the neighborhood, Assumption 2
and  the fact that $F_\bfa(\param) \geq f_\bfa(\param)$.  We check (a) of (III). 

Fix $r>0$ and let $\qq_{\rm max}(r)$ be as in \lemref{Lem:Max}. Assume for contradiction that there 
exist $\qq \leq \qq_{\rm max}(r)$, a sequence $r^n_\bfa \to \infty$ and  $\bfb^n \in \calU(\bfa, r_\bfa^n)$ such that for 
$r^n_\bfb = 2r_\bfa^n$, there exist $\bfc^n \in \calU(\bfb^n, r_\bfb^n)$  and $\qq$--rays 
$\gamma^n \in \bfc^n$ such that $\gamma^n$  cannot be $F_\bfa(\qq)$--redirected to $\alpha_0$. 
After taking a subsequence, we can assume the central elements $\beta_0^n \in \bfb^n$ and $\gamma^n \in \bfc^n$ 
point-wise converge. That is, for $t>0$, 
\[ 
 \lim_{n \to \infty} \beta_0^n(t)  = \beta_0(t)  \qquad\text{and}\qquad  \lim_{n \to \infty}  \gamma^n(t) = \gamma(t), \]
 where $\beta_0$ is a $\qq_0$--ray and $\gamma$ is a $\qq$--ray. 

By applying Lemma~\ref{onestep} we see that $\beta_0$ can be redirected to $\alpha_0$ for all $R>0$ hence 
$[\beta_0]  \preceq \bfa$. Also $\gamma$ can be redirected to $\beta_0$ for all $R>0$ hence $\gamma \prec \beta_0$. 
Therefore, $[\gamma] \prec \bfa$ and, by Assumption 2, $\gamma$ can be $f_\bfa(\qq)$ redirected to $\alpha_0$ 
for all $R>0$. Applying Lemma~\ref{onestep} again, we see that $\gamma^n$ can be $f_\bfa(\qq) + (0,1)$
redirected to $\alpha_0$ at some radius that goes to infinity. But $f_\bfa(\qq) + (0,1) \leq F_\bfa(\qq)$ which 
contradicts the choice of $\gamma^n$. This proves part (a).

To see part (b), define
\[
r_\bfa = \max_{\qq \leq \qq_{\rm max}(r)} R_\bfa(\qq, F_\bfa(((q, Q+3)), r)+1.
\]
Let $\mathbb{z}$ be a point in $B( \mathbb{x}, 1)$ and let $\zeta$ be any $\qq$--ray in $ \mathbb{z}$ ( i.e. $\zeta$ passes through the point $z \in X$). By 
Lemma~\ref{twostep}, there is a $(q, Q+3)$--ray $\zeta'$ that passes through $x$ and 
$\zeta'|_{r_\bfa-1} = \zeta|_{r_\bfa-1}$. Since $ \mathbb{x} \in \calU(\bfa, r_\bfa)$, $\zeta'$ can be redirected to 
$\alpha_0$ by a $F_\bfa(((q, Q+3))$--ray. By the choice of $r_\bfa$ and Proposition~\ref{Prop:Landing}, 
$\zeta$ can be $F_\bfa(\qq)$--redirected to $\alpha_0$ and thus $z \in \calU(\bfa, r)$. 
\end{proof}

\subsection*{A system of neighbourhoods}
For each $\bfa \in P(X)$, define 
\[
\calB(\bfa) = \Big\{ \calV \subset X \cup P(X) \ST 
\calU(\bfa, r) \subset \calV \quad \text{ for some $r>0$ } \Big\}
\]
and for every $x \in X$, define 
\[
\calB(\mathbb{x}) = \Big\{ \calV \subset X \cup P(X) \ST 
B(\mathbb{x}, r) \subset \calV \quad \text{ for some $r>0$ } \Big\}.
\]

We claim that these sets form a fundamental system of
neighbourhoods that can be used to define a topology on the set $X \cup P(X)$. 
We need to check they satisfy some basic properties. 

\begin{proposition} \label{Prop:PreTopology}
For every $\Ga \in X \cup P(X)$, the set $\calB(\Ga)$ satisfies the following properties 
defining a pretopology on $X \cup \partial P(X)$:
\begin{enumerate}[(I)]
\item Every subset of $X \cup P(X)$ which contains a set belonging to $\calB(\Ga)$ 
itself belongs to $\calB(\Ga)$.
\item Every finite intersection of sets of $\calB(\Ga)$ belongs to $\calB(\Ga)$.
\item The element $\Ga$ is in every set of $\calB(\Ga)$.
\item If $\calV \in \calB(\Ga)$ then there is $\calW \in \calB(\Ga)$ such that, 
for every $\bfb \in \calW$, we have $\calV \in \calB(\Gb)$.
\end{enumerate}
\end{proposition}

\begin{proof}
Property (I) is the definition of $\calB(\Gb)$. 
To see (II), consider sets $\calV_1, \dots, \calV_k \in \calB(\Gb)$. First assume $\Ga = \bfa \in P(X)$ and 
let $r_i$ be such that $\calU(\bfa, r_i) \subset \calV_i$ and let $r = \max r_i$. 
Note that $\calU(\bfa, r)  \subseteq \calU(\bfb, r_i)$ by part (II) of Lemma~\ref{Lemma:Ubeta}. 
Therefore, 
\[
\calU(\bfb, r) \subset \bigcap_i \calV_i
\]
and hence the intersection is in $\calB(\Gb)$.
If  $\Ga = x \in X$, let $r_i$ be such that  $B(x, r_i) \subset \calV_i$ and let $r = \min r_i$. 
Then  
\[
B(x, r) \subset \bigcap_i \calV_i
\]
and again the intersection is in $\calB(\Gb)$.

Property (III) is trivial when $\Ga \in X$ and it is part (I) of Lemma~\ref{Lemma:Ubeta}
when $\Ga \in P(X)$. 

We now check that Property (IV). If $\Ga =  \mathbb{x}$, we 
have $B( \mathbb{x},r) \subset \calV$ and we can take $\calW = B( \mathbb{x}, r/2)$. Then 
for every $ \mathbb{y} \in \calW$, we have $B( \mathbb{y}, r/2) \subset \calV$ and $V \in \calB(y)$. 

Now assume $\Ga = \bfa \in P(X)$ and take $\calV \in \calB(\Ga)$. Let $r>0$ be 
such that $\calU(\bfa, r) \subset \calV$, let $r_\bfa$ be as in part (III) of \lemref{Lemma:Ubeta}
and let $\calW = \calU(\bfa, r_\bfa)$. For every $\bfb \in \calW$, part (III) of 
Lemma~\ref{Lemma:Ubeta} implies that 
\[
\calU(\bfb, r_\bfb) \subset \calU(\bfa, r ) \subset \calV.
\]
Thus $\calV \subset \calB(\bfb)$ and we are done. 
\end{proof}

Our system of neighborhoods in fact defines a topology on $X \cup \partial X$: 

\begin{proposition}[\cite{Bourbaki} Proposition 2]
If to each elements $\Gb \in \partial X$ there corresponds a set $\calB(\Gb)$ of
subsets of $\partial X$ such that properties (I) to (IV) above are satisfied, then there is a
unique topological structure on $\partial X$ such that for each $\Gb \in \partial X$, 
$\calB(\Gb)$ is the set of neighborhoods of $\Gb$ in this topology. 
\end{proposition}

We can argue, similar to above, to see that $\calU(\bfa, r) \cap \partial X$ defines a system of 
neighborhoods for $\partial X$ and the resulting topology is the subspace topology 
induces by the topology defined on $X \cup \partial X$. That is, $X$, $\partial X$ and
$X \cup \partial X$ are all topological spaces. 

\subsection{Invariance under quasi-isometries}
\defref{Def:The-In-Topology} is written in such a way that neighborhoods $\calU(\bfb, r)$ are mapped to 
neighbourhoods under quasi-isometries, except, the functions $F_\bfa(\param)$ may have to 
replaced with larger functions. We now check that this does not impact the definition of the topology. 

\begin{lemma} \label{Lem:F}
Let $F'_\bfa \from [1,\infty) \times [0,\infty)  \to \RR^2$ be a family of functions such that, 
for every $\bfa \in P(X)$,  $F_\bfa'(\qq) \geq F_\bfa(\qq)$ for all $\qq$. 
Define
\begin{align*}
 \calU(\bfa, r, F_\bfa') := \Big  \{  \Gb \in P(X) \cup X \, \ST 
 &\text{ every $\qq$--ray in $\Gb$ can be $F_\bfa'(\qq)$--redirected to $\alpha_0$ at radius $r$} \Big \}   
\end{align*}
Then, for every $r>0$, there exists $R >0$ such that 
\[  \calU(\bfa, R, F_\bfa') \subset \calU(\bfa, r). \]
\end{lemma}

\begin{proof}
For a given $r$ let $q_{\rm max}(r)$ be as in \lemref{Lem:Max}. Let 
\[
R = \max_{\qq \leq \qq_{\rm max}} R_\bfa(\qq, F_\bfa'(\qq), r).
\]  
Let $\Gb$ be a point in $\calU(\bfa, R, F_\bfa')$ and let $\beta \in \Ga$ be a $\qq$--ray with 
$\qq \leq \qq_{\rm max}$. By definition of $\calU(\bfa, R, F_\bfa')$, there is a $F_\bfa'(\qq)$--ray
$\alpha \in \bfa$ such that $\beta|_R = \alpha|_R$. Now part (III) of \propref{Prop:Landing}
implies that there is $F_\bfa(\qq)$--ray $\alpha' \in \bfa$ such that $\alpha'|_r = \alpha|_r =\beta|_r$. 
Therefore, $\Gb \in \calU(\bfa, r)$. 
\end{proof}

\begin{remark}\label{Rem:F}
Note that, since $F_\bfa' \geq F_\bfa$, we have $\calU(\bfa, r) \subseteq \calU(\bfa, r, F_\bfa')$.
Hence the above Lemma implies that the topology defined by the neighborhoods $\calU(\bfb, r, F')$ 
is the same as the topology defined by  $\calU(\bfb, r)$. That is, the definition of 
topology on $\partial X$ is robust and does not depend on the family of function $F_\bfa$. 
\end{remark} 
 
\begin{theorem} \label{Thm:QI} 
Let $X$ be a metric space satisfying Assumptions 0, 1 and 2 and let 
\[
\Psi \from X \to Y
\] 
be a quasi-isometry. Then $Y$ also satisfies Assumptiona 0, 1 and 2 and hence the quasi-redirecting boundary 
exists for both $X$ and $Y$. Furthermore, the induced map from 
\[
\Psi^*:\partial X \to \partial Y
\]
is a homeomorphism. 
\end{theorem} 

\begin{proof}
Assumptions  0, 1 and 2 are written in a way that, the fact that they hold for $X$ 
immediately implies they also hold for $Y$. We have shown in \propref{PXisQIinvariant} that the 
map $\Psi^*$ is a bijection. We only need to show that it is continuous. Let $\bfa$ be a point in 
$\partial X$ and let $F_{\Psi(\bfa)}$ be the redirecting function defining neighborhoods for  
the class $\Psi(\bfa)$ in $Y$. For $r>0$, consider a neighborhood $\calU(\Psi(\bfa), r)\cap \partial Y$ 
around $\Psi(\bfa)$ in $\partial Y$. It is immediate from the definition of neighborhoods that, 
for some $r'>0$ and some function $F' \from [1,\infty) \times [0, \infty) \to \RR$, we have 
\[
\calU(\bfa, r', F') \subset \Phi^{-1}(\calU(\Psi(\bfa), r)\cap \partial Y).
\]
Now define $F'' = \max (F_\bfa, F')$. 
Now, by \lemref{Lem:F}, there is $R>0$ such that 
\[
\calU(\bfa, R, F'') \subset \calU(\Psi(\bfa), r', F'). 
\]
Also, since $F_\bfa \leq F''$, $\calU(\bfa, R)  \subset \calU(\bfa, R, F'')$. Hence
\[
\calU(\bfa, R) \cap \partial X \subset \calU(\Psi(\bfa), r', F') \cap \partial X 
\subset \Phi^{-1}(\calU(\Psi(\bfa), r)\cap \partial Y) \cap \partial X. 
\]
This proves continuity of $\Psi^*$ at $\bfa$. The proof of continuity of $(\Psi^*)^{-1}$ is similar. 
\end{proof}

\section{The sublinearly-Morse boundary and the quasi-redirecting boundary} \label{Sec:kappa}
In this section we prove the sublinearly Morse boundary, defined in \cite{QRT1} and \cite{QRT2}, is naturally a topological 
subspace of $\partial X$. Points in the sublinearly Morse boundary are defined via sublinear fellow traveling
of quasi-geodesic rays. We first show that each sublinearly Morse equivalence class is also a redirecting 
equivalence class. We then show that the topology on the two spaces are compatible. In addition, we show that the sublinearly Morse classes are minimal elements in the partially set $P(X)$.  First we recall the construction of 
the sublinearly Morse boundary. 

\subsection{Background on $\kappa$-Morse boundaries}
Let $(X, d_X)$ be a metric space satisfying the Assumption 0.  We follow the notation of \cite{QRT2}
and refer the reader to \cite{QRT2} for more details.  By a \emph{sublinear function} we mean 
a concave continuous function $\kappa \from [0, \infty) \to [1, \infty)$ such that 
\[
\lim_{t \to \infty} \frac{\kappa(t)}{t}  = 0.
\]
For $x  \in X$, define
\[
\Norm x : = d_X(\go, x).
\] 
We often need to refer to $\kappa(\Norm x)$ and, 
for simplicity, we will write $\kappa(x)$ instead of $\kappa(\Norm x)$. 
Given a quasi-geodesic ray $\alpha$ (here we think of $\alpha$ as a subset of  $X$, that is, 
we abuse the notation and let $\alpha$ represents both the map 
$\alpha \from [0, \infty) \to X$ and the image of this map) and a constant $m > 1$, we define:
\[ \mathcal{N}_{\kappa}(\alpha, m) := \Big\{ x \in X \ : \ d_X(x, \alpha) \leq m \cdot \kappa(x) \Big\}. \]

We recall two equivalent definitions of $\kappa$--Morse sets. Their equivalence is established in 
\cite[Proposition 3.10]{QRT2}, hence we use whichever is needed as the definition.

\begin{definition}($\kappa$-Morse I)\label{Def:Weakly-Morse}
We say a quasi-geodesic ray $\alpha$ is \emph{$\kappa$-Morse I} if there exists a proper 
function $m_\alpha : \mathbb{R}^2 \to \mathbb{R}$ such that for any $(\qq, \sQ)$-quasi-geodesic 
$\gamma \from [s, t] \to X$ with endpoints on $\alpha$, we have 
\[
\gamma([s, t]) \subset \calN_\kappa \big(\alpha, m_\alpha(\qq, \sQ)\big).
\]
The function $m_\alpha$ will be called a $\emph{$\kappa$--weakly Morse gauge}$ of $\alpha$. 
\end{definition}

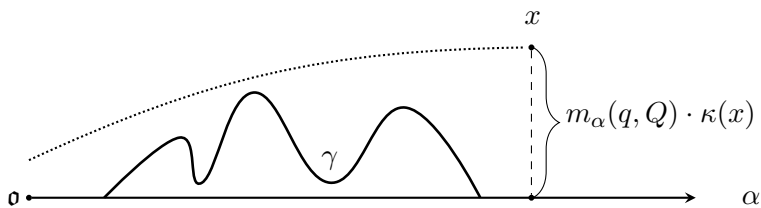
\begin{figure}[h!]
\begin{tikzpicture}
 \tikzstyle{vertex} =[circle,draw,fill=black,thick, inner sep=0pt,minimum size=.5 mm] 
[thick, 
    scale=1,
    vertex/.style={circle,draw,fill=black,thick,
                   inner sep=0pt,minimum size= .5 mm},
                  
      trans/.style={thick,->, shorten >=6pt,shorten <=6pt,>=stealth},
   ]

 \node[vertex] (a) at (0,0) {};
 \node at (-0.2,0) {$\go$};
 \node (b) at (9, 0) {};
 \node at (9.6, 0) {$\alpha$};
 \node (c) at (6.7, 2) {};
 \node[vertex] (d) at (6.68,2) {};
 \node at (6.7, 2.4){$x$};
 \node[vertex] (e) at (6.68,0) {};
 \node at (4, 0.5){$\gamma$};
 \draw [-,dashed](d)--(e);
 \draw [decorate,decoration={brace,amplitude=10pt},xshift=0pt,yshift=0pt]
  (6.7,2) -- (6.7,0)  node [black,midway,xshift=0pt,yshift=0pt] {};

 \node at (8.4, 1.1){$m_\alpha(q, Q) \cdot \kappa(x)$};
 \draw [thick, ->](a)--(b);
 \path[thick, densely dotted](0,0.5) edge [bend left=12] (c);

  \pgfsetlinewidth{1pt}
  \pgfsetplottension{.75}
  \pgfplothandlercurveto
  \pgfplotstreamstart
  \pgfplotstreampoint{\pgfpoint{1cm}{0cm}}  
  \pgfplotstreampoint{\pgfpoint{2cm}{0.8cm}}   
  \pgfplotstreampoint{\pgfpoint{2.3cm}{0.2cm}}
  \pgfplotstreampoint{\pgfpoint{3cm}{1.4cm}}
  \pgfplotstreampoint{\pgfpoint{4cm}{0.2cm}}
  \pgfplotstreampoint{\pgfpoint{5cm}{1.2cm}}
  \pgfplotstreampoint{\pgfpoint{6cm}{0cm}}
  \pgfplotstreamend
  \pgfusepath{stroke} 
  
\end{tikzpicture}
\caption{Weakly Morse: The $(\kappa, n)$--neighbourhood of the geodesic ray $b$.}
\end{figure}

\begin{definition}\label{Def:Strongly-Morse}($\kappa$-Morse II)
We say a quasi-geodesic ray $\alpha$ is \emph{$\kappa$--Morse II} if there exists a proper function 
$m_\alpha : \mathbb{R}^2 \to \mathbb{R}$ such that for any sublinear function $\kappa'$ 
and for any $r > 0$, there exists $R$ such that for any $\qq$--ray $\beta$
with $m_Z(\qq) \leq \frac{r}{2\kappa(r)}$, if 
\[
d_X(\beta_R, \alpha) \leq \kappa'(R)
\qquad\text{then}\qquad
\beta|_r \subset \calN_\kappa \big(\alpha, m_\alpha(\qq)\big).
\]
The function $m_\alpha$ will be called a $\emph{$\kappa$--strongly Morse gauge}$ of $\alpha$. 
\end{definition}

Sublinear fellow traveling defines an equivalence relation between quasi-geodesic rays. We write, 
\[
\alpha \sim_s \beta 
\qquad \Longleftrightarrow \qquad
\lim_{r \to \infty} \frac{d(\alpha_r, \beta_r)}{r} = 0. 
\]
The sublinear fellow traveling equivalence class of $\alpha$ is denoted $[\alpha]_s$. 
We also recall the following basic facts about these classes from \cite[Lemma 3.4]{QRT2}:

\begin{proposition}
Let $\alpha$ be a $\kappa$--Morse $\qq$--ray. Then, for any other $\qq$--ray $\beta$, 
if $\alpha \sim_s \beta$ then:
\begin{itemize}
\item $\beta$ is $\kappa$-Morse.
\item there exists $m(\qq)$ and $m'(\qq)$ that depends only on $[\alpha]_s$ and $\qq$ such that 
\[
\alpha \in \calN_\kappa \big(\beta, m(\qq)\big)
\] 
and 
\[
\beta \in \calN_\kappa \big(\alpha, m'(\qq)\big).
\] 
\end{itemize}
\end{proposition}

The boundary $\partial_\kappa X$, as a set, is the set of sublinear fellow traveling
classes of $\kappa$--Morse geodesics. We now show that $\partial_\kappa X \subset P(X)$. 

\begin{lemma}\label{kapparedirect}
Let $\alpha$ be a geodesic ray that is $\kappa$--Morse and let $\beta \in [\alpha]_s$ be a $\qq$--ray, 
then $\beta \preceq \alpha$ and $\alpha \preceq \beta$. That is, $[\alpha]_s \subset [\alpha]$, where $[\alpha]$ is the equivalence class of $\alpha$ under quasi-redirecting. 
\end{lemma}
\begin{proof}
Since $\beta \in [\alpha]_s$, there is a sublinear function $\kappa'$ such that, for every $R>0$, 
\[d(\beta_R, \alpha) \leq \kappa'(R).\] \defref{Def:Strongly-Morse} implies that 
\[ \beta \in \calN_\kappa(\alpha, m_\alpha(\qq)).\]
Let $r$ be large enough such that 
\[
d( \beta_r, \alpha)  \leq m_\alpha(\qq) \kappa(r) < \frac{r}2 .
\]
Surgery Lemma \ref{redirect11} implies that $\beta$ can be redirected to $\alpha$ at radius $r$ by a 
$(9q, Q)$--quasi-geodesic ray where $\qq=(q,Q)$.  Since this holds for all large values of $r$, we have that 
$\beta \preceq \alpha$.

We now show $\alpha \preceq \beta$. Let $y \in X$ be the last point on $\beta$ that is in the ball of radius $3r$.
We denote the portion of $\beta$ from $y$ onward by $\beta_{[y, \infty)}$.

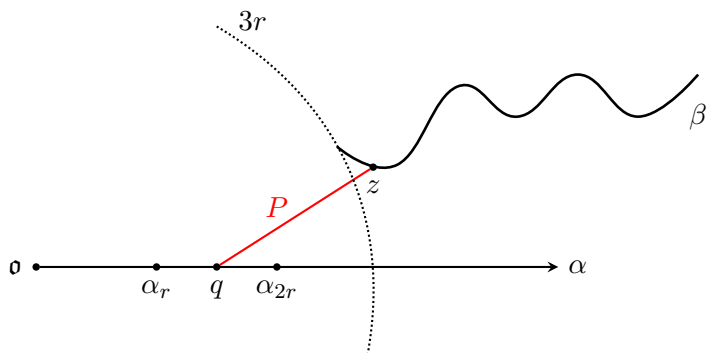
\begin{figure}[ht]
\begin{tikzpicture}[scale=0.8]
 \tikzstyle{vertex} =[circle,draw,fill=black,thick, inner sep=0pt,minimum size=2pt] 

 \node[vertex] (a) at (0,0)  [label=180:$\go$] {};
 \node  (b) at (9, 0) {$\alpha$};
 \draw [thick, ->](a)--(b);
 
 \node[vertex]  at (2, 0) [label=-90:$\alpha_r$] {}; 
 \node[vertex]  at (4, 0) [label=-90:$\alpha_{2r}$] {}; 
  
 \node[vertex] (q) at (3, 0)  [label=-90:$q$] {};
 \node[vertex] (z) at (5.6, 1.66)  [label=-90:$z$] {};
 \draw [red, thick] (q)--(z);

 \node at (4, 1){$\textcolor{red}{P}$};
%
 \path[thick, densely dotted](3,4) edge [bend left=35] (5.5,-1.5);
 \node at (3.6, 4.1){$3r$};

 \node at (11, 2.5) {$\beta$};
 
 \pgfsetlinewidth{1pt}
 \pgfsetplottension{.75}
 \pgfplothandlercurveto
 \pgfplotstreamstart
 \pgfplotstreampoint{\pgfpoint{5 cm}{2cm}}  
 \pgfplotstreampoint{\pgfpoint{6cm}{1.7cm}}   
 \pgfplotstreampoint{\pgfpoint{7 cm}{3cm}}
 \pgfplotstreampoint{\pgfpoint{8 cm}{2.5cm}}
 \pgfplotstreampoint{\pgfpoint{9 cm}{3.2cm}}
 \pgfplotstreampoint{\pgfpoint{10 cm}{2.5cm}}
 \pgfplotstreampoint{\pgfpoint{11cm}{3.2cm}}
 \pgfplotstreamend
 \pgfusepath{stroke} 
  
\end{tikzpicture}
\caption{The path $[\go, q] \cup P \cup \beta_{[z, \infty)}$ is a quasi-geodesic ray that redirects $\alpha$ to $\beta$.}
\end{figure}
Let $P$ be the geodesic segment that realizes the set distance between $\alpha|_{2r}$ and 
$\beta_{[y, \infty)}$. Such a segment exists  since $X$ is proper.  
Suppose $P$ has an end point $q$ on $\alpha|_{2r}$ and an end point $z$ on $\beta_{[y, \infty)}$.
We denote the portion of $\beta$ from $z$ onward by $\beta_{[z, \infty)}$. 
By Surgery Lemma~\ref{surgery} I, $P \cup \beta_{[z, \infty)}$ is a $(3q, Q)$--quasi-geodesic.  
Also, the closets point projection of any point in $[\go, q]$ to $P \cup \beta_{[z, \infty)}$ is $q$,
otherwise, $P$ would not be the shortest path from $\beta_{[y, \infty)}$ to $\alpha|_{2r}$. 
Thus, again by Surgery Lemma~\ref{surgery} I, the concatenation 
\[
\ell : = [\go, q] \cup P \cup \beta_{[z, \infty)}
\]
 is a $(9q, Q)$--ray. 
 
It remains to prove that, for $r$ large enough, we have $\Norm{q}\geq r$. First notice that 
\[
\Norm{z} \leq d(z, q) + d(q, \go) \leq d(y, q) + 2r \leq d(y, \go) + d(\go, q) + 2r \leq 7r. 
\]
 Choose $r$ large enough such that 
 \[
d( z , \alpha)  \leq m_\alpha(\qq) \kappa(7r) < \frac{r}2.
\]
Let $z_{\alpha}$ be a closest point in $\alpha$ from $z$. Then $\Norm{z_{\alpha}} \geq 3r - r/2 > 2r$,
that is, $\alpha_{2r}$ lies in between $q$ and $z_\alpha$. We have 
\[
d(q, z_{\alpha}) - d(z_\alpha, z)  \leq d(q, z) \leq 
d(\alpha_{2r},z) \leq d(\alpha_{2r}, z_{\alpha}) + d(z_\alpha, z).
\]
Therefore, 
\[
d(q, z_{\alpha})  - d(\alpha_{2r}, z_{\alpha}) \leq  2 d(z_\alpha, z) \leq r. 
\]
This implies that $d(q, \alpha_{2r}) \leq r$ and hence $\Norm{q} \geq r$. 
We have shown that the geodesic $\alpha$ can be $(9q, Q)$--redirected to $\beta$ at radius $r$ 
for all $r$ large enough.
\end{proof}

\begin{proposition}\label{SubMorsemin}
Let $X$ be a proper geodesic metric space satisfying Assumption 0, let $\bfa \in P(X)$ be
a quasi-redirecting class and let $\alpha \in \bfa$ be $\kappa$--Morse quasi-geodesic ray. 
Then $\bfa = [\alpha]_s$ and, for an appropriate choice of a central element in $\bfa$,
we can take $f_{\bfa}(q, Q) = (9q, Q)$.  Moreover, $\bfa=[\alpha]_s$ is minimal in the partial order of $P(X)$. That is to say, if $\bfa$ is $\kappa$ Morse, then
\[
\bfb \preceq \bfa \Longrightarrow \bfb \sim \bfa.
\]
\end{proposition}

\begin{proof}
By \cite[Lemma 4.2]{QRT2}, the class $[\alpha]_s$ always contains a geodesic representative $\alpha_0$
which is also $\kappa$--Morse. Then $[\alpha]_s = [\alpha_0]_s$ and by  \lemref{kapparedirect}, 
$[\alpha_0]_s \subseteq [\alpha_0]$. In particular, $\alpha \in [\alpha_0]$ and $[\alpha] = [\alpha_0]$. It remains to show that 
$[\alpha_0] \subseteq [\alpha_0]_s$. 

Pick a $\qq$--ray $\beta \in [\alpha_0]$. Then there is $\qq'$ and, for $r>0$, there is a $\qq'$--ray, denoted $\gamma_r$, 
quasi-redirecting $\beta$ to $\alpha_0$ at radius $r$. Since $\alpha_0$ is $\kappa$--weakly Morse (\defref{Def:Weakly-Morse}), 
we have 
\[
\gamma_r \subset \calN_\kappa\big(\alpha, m_\alpha(f_{\alpha_0}(\qq))\big).
\] 
But this holds for every $r$ and $\beta \subset \cup_r \gamma_r$. Hence $\beta$
sublinearly fellow travels $\alpha_0$ and $\beta \in [\alpha_0]_s$. 

Note that we just showed $\beta \prec \alpha_0$ implies $\beta \in [\alpha_0]_s = \bfa =[\alpha]$. 
That is, any quasi-geodesic ray smaller than $\alpha$ is in $\bfa$. This means 
$\bfa$ is minimal in $P(X)$. The assertion that 
\[f_\bfa(\qq) = (9q, Q)\] follows from the proof
of \lemref{kapparedirect}. 
\end{proof}

\subsection{Topology of $\pka X$} \label{Sec:k-Topology}
Similar to $X \cup \partial X$, the topology in $X \cup \partial_\kappa X$ is defined using 
a neighborhood basis (see \cite{QRT2} for more details). Namely, for $\bfa \in \pka X$, let 
$m_{\alpha_0}$ be the Morse gauge for $\alpha_0$ (the central element in $\bfa$). 
For $r>0$, we define the set $\calU_\kappa(\bfa, r) \subseteq X \cup \partial_\kappa X$ as follows:
\begin{itemize}
\item  An equivalence class $\bfb \in \partial_\kappa X$ belongs to $\calU(\bfa, r)$ if,
for any $\qq$--ray $\beta \in \bfb$, where $m_{\alpha_0}(\qq) \leq \frac{r}{2 \kappa(r)}$,
we have
\[ 
\beta\vert_r \subseteq \calN_{\kappa}\big(\alpha_0, m_{\alpha_0}(\qq)\big).
\]
\item  A point $p \in X$ belongs to $\calU_\kappa(\bfa, r)$ if $d_X(\go, p) \geq r$ and, 
for every $\qq$--quasi-geodesic segment $\beta$ connecting $\go$ to $p$, 
where $m_{\alpha_0}(\qq) \leq \frac{r}{2 \kappa(r)}$,  we have
\[
\beta|_{r} \subseteq \calN_{\kappa}\big(\alpha_0, m_{\alpha_0}(\qq)\big).
\]
\end{itemize}

\begin{theorem} \label{Thm:two-bordifications}
The bordification $X \cup \, \pka X$ is a topological subspace of $X \cup \, \partial X$. 
\end{theorem}

\begin{proof}
The restriction of both topologies to $X$ is the topology defined by the metric on $X$. 
We will prove the theorem with the following two claims:
\begin{claim}
For every $\bfa \in \pka X$ and for all $r>0$, there exists $R$ such that 
\[
\big (\calU(\bfa, R) \cap \partial_\kappa X  \big ) \subseteq \calU_\kappa (\bfa, r).
\]
\end{claim}
\begin{proof}[Proof of Claim] \renewcommand{\qedsymbol}{$\blacksquare$} 
For $r>0$, let 
\[
\kappa' (t) =  \sup_{m_{\alpha_0}(\qq) \leq \frac{r}{2 \kappa(r)}} m_{\alpha_0}(F_\bfa(\qq)) \cdot  \kappa(t)
\]
and let $R$ be as in \defref{Def:Strongly-Morse} where $\alpha$ is $\alpha_0$. 
Let $\bfb \in \calU (\bfa, R) \cap \pka X$. Then every $\qq$--ray $\beta \in \bfb$ can be 
$F_\bfa(\qq)$--redirected to $\alpha_0$ at radius $R$.  That is, $\beta|_R$ is a subsegment 
of a $F_\bfa(\qq)$--quasi-geodesic segment with end points on $\alpha_0$. But $\alpha_0$ is $\kappa$--Morse
and, by Definition~\ref{Def:Weakly-Morse}, 
\[
\beta|_R \subset \calN_\kappa(m_\bfa(F_\bfa(\qq)), \alpha_0 ). 
\]
From definition of $\calU_\kappa(\bfa, r)$, we need to check that if 
$m_{\alpha_0}(\qq) \leq \frac{r}{2 \kappa(r)}$ then 
\[
\beta|_r \subset \calN_\kappa(m_\bfa(\qq), \alpha_0 ).
\]
But that is exactly what \defref{Def:Strongly-Morse} implies for $R$ chosen as above. 
The proof for a point $x \in \calU (\bfa, R) \cap X$ is similar.
\end{proof}

\begin{claim}
For every $\bfa \in \pka X$ and $r> 0$ there exists $R > 0$ such that
\[ \calU_\kappa(\bfa, R) \subseteq \calU(\bfa, r).\]
\end{claim}

\begin{proof}[Proof of Claim] \renewcommand{\qedsymbol}{$\blacksquare$} 
We can use the same argument as in the proof of \lemref{kapparedirect}. 
In that proof, we needed to know that $z$ with $\Norm{z} \leq 7r$ is 
still in the sublinear neighborhood of the central geodesic. Hence, 
for $r$ large enough and $R \geq 7r$. The claim holds. 

Again, the proof for a point $\mathbb{x} \in \calU (\bfa, R) \cap X$ is similar is hence it is omitted.
\end{proof}
Since the sets $\calU (\bfa, R)$ and $\calU_\kappa (\bfa, R)$ for neighborhood basis for
$X \cup \, \pka X$ and  $X \cup \, \partial X$ respectively, the above claims prove the
Proposition. 
\end{proof}

\begin{corollary}
Let $X$ be a proper Gromov hyperbolic space. Then $X \cup \partial X$ is homeomorphic
to the Gromov compactification of $X$. 
\end{corollary}

\begin{proof} 
Since all geodesics in $X$ are Morse, the Gromov boundary is homeomorphic to the 
$\kappa$--boundary for $\kappa =1$. Also, again because all geodesics are $\kappa$-Morse (where $\kappa$=1), 
the inclusion $\partial_\kappa \subset \partial X$ is in fact a bijection. Therefore, 
in view of \thmref{Thm:two-bordifications}, the Gromov boundary is also homeomorphic 
to $\partial X$. 
\end{proof} 

\section{Symmetry of redirecting}
The partial order we defined on quasi-geodesics  is not in general symmetric (see \secref{counterexample} for an example).  
However,  symmetry is also a natural assumption to consider and it holds for our main class of examples  
that is discussed in \secref{Sec:ATM}. Here, we briefly investigate the consequences of the symmetry assumption 
in general terms (not specific examples). Notably, we will show that under this assumption, the boundary 
$\partial X$, when defined, is Hausdorff. 

\subsection*{Assumption 3}(Symmetry of redirecting) 
For every pair of quasi-geodesic rays $\alpha$ and $\beta$,
\[ 
\alpha \preceq \beta \quad \Longrightarrow \quad \beta \preceq \alpha.
\]

First we note that  Assumption 3 can be used instead of  Assumption 1. 

\begin{lemma}\label{geo-rep}
Assumption 0 and Assumption 3 imply  Assumption 1.
\end{lemma}

\begin{proof}
Consider a class $\bfa = [\alpha]\in \partial X$ where $\alpha$ is a $\qq$--ray. By Assumption 0, there is 
geodesic segment $\alpha_i$ connecting $\go$ to $\alpha(i)$. Taking a subsequence, we can assume 
that segments $\alpha_i$ point-wise converge to a geodesic ray $\alpha_0$ and we have as usual
$\alpha_0 \preceq \alpha$. Hence by Assumption 3 $\alpha \preceq \alpha_0$ and thus $\alpha_0 \in \bfa$. That is, the class $\bfa$ has a representative 
that is a geodesic ray. Therefore,  Assumption 1 holds for $\qq_0=(1,0)$. 
\end{proof}

However, Assumption 2 is still necessary. We demonstrate  that in the following example:
\begin{example}\label{example2}
Here we construct a geodesic metric space that satisfies Assumption 0, 1 and 3 but not
Assumption 2. The space $X$ consists of the following:
\begin{itemize}
\item All points in the xy-plane where $y \geq 0, x\geq 0$ and $y \leq \frac 1x$ (shaded area).
\item For every positive integer $n$, points in the xy-plane on the ray $y = n x$, $y \geq 0$
(we denote this ray by $a_n$). 

\begin{figure}[ht]
\begin{tikzpicture}[domain=0.25:4]

\draw[thick] (0,0) -- (4,0) node[right] {$a_0$};
\draw[thick] (0,0) -- (0,4) node[above] {$b$};
\draw[thick] (0,0) -- (1,4) node[above]{$a_4$};
\draw[thick] (0,0) -- (1.5,4) node[above]{$a_3$};
\draw[thick] (0,0) -- (3,4) node[above]{$a_2$};
\draw[thick] (0,0) -- (4,4) node[above]{$a_1$};
\draw[color=blue] plot (\x,{1/(\x)});
\node at (3,1){$\dfrac{1}{x}$};
\draw[dashed] (3.8,0.05) -- (3.8,0.21);
\draw[dashed] (3.8,0.05) -- (3.8,0.21);
\draw[dashed] (3.8,0.05) -- (3.8,0.21);
\draw[dashed] (3.6,0.05) -- (3.6,0.22);
\draw[dashed] (3.4,0.05) -- (3.4,0.25);
\draw[dashed] (3.2,0.05) -- (3.2,0.27);
\draw[dashed] (3,0.05) -- (3,0.29);
\draw[dashed] (2.8,0.05) -- (2.8,0.31);
\draw[dashed] (2.6,0.05) -- (2.6,0.33);
\draw[dashed] (2.4,0.05) -- (2.4,0.35);
\draw[dashed] (2.2,0.05) -- (2.2,0.39);
\draw[dashed] (2,0.05) -- (2,0.43);
\draw[dashed] (1.8,0.05) -- (1.8,0.49);
\draw[dashed] (1.6,0.05) -- (1.7,0.53);
\draw[dashed] (1.4,0.05) -- (1.55,0.57);
\draw[dashed] (1.2,0.05) -- (1.4,0.62);
\draw[dashed] (1,0.05) -- (1.25,0.7);
\draw[dashed] (0.8,0.05) -- (1.1,0.78);
\draw[dashed] (0.6,0.05) -- (0.95,0.89);
\draw[dashed] (0.4,0.05) -- (0.9,1);
\draw[dashed] (0.2,0.05) -- (0.75,1.2);
\draw[dashed] (0,0.05) -- (0.6,1.4);
\draw[dashed] (0,0.5) -- (0.4,2.2);
\draw[dashed] (0,1) -- (0.35,2.6);
\draw[dashed] (0,1.5) -- (0.3,3);
\draw[dashed] (0,2) -- (0.25,3.4);
\draw[dashed] (0,2.5) -- (0.2,3.8);
\draw[dashed] (0,3) -- (0.1,3.9);
\end{tikzpicture}
\qquad
\begin{tikzpicture}[scale=0.9]
[point/.style={circle,fill,inner sep=1.3pt} ] \node (N) at (1.5,0){}; \node (N') at (1.3,0.8){}; \node (O) at (0,0){}; \node (A) at (2,0){}; \node (A') at (1.8,1.1){}; \node (B) at (2.5,0){}; \node (B') at (2.3,1.4){}; \node (C) at (3,0){}; \node (C') at (2.8,1.7){}; 
 \draw[thick] (0,0) -- (7,0) node[right]{$a_0$}; 
 \draw[thick] (0,0) -- (6,3.6) node[right]{$a_n$}; 
 \pic [thick, draw=black, angle radius=14.9mm,"$n^2$", angle eccentricity=1.2]{angle=N--O--N'}; 
 \path[thick] (1.6,0.4)edge [bend right=10](1.6,3){};
  \path[thick] (2.9, 0.8)edge [bend right=10](2.9,4){};
    \path[thick] (4.3, 1.2)edge [bend right=10](4.3,5){};
       \path[thick] (5.7, 1.5)edge [bend right=10](5.7,6){};
 \pic [thick, draw=black, angle radius=27mm,"$2n^2$",angle eccentricity=1.15]{angle=A--O--A'}; 
 \pic [thick, draw=black, angle radius=40mm,"$3n^2$",angle eccentricity=1.1]{angle=B--O--B'}; 
 \pic [thick, draw=black, angle radius=53mm,"$4n^2$",angle eccentricity=1.1]{angle=C--O--C'}; 
 \node at (0.75,-0.2){$n$}; \node at (2.5,-0.2){$n$}; \node at (3.8,-0.2){$n$}; 
 \node at (5.4,-0.2){$n$}; \node at (0.75,0.7){$n$}; 
 \node at (2,1.45){$n$}; \node at (3.3,2.2){$n$};
  \node at (4.3,2.8){$n$};   
  \end{tikzpicture}
\caption{A copy of the graph on the right is attached to the pairs $(a_0, a_n)$ for $n=1, 2, 3, etc$.}
\end{figure}
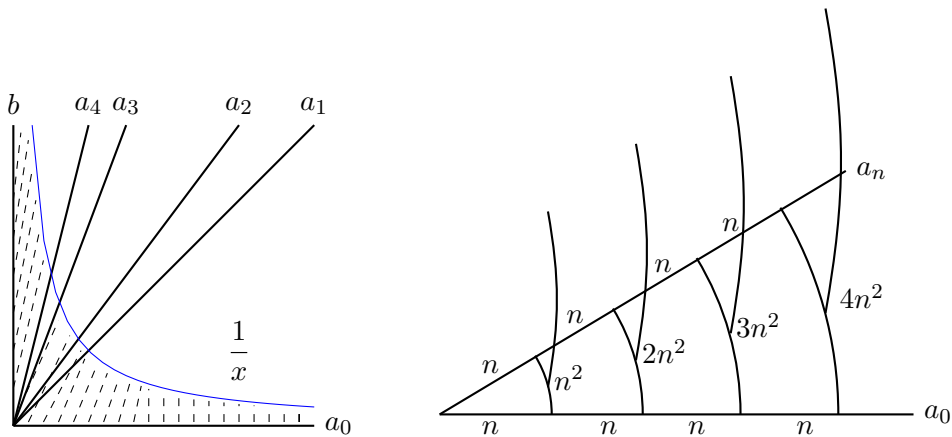

\item For every pair of positive integers $n$ and $k$, a segment $\omega_{k,n}$ of length $kn^2$ 
connecting $a_n(kn) = (kn, kn^2)$ to $a_0(kn)= (kn, 0)$. This segment is not embedded in the $xy-plane$. 
\item At the mid-point of each segment $\omega_{k,n}$ from $a_0(kn)$ to  $a_n(kn)$, attach a \emph{hair}, that is, an infinite geodesic ray.
\end{itemize}

The rays $a_n$ is isometrically embedded in $X$. Also, the rays
\[
a_n[0,kn] \cup \omega_{k,n} \cup a_0[kn, \infty) 
\quad\text{and}\quad 
a_0[0,kn] \cup \omega_{k,n} \cup a_n[kn, \infty) 
\]
are $((n+1), 0)$--quasi-geodesic rays. That is, $a_0$ and $a_n$ can be quasi-redirected to each other 
for $n=0, 1, 2, \dots$ are hence all $a_n$ are in the same class.  Choosing 
$a_0$ as the central element, the redirecting constants from $a_n$ to $a_0$ is 
$((n+1), 0)$. Thus Assumption 2 is not satisfied.

To see that $X$ satisfies Assumption 0, we note that by construction $X$ is a geodesic space. The region under the graph $y = 1/x$ ensures that the
space $X$ is proper. The attached hairs ensures every point outside of the shaded region lies on a geodesic ray. In fact, each hair
represents its own equivalence class and each such class contain exactly two geodesic representatives. 
Furthermore, for $\qq_0$ large enough, we can ensure that every point in the shaded region lies on a $\qq_0$-quasi-geodesic ray. Thus Assumption 0 is satisfied.

Note that $a_0$ can be quasi-redirected 
to every $a_n$ and vice versa. Hence, all geodesic rays $a_n$ belong to the same class $\bfa$. 
The geodesic ray $b$ and rays associated to each hair represent incomparable classes, which is also 
not comparable to $\bfa$. Therefore, every class has a geodesic representative and hence, Assumption 1
holds. Also, since all the classes are incomparable, Assumption 3 holds. But as we saw in the preceding paragraph, 
Assumption 2 does not hold. 
\end{example}

Lastly, we prove the main theorem of this section. 

\begin{theorem}\label{sym-Haus}
Let $X$ be a proper geodesic metric space where Assumption 0, 2 and 3 hold.
Then $\partial X$ is a Hausdorff. 
\end{theorem} 

\begin{proof}
Consider distinct points $\bfa, \bfb \in \partial X$. We will show that $\calU(\bfa, r) \cap \calU(\bfb, r) = \emptyset$
for some $r>0$. Otherwise, there is a sequence 
$r_i \to \infty$ and a sequence of points 
\[
\bfc_i \in \calU(\bfa, r_i) \cap \calU(\bfb, r_i).
\]
Let $\gamma_i$ be a $\qq_0$--ray in $\bfc_i$. Up to taking a subsequence, we can assume that 
$\gamma_i$ point-wise converges to a $\qq_0$--ray $\gamma$ representing a class $\bfc$. 
\lemref{Lem:limit} implies that $\gamma \prec \alpha_0$ and  $\gamma \prec \beta_0$. By symmetry $\alpha_0$ and $\beta_0$ 
can also be quasi-redirected to $\gamma$ and, by transitivity (\lemref{L:Transitive}), $\alpha_0$ and $\beta_0$ 
can be quasi-redirected to each other. Therefore, $\bfa = \bfb$. That is, $\bfa$ and $\bfb$ could not have been distinct. 
\end{proof} 

\section{Asymptotically tree-graded spaces relative to mono-directional subsets}\label{Sec:ATM}
In this section we study a class of spaces called \emph{asymptotically tree-graded spaces}. These spaces are first defined in \cite{DS05} 
and  are studied in \cite{DS08, sis12, sistoAT, OS}, among others. In this section we adopt the language and notation of \cite{DS05}.
These are metric space analogues of relatively hyperbolic groups. 
We focus on a sub-class where the peripheral subsets are mono-directional (see Section~\ref{Sec:Single}). 
We refer to these spaces as \emph{asymptotically tree-graded spaces relative to mono-directional subsets}, or ATM spaces, for short. 
We show that these spaces satisfy Assumptions 0, 1 and 2  and thus have well-defined quasi-redirecting boundaries. 
We also show that if $X$ is an ATM space then $\partial X$ is metrizable and compact. Finally, we see that, 
when $X$ is the Cayley graph of a group, $\partial X$ is an alternate description of the 
Bowditch boundary. 

\subsection{Background on asymptotically tree-graded spaces}
We start by recalling the necessary definitions and properties as laid out in \cite{DS05}. We refer the readers to 
\cite{DS05} for a more complete treatment. Asymptotically tree-graded spaces are metric spaces whose asymptotic cones are tree-graded. However, one of the main results of \cite[Theorem 4.1]{DS05} states that asymptotically tree-graded spaces 
can be characterized without referring to asymptotic cones. We use this characterization.

\begin{definition}\label{char-of-AT} 
We say that a subset of $X$ is a geodesic $k$--\emph{gon} if it is a union of $k$ geodesic segments $p_1,...,p_k$  with 
pairs of endpoints $((p_i)_-,  (p_{i})_+)$ such that, $(p_i)_+= (p_{i+1})_-$ for every $i=1,...,(k-1)$ and $(p_k)_+= (p_{1})_-$. Also for each of $i = 1, . . . , k$, 
we denote the polygonal curve $P \backslash (p_{i-1} \cup p_i )$ by $\calO_{x_i}$, where
 $x_i = (p_{i-1})_{+} = (p_i)_{-}$. 
 
Let $\vartheta>0$ and $v \geq 8$ be constants. We say a $k$--gon $P$ is \emph{$(\vartheta, 2, v)$--fat} if the following 
properties hold:
\begin{enumerate}[(F1)]
\item (Large comparison angles and large inscribed radii in interior points) For every edge $p$ with endpoints $\{x, y\}$ we have
\[
d(p\backslash\calN_{2\vartheta}({x,y}),P\backslash p) \geq \vartheta;\]
\item (Large edges and large inscribed radii in vertices) For every vertex $x$ we have
\[
d(x,\calO_x)\geq v\, \vartheta.
\]
\end{enumerate}
Let $(X,d)$ be a geodesic metric space and let 
$\calA$ be a collection of proper geodesic subsets of $X$. The metric space $X$ is asymptotically tree-graded 
with respect to $\calA$ if and only if the following properties are satisfied:
\begin{itemize}
\item[(AT1)](Isolated subsets) For every $\delta > 0$ there $D>0$ such that, for distinct sets $A, A' \in \calA$, the diameter of the intersection
$\calN_\delta (A) \cap \calN_\delta (A')$ is  bounded by $D$.
\item[(AT2)](Hyperbolicity outside of the special subsets) For every $\qq=(q, Q)$ and every $\theta \in [0, 0.5)$ there exists a number 
$M_{\qq}> 0$ such that for every $\qq$-quasi-geodesic  segment $\beta$ defined on $[0,\ell]$ and every $A \in \calA$ 
with $\beta(0), \beta(\ell) \in \calN_{\theta\ell/q} (A)$ we have $\beta([0, \ell]) \cap \calN_{M_{\qq}}(A) \neq \emptyset$.

 When $\qq = (1, 0)$ we also write $M_0$ for when $\qq = (1,0)$. We note that even without assumption of cocompactness, the choice of $M_{\qq}$ is uniform and 
 does not depend on specific choices of $A$.
 
\item[(AT3)](Fat polygons are entirely in $A$.) For every $k \geq 2$ there exist $\vartheta >0, v \geq 8$ and $\chi > 0$ such that every $k$-gon $P$ with geodesic edges which is  $(\vartheta, 2, v)$-fat satisfies $P \subset \calN_\chi(A)$ for some $A \in \calA$.  
\end{itemize}
\end{definition}

\begin{definition}[ATM] \label{Def:ATM} 
We say a space $X$ is \emph{ATM} if it satisfies Assumption 0 and it is 
asymptotically tree-graded with respect to a collection of mono-directional subspaces $\calA$ that each satisfies Assumption 2.  
We denote the redirecting function with respect to the central element of $A$ by $f_A$.
Note that the function $f_A$ depends on the choice of base point, however, we do not 
include this in the notation.  
\end{definition}

\subsection*{Transition points and transient rays}
We now recall some basic properties of ATM spaces.  The following definitions 
and results are developed in \cite{sis12, Hru10}.

 \begin{definition}\cite[Definition 3.9]{sis12}\label{deep}
 Let $\alpha$ be a path in $X$. For $M, c>0$, define the $\deep_{M,c} (\alpha)$ to be the set of points
 $x \in \alpha$ such that there exists a subpath of $\alpha$ containing $x$ 
 with endpoints $x_1,x_2$ and $A \in \calA$ where
 \[
x_1,x_2 \in N_M(A)  
\qquad\text{and}\qquad 
d(x,x_i) \geq c \quad \text{for $i = 1, 2$}. 
\]
Thinking of $\alpha$ as a subset of $X$, define 
\[
\trans_{M,c}(\alpha) = \alpha - \deep_{M,c} (\alpha)
\]
to be the set of $(M,c)$--transition points of $\alpha$. 
\end{definition}

\begin{proposition}\cite{sis12, DS05} \label{sis12}
Let $X$ be a proper, geodesic, asymptotically tree-graded metric space.
For every $\qq$ there exist constant $M = M(\qq)$, $c=c(\qq), D=D(\qq)$ and $\rho(\qq)$ such that the followings hold. 
Let $\alpha \from [a,b] \to X$ be a $\qq$--quasi-geodesic segment.
\begin{enumerate}[(I)]
\item  The set $\deep_{M, c}(\alpha)$ is a disjoint union 
of subpaths each contained in $N_{\rho M}(A)$ for distinct sets $A \in \calA$. 
\item For any pair of $\qq$--quasi-geodesic 
segments $\alpha, \beta$ with the same endpoints, we have
\[ 
d_{\rm Haus} \big( \trans_{M,c}(\alpha), \trans_{M,c}(\beta)\big) \leq D.
\]

\item Moreover, for every $A \in \calA$ there are times $t, s \in [a,b]$ such that during the interval $[a, s]$
$\alpha$ approaches $A$ at a linear speed, during the interval $[t, b]$ $\alpha$ moves away from $A$ 
at a linear speed and $\alpha[s,t] \subset N_{\rho M} (A)$. 
\end{enumerate}
The same also holds for quasi-geodesic rays. 
\end{proposition}

The statements of (1) and (2) is contained \cite[Proposition 5.7]{sis12}. The statement (2) follows from 
\cite[Lemma 4.17]{DS05}.
 
 \begin{definition} \label{Def:transition} 
 Let $\alpha$ be a $\qq$--quasi-geodesic segment or $\qq$--ray in $X$. We say a point $\alpha(t)$ 
 is a \emph{$\qq$--transition point} of $\alpha$ if
\[
\alpha(t)  \in \trans_{M(\qq), c(\qq)}(\alpha),
\] 
where $M(\qq), c(\qq)$ are as Proposition~\ref{sis12}. 
\end{definition}
 
\begin{definition} \label{Def:Transient} 
Let $\alpha$ be a $\qq$--ray. We say $\alpha$ is a \emph{$\qq$--transient} ray if,
there is a sequence of times $t_i \to \infty$ such that $\alpha(t_i)$ is a 
$\qq$--transition point of $\alpha$. 
\end{definition}

Note that if $\qq' \geq \qq$ and $\alpha$ is a $\qq$--ray, then $\alpha$ is also a 
$\qq'$--ray. But, the set of $\qq$--transition points in not necessary a subset or a superset
of the set of $\qq'$--transition points because, to ensure 
\[
 \deep_{M_1,c_1} (\alpha)  \subset \deep_{M_2,c_2} (\alpha)
\] 
we need $c_1 \geq c_2$ and  $M_1\leq M_2$. However, as we shall see, the quality of being
a transient ray is independent of choice of $\qq$. 

\begin{lemma}
Let $\alpha$ be a $\qq$--ray and let $M, c$ and $\rho$ be as in \propref{sis12}.
Then either $\alpha$ is a $\qq$--transient ray or $\alpha$ is eventually contained in $N_{\rho M}(A)$ 
for some $A \in \calA$. 
Furthermore, if $\alpha$ is a $\qq$--transient ray and $\qq' \geq \qq$, then $\alpha$ is also a 
$\qq'$--transient ray. 
\end{lemma} 

\begin{proof} 
By definition, if $\alpha$ is not $\qq$--transient then $\alpha$ is eventually contained in 
$\deep_{M, c}(\alpha)$. 
This means that there is $A \in \calA$ and $t_A$ such that $\alpha(t_A) \in N_{M}(A)$ and, for $t \geq t_A$,
there is $s \geq t$ such that $\alpha(s)  \in N_{M}(A)$. Hence, $\alpha[t_A, \infty) \subset N_{\rho M}(A)$. 

If $\alpha$ is $\qq$--transient, then for every $A \in \calA$, $\alpha$ eventually leaves every neighborhood of $A$
(see part (III) of \propref{sis12}). In particular, $\alpha$ leaves the $(\rho(\qq') \cdot M(\qq'))$--neighborhood of 
every $A\in\calA$. The above argument shows that $\alpha$ is also $\qq'$--transient. 
\end{proof} 

Part (II) of \propref{sis12} can be restated as follows: 
 
\begin{corollary}[Relative fellow traveling property]\label{RFT}  
Suppose $\alpha, \beta$ are $\qq$--quasi-geodesic segments 
that start and end at the same point and let $x \in \alpha$ be a $\qq$--transition point of $\alpha$. 
Then there exists a point $y \in \beta$ that is a $\qq$--transition point of $\beta$ and 
\[
d(p, q) \leq D(\qq).
\]
\end{corollary}

Finally we recall the notion of \emph{saturation} of a quasi-geodesic (see \cite[Definition 4.20]{DS05}).  

\begin{definition} \label{Def:saturation}
Let $\alpha$ be a $\qq$--ray or $\qq$--segment. The \emph{saturation} of $\alpha$, denoted by $\Sat(\alpha)$, is 
the union of $\alpha$ and all $A \in \calA$ with $N_{M(\qq)}(A) \cap \alpha \neq \emptyset$.
\end{definition} 

The saturation is quasi-convex (see \cite[Lemma 4.25]{DS05}). 

\begin{lemma}[Uniform quasi-convexity of saturations]\label{Lem:sat-convex}
 For every $\qq'$, there exists $\tau(\qq')>0$ such that for every $L\geq1$ and every $\qq$--ray or $\qq$--segment 
 $\alpha$, $\Sat(\alpha)$ has the property that, for every $\qq'$--segment $\gamma$ with endpoints $N_L(\Sat(\alpha))$, 
 we have  
 \[
  \gamma \subset N_{\tau(\qq') \cdot L}(\Sat(\alpha)).
 \] 
\end{lemma}

 \subsection*{Redirecting in ATM spaces}
In this section we show that $\preceq$ is a symmetric relation. Elements of $P(X)$ are divided to 
transient class with infinitely many transition points and non-transient class that are eventually contained 
in a bounded neighborhood of some $A \in \calA$.
 
\begin{lemma}\label{transient-class}
If $\alpha$ is a transient quasi-geodesic ray and $\beta \preceq \alpha$, then $\beta$ is also a transient ray.
\end{lemma}

\begin{proof}
We choose $\qq$ large enough so that both $\alpha$ and $\beta$ are $\qq$--rays and $\beta$ can be 
$\qq$--redirected to $\alpha$ and let $M$, $c$ and $D$ be as in \propref{sis12}. Assume for contradiction that 
$\beta$ is not transient. Then there a radius $r_0$ such that 
\[
\beta|_{\geq r_0} \subset \deep_{M,c}(\beta).
\] 
In particular, for every $r>0$, there are points $z_1, z_2 \in \beta$ such that $\Norm{z_1} \leq r_0-c$, $\Norm{z_2} \geq r+c$ and
$d(z_i, A) \leq M$ for $i=1,2$ and some $A \in \calA$.

Since $\alpha$ is transient, there exists $x \in \alpha$ with $\Norm{x} \geq r_0 + D$ such that $x$ is a $\qq$--transition point.
Let $r = \Norm{x} + D$ and let $z_1$ and $z_2$ be as above. Let $R \geq \Norm{z_2}$ and let $\gamma$ be a 
$\qq$--ray redirecting $\beta$ to $\alpha$ at radius $R$. 

By \corref{RFT}, there is a point $y \in \gamma$ with $d(x, y) \leq D$ such that $y$ is a $\qq$--transition point 
of $\gamma$. Note that $r_0 \leq \Norm{y} \leq r$ which implies, $d(y, z_i) \geq c$ for $i=1,2$.  Also, 
\[
B_R(\go) \cap \gamma = B_R(\go) \cap \beta,
\]
thus, the segment $[z_1, z_2]_\beta$ is also a subsegment of $\gamma$.
Hence $y \in \deep_{M,c}(\gamma)$. This is a contradiction. 
\end{proof}

 As a consequence, being transient is a property of an equivalence class. We say $\bfa \in P(X)$ is transient if some 
 quasi-geodesic ray in $\bfa$ is transient which implies all rays in $\bfa$ are transient.

 \begin{proposition}\label{Prop:Assump}
Let $\bfa \in P(X)$ be a transient class.  Then $\bfa$ contains a geodesic ray $\alpha_0$. 
If we choose $\alpha_0$ as the central element of $\bfa$, we have 
$f_\bfa(q, Q) = (9q, Q)$. Furthermore, if $\beta \preceq \alpha_0$ then $\beta \in \bfa$. 
\end{proposition}
\begin{proof}
Let $\bfa \in P(X)$ be a transient class and $\alpha \in \bfa$ be a transient $\qq$--ray. 
By Lemma~\ref{limit-geo}, there exists a geodesic ray $\alpha_0 \preceq \alpha$. 
Since $\alpha$ is transient, \lemref{transient-class} implies that $\alpha_0$ is also transient. 

In fact, assume $\alpha_0$ can be $\qq'$--redirected to $\alpha$. Then there is a 
sequence of times $t_i \to \infty$ such that $\alpha(t_i)$ is a $\qq'$--transition point. 
\lemref{RFT} implies that, for every $t_i$, there is $s_i$ such that 
$\alpha_0(s_i)$ is a $\qq'$--transition point and  $d(\alpha_0(s_i), \alpha(t_i)) \leq D(\qq')$. 
For $i$ large enough, we have 
\[
d\big( \alpha_0(s_i), \alpha(t_i) \big)\leq \frac 12 \Norm{\alpha_0(s_i)}.  
\]
Thus by Part (II) of Surgery Lemma~\ref{surgery}, $\alpha$ can be $(9q, Q)$--redirected to $\alpha_0$. 
In particular, $\alpha \preceq \alpha_0$. Every other $\qq'$--ray in $\bfa$ can similarly 
be $f_\bfa(\qq')$--redirected to $\alpha$ using  Part (II) of Lemma~\ref{surgery} for 
$f_\bfa$ as in the statement of the proposition. 

The proof of the last assertion is the same as above. If $\beta \preceq \alpha_0$ then, 
by \lemref{transient-class}, $\beta$ is also transient and $\beta$ and $\alpha_0$ are
near each other at all the transition points. Part (II) of Surgery Lemma~\ref{surgery}
implies that $\alpha_0$ can be quasi-redirected to $\beta$. 
\end{proof}

\begin{lemma}\label{non-transient-class}
If $\alpha$ is a non-transient quasi-geodesic ray and $\beta \preceq \alpha$, then $\beta$ is also a non-transient ray.
In fact, there is $A \in \calA$ such that both $\alpha$ and $\beta$ are eventually contained in a bounded neighborhood of $A$.
\end{lemma}

\begin{proof}
Let $\qq$ be large enough such that $\alpha$ and $\beta$ are $\qq$--rays and there is a family of $\qq$--rays 
$\gamma_i$ redirecting $\beta$ to $\alpha$ at radius $r_i$, with $r_i \to \infty$. 
Let $M, c, D$ and $\rho$ be as in \propref{sis12}. 

The set $\deep_{M,c}(\alpha)$ is a disjoint union of subpaths
and there is a $\qq$--transition point between any two adjacent ones.  Since $\alpha$ is non-transient, the tail of $\alpha$ after the last transition 
point is all in a deep segment which, by part (I) of \propref{sis12}, stays in a bounded neighborhood of some $A \in \calA$. 

If $\beta$ contains $\qq$--transition points at arbitrarily large radii, then $\gamma_i$ contain transition points at arbitrary radii 
for large enough $i$.  But $\gamma_i$ are eventually equal to $\alpha$, hence (by \propref{RFT}) $\alpha$ contains $\qq$--transition point at arbitrarily large radius. 
This is a contradiction since $\alpha$ is not transient.  

Therefore, $\beta$ is non-transient and hence, for some $B \in \calA$, $\beta$ is eventually contained in a bounded neighborhood of $B$. 
If $B \not = A$, then $\gamma_i$ stay near $B$ for a long time and then near $A$. This means $\gamma_i$ has a $\qq$--transition 
at arbitrarily larger radii for large values of $i$. Hence $\alpha$ has infinitely many $\qq$--transition points which is again a contradiction. 
The contradiction implies $B=A$. This finishes the proof. 
\end{proof} 

As a consequence, being non-transient is also a property of an equivalence class. We say $\bfa \in P(X)$ is a non-transient if some 
quasi-geodesic ray in $\bfa$ is non-transient which implies all rays in $\bfa$ are non-transient and in fact, they all eventually 
stay in a bounded neighborhood of some $A \in \calA$. 

\begin{proposition} \label{Prop:Assump-A}
Let $\bfa$ be a non-transient class, where all the rays in $\bfa$ eventually stay in a bounded neighborhood of some $A \in \calA$. 
Then every other quasi-geodesic ray that eventually stays in a bounded neighborhood of $A$ is in $\bfa$. 
Furthermore, $\bfa$ contains a geodesic ray and we can choose a function $f_\bfa$ that depends only on $f_A$. 
\end{proposition} 

\begin{proof}
Let $\alpha$ be a quasi-geodesic ray in $\bfa$. By Lemma~\ref{limit-geo}, there is a geodesic ray $\alpha_0 \preceq \alpha$. 
\lemref{non-transient-class} implies that $\alpha_0$ is non-transient and eventually stays in a
bounded neighborhood of $A$. Let $\beta$ be any $\qq$--ray that eventually stays in 
a bounded neighborhood of $A$. We will show that $\beta$ can be quasi-redirected to $\alpha_0$ where
the redirecting constant depends uniformly on $f_A(\qq)$. Note that this in particular proves (setting $\beta = \alpha$)
that $\alpha_0 \in \bfa$ and hence finishes the proof. 

Let $\rho_0 = \rho(1,0)$, $M_0 = M(1,0)$, $\rho =\rho(\qq)$ and $M= M(\qq)$ be as in 
\propref{sis12}. Let $t_0$ be the first time $\alpha_0$ enters the $M_0$--neighborhood of $A$ and let 
$\go_A \in A$ be a closest point in $A$ to $\alpha_0(t_0)$. Note that $\alpha_0$ stays in the 
$\rho_0 M_0$--neighborhood of $A$ after $t_0$. Let $\alpha_0^A$ be a quasi-geodesic ray in 
$A$ starting at $\go_A$ that fellow travels $\alpha_0$, namely, we can compose $\alpha_0|_{[t_0, \infty)}$ with 
the closest point projection to $A$ and use \lemref{Lem:Tame} to tame the resulting quasi-geodesic. 

Similarly, let $t_\beta$ be the first time $\beta$ enters the $M$--neighborhood of $A$. 
Consider $\Sat(\alpha_0)$  (See Definition~\ref{Def:saturation}) and recall the uniform quasi-convexity of saturations \cite[Lemma 4.25]{DS05}. Since $\beta$ 
is contained in a bounded neighborhood of the saturation of $\alpha_0$, $\beta$ has to enter the $M$--neighborhood 
of $A$ in a bounded distance from $\go_A$ where the bounds depend only on $\qq$. 
Let $\beta^A$ be a quasi-geodesic ray in $A$ starting at $\go_A$ that fellow travels $\beta$ constructed as above. 

Now $\alpha_0^A$ and $\beta^A$ are $\qq'$--rays in $A$ where $\qq'$ depends uniformly on $\qq$. 
Since $A$ is mono-directional, $\beta^A$ can be $f_A(\qq')$ redirected to $\alpha_0^A$. 
Part (II) of \lemref{surgery} implies that $\beta$ can be quasi-redirected to $\beta^A$, and $\alpha_0^A$ can be 
quasi-redirected to $\alpha_0$. Now, arguing as in \lemref{L:Transitive}, we see that $\beta$ 
can be quasi-redirected to $\alpha_0$ with quasi-redirection constant uniformly depending $f_A(\qq')$. 
This finishes the proof. 
\end{proof}

\begin{corollary} \label{Cor:Symmetric} 
The relation $\preceq$ is symmetric. In particular, different classes in $P(X)$ are not comparable.
\end{corollary}

\begin{proof}
This follows immediately from \propref{Prop:Assump} and \propref{Prop:Assump-A}. 
\end{proof}

\subsection*{Topological properties of $\partial X$}
In this subsection, we show that $\partial X$ is compact and metrizable. We note that 
$X$ satisfies Assumption 0, 1 and 2; Assumption 0 holds by the definition of an ATM space and  
Assumptions 1 and 2 follow from  \propref{Prop:Assump} and \propref{Prop:Assump-A}. 

\begin{proposition}\label{secondcountable}
Let $X$ be an ATM space. Then $\partial X$ is second countable.
\end{proposition}
\begin{proof}
Take a sequence $r_i \to \infty$ and let 
 \[
 \calS_{i} : = \big\{ x \in X \st \Norm{x} = r_i \big\}, 
\] 
be the sphere for radius $r_i$ in $X$. Since $X$ is proper, there are finitely many points that are $1$--dense 
on any given sphere. By Assumption 0, every point $x \in X$ lies on a $\qq_0$--ray. 
Thus there exists a finite set $\calE_i$ of $\qq_0$--rays such that the set 
\[
\big\{ \alpha_{r_i} \st \alpha \in \calE_i \big\} \qquad \text{is $1$--dense in $\calS_i$}. 
\]
In addition, let $\calZ$ be the set of non-transient classes in $P(X)$. Note that $\calZ$ is countable. Now, define 
 \[
 \calY: =  \Big\{ \calU(\bfa, R) \ST  \bfa \in \calZ  \quad \text{and}\quad R \in \NN \Big\}
\ \bigcup \
 \Big\{ \calU([\alpha], R) \ST \alpha \in \bigcup_i \calE_i  \quad \text{and}\quad R \in \NN \Big\}.
 \]
This is a countable set of neighborhoods. To prove the proposition, it is sufficient to show that,
for every $\bfb \in \partial X$ and every $r>0$, there exists $ \calU(\bfa, R) \in \calY$ such that 
\begin{equation} \label{Eq:need-to-show}
 \bfb \in \calU(\bfa, R) \subset  \calU(\beta_0, r) . 
\end{equation} 
If $\bfb$ is a non-transient class, then $\bfb \in \calZ$ and we can take $\bfa=\bfb$ and 
choose any $R \geq r$. Thus, we assume $\bfb$ is a transient class.

Let the geodesic ray $\beta_0$ be the central element of $\bfb$. For any $r>0$, let 
$(q_r, Q_r) =\qq_{\max}(r)$ be as in \lemref{Lem:Max}, let $\qq_r = (9q_r, Q_r)$
and let $D_r = D(\qq_r)$ be as in \propref{sis12}. Let 
\[
r' \geq  2r + 6D_r
\]
be such that $(\beta_0)_{r'}$ is a $(1,0)$--transition point for $\beta_0$. 
Choose $R \gg r'$, let $(q_R, Q_R)= \qq_{\max}(R)$ be as in \lemref{Lem:Max},
$\qq_R=(9 q_R, Q_R)$ and let $D_R = D(\qq_R)$ be as in \propref{sis12}. Again, let 
\[
R' \geq  2 R + 6 D_R
\]
be such that $(\beta_0)_{R'}$ is a $(1,0)$--transition point for $\beta_0$. Finally,  choose $r_i \gg 2R'$ and let 
$\alpha$ be a $\qq_0$--ray in $\calE_i$ such that $d((\beta_0)_{r_i}, \alpha_{r_i}) \leq 1$. Let $\bfa = [\alpha]$. 
We check that \eqref{Eq:need-to-show} holds.

We start by showing that $\bfb \in \calU(\bfa, R)$. Part (II) of \lemref{surgery} implies that 
$\alpha$ can be $(9q_0, Q_0)$--redirected to $\beta_0$ at radius $r_i/2 \gg R'$. 
Since $(\beta_0)_{R'}$ is a transition point for $\beta_0$, arguing as in the proof of \lemref{transient-class},
we see that there is a transition point $x \in \alpha$ with 
\[
d(x, (\beta_0)_{R'}) \leq D_R
\] 
that is a $(9q_0, Q_0)$--transition point for $\alpha$. 

Let the geodesic ray $\alpha_0 \in \bfa$ be the central element of $\bfa$. 
Then there is a point $x_0 \in \alpha_0$ near $x$ that is a transition point for $\alpha_0$, that is 
\[
d(x_0, (\beta_0)_{R'}) \leq d(x_0, x) + d(x, (\beta_0)_{R'})  \leq 2 D_R.
\]
Let $\beta \in \bfb$ be a $\qq=(q, Q)$--ray with $\qq \leq \qq_R$. Since $(\beta_0)_{R'}$ is a transition point for $\beta_0$, 
there is a transition point $y \in \beta$ with $d(y, (\beta_0)_{R'}) \leq D_R$. This means 
\[
d(x_0, y) \leq d(x_0, (\beta_0)_{R'}) + d((\beta_0)_{R'}, y)  \leq 3D_R \leq R'/2. 
\]
Part (II) of \lemref{surgery} implies that $\beta$ can be $(9q, Q)$--quasi-redirected to $\alpha_0$ at radius $R'/2 \geq R$. 
Therefore $\bfb \in \calU(\bfa, R)$. 

We now show $\calU(\bfa, R) \subset  \calU(\beta_0, r)$. The proof is similar to above and uses the transition points. 
For the sake brevity, we omit some of the intermediate constants in our proof. 
Let $\bfc$ be a point in $\calU(\bfa, R)$ and let $\gamma \in \bfc$ be a $\qq$--ray with $\qq \leq \qq_{\rm max}(r)$. 
We need to show that $\gamma$ can be $(9q,Q)$--redirected to $\beta_0$ at radius $r$. Note that, $\gamma$ be the 
redirected to $\alpha_0$ at radius $R$. Since $(\beta_0)_{r'}$ is a $(1,0)$-transition points, $\alpha$ has to have 
a nearby transition point, and thus $\alpha_0$ has a nearby transition point which implies $\gamma$ has a nearby 
transition point. That is, there is a point $x$ in $\gamma$ such that 
\[
d(x, (\beta_0)_{r'}) \leq 3 D_r \leq r'/2. 
\]
Part (II) of \lemref{surgery} implies that $\gamma$ can be $(9q, Q)$--redirected to $\beta_0$ at radius $r'/2 \geq r$. 

Since $\calY$ is a countable basis for the topology we have that $\partial X$ is a second countable topological space.
\end{proof}

We use the notation 
\[
A \gg B
\]
for quantities $A$ and $B$ to mean that there exists a sufficiently large constant $C>0$ not depending 
on $\qq$ such that when $A-B \geq C$ the argument holds.
This simplifies the exposition when there are many additive errors accumulating but the errors
do not get larger when the quasi-geodesic constant get larger. 

We first establish an a criterion for an element $\bfc$ to be contained in a neighborhood 
$\calU(\bfa, r)$.

\begin{lemma}\label{Lem:Inclusion-Criterion}
Let $\bfa \in \partial X$ and $r>0$ be given and let the geodesic ray $\alpha_0$ be the central element of 
$\bfa$. We can choose the redirecting function $f_\bfa$ large enough such that the following holds. 
Then there exists $r'>r$ such that, if a geodesic $\gamma$ can be $f_\bfa(1,0)$--redirected to 
$\alpha_0$ at radius $r'$ and $\gamma \in \bfb$ (not necessarily the central element of $\bfb$) 
then $\bfb \in \calU(\bfa, r)$. 
\end{lemma} 

\begin{proof} 
Suppose that $\bfa$ is a transient class. Let $x_0$ be a $(1,0)$--transition point along $\alpha_0$ with 
\[
\Norm{x_0} \gg r \qquad\text{and let}\qquad  r' \gg \Norm{x_0}.
\] 
Since $\gamma$ can be redirected to $\alpha_0$ 
at radius $r'$, \lemref{RFT} implies that $\gamma$ has a $(1,0)$--transition 
point near $x_0$. Similarly,  any $\qq$--ray $\beta \in \bfb$ ($\qq \leq \qq_{\rm max}(r))$
has a $\qq$--transition point $y$ near $x_0$. Choosing the $\Norm{x_0}$ large enough, we can ensure 
that 
\[
d(x_0, y) \leq \frac{\Norm{x_0}}2.
\] Then Part (II) of \lemref{surgery} implies that $\beta$ can be $(9q, Q)$--redirected to $\alpha_0$ at radius $\Norm{x_0}/2 > r$. Since 
\[
(9q, Q) = f_\bfa(q,Q) < F_\bfa(q, Q),
\]
and this holds for every quasi-geodesic ray 
$\beta \in \bfb$, we have $\bfb \in \calU(\bfa, r)$.  Suppose otherwise that $\bfa$ is a non-transient class. Let 
\begin{align*}
&M_0=M(1,0), c_0 = c(1,0), \rho_0 = \rho(1,0)\\
&M = M(\qq_{\rm max}(r)), c = c(\qq_{\rm max}(r)), \rho = \rho(\qq_{\rm max}(r))
\end{align*}
be as in \propref{sis12}. 
For the rest of the proof, by a bounded constant, we mean a constant that depends uniformly 
on these constants. Since $\bfa$ is non-transient, there is $A \in \calA$ and $t_0$ such that 
$\alpha_0|_{[t_0, \infty)}$ is in a $(\rho_0 M_0)$--neighborhood of $A$. Let $\go_A$ be a closest point 
in $A$ to $\alpha_0(t_0)$. Then any $\qq$--ray that enter an $M(\qq)$--neighborhood of $A$, does 
so a bounded distance away from $\go_A$ (see \lemref{Lem:sat-convex}). 
Let $\alpha_0^A$ be a quasi-geodesic ray in $A$ that fellow travels $\alpha_0$ (meaning they stay 
a bounded distance from each other in a parametrized manner). 

We would like to argue that, every quasi-geodesic that stays near $A$ for a long time can be redirected to 
$\alpha_0$ at a radius $r$. We formulate this in terms of geodesic segments in $A$.

\begin{claim}
For every $r_A>0$ there is $r'_A\geq r_A$ such that every $\qq$--segment $\gamma_A$ in $A$ connecting $\go_A$ to $x$ with 
$d(\go_A, x) \geq r'_A$ can be $f_A(\qq)$--redirected to $\alpha_0^A$ at a radius $r_A$. That is, there is a $f_A(\qq)$--ray
$\gamma_A'$ in $A$ that $\gamma_A|_{r_A} = \gamma_A'|_{r_A}$ and $\gamma_A'$ is eventually equal to $\alpha_0^A$. 
\end{claim} 

\begin{proof}[Proof of Claim]\renewcommand{\qedsymbol}{$\blacksquare$} 
Let $r_A$ be given and assume for contradiction there is a sequence of radii $r_i^A \to \infty$ and
$\qq$--segments $\gamma_i^A$ ($\qq \leq \qq_{\rm max}(r_A)$), that cannot be $f_A(\qq)$--redirected to 
$\alpha_0^A$. Taking a subsequence, we find a $\qq$--ray $\gamma_A$ that cannon be 
$f_A(\qq)$--redirected to $\alpha_0^A$. This contradicts the fact that $A$ is mono-directional. 
\end{proof} 

Let $r_A = r -d(\go, \go_A) + 2 \rho \cdot M$ and choose 
\[
r' \gg d(\go, \go_A) + r'_A + 2\rho \cdot M
\]
where $r_A'$ is as in the claim. 

Let $\gamma$ be a geodesic ray that redirects to $\alpha_0$ at radius $r'$. 
Let $\Ge_\gamma$ be the points where $\gamma$ exits the $M$--neighborhood of $A$. 
Let $\gamma'$ be the $f_\bfa(1,0)$--ray where $\gamma'|_{r'} = \gamma|_{r'}$ and 
$\gamma$ is eventually in a bounded neighborhood of $A$. Then $\gamma'$ is entirely 
contained in a $\rho M$--neighborhood of $A$. Hence $|\Norm{\Ge_\gamma} - r'|$ is 
bounded. 

Let $\beta$ be a $\qq$--ray in $\bfb$ with $\qq \leq \qq_{\rm max}(r)$. Then, 
by quasi-convexity of the saturation (Lemma~\ref{Lem:sat-convex}), $\beta$ has to also enter the $M$--neighborhood of $A$ 
at a point $x$ at most a bounded distance away $\go_A$ and exit the neighborhood near at a point 
$y$ that is at most a bounded distance away from $\Ge_\gamma$.
\lemref{Lem:sat-convex} implies that, $[x, y]_\beta$ stays in a bounded neighborhood of $A$. 
Let $\beta_A$ be the $\qq'$--quasi-geodesic segment in $A$ starting from $\go_A$ 
that fellow travels $[x,y]_\beta$ where $\qq'$ depends only on $\qq$. By the claim, $\beta_A$ can 
be quasi-redirected to $\alpha_0^A$ at 
radius $r_A$ via a $f_A(\qq')$--ray $\beta'$ in $A$. Now $\beta|_r$ fellow travels 
$[\go, \go_A] \cup \beta_A'|_{r_A}$. By \lemref{onestep}, $\beta$ can be quasi-redirected to 
$[\go, \go_A] \cup \beta_A'|_{r_A}$. Also, $[\go, \go_A] \cup \beta_A'$ is eventually equal to 
$\alpha_0^A$ and hence fellow travels $\alpha_0$. Again, by \lemref{onestep}, it can be 
quasi-redirected to $\alpha_0$ at any radius. By transitivity, we see that 
$\beta$ can be quasi-redirected to $\alpha_0$ at radius $r$ with uniform constant that depends 
only on $f_A(\qq)$. Hence, $\bfb \in \calU(\bfa, r)$ for an appropriate function $F_\bfa$.  
\end{proof}

Recall that a topological space is regular if points can be separated from closed sets. 
\begin{proposition}\label{regular}
Let $X$ be an ATM space. Then $\partial X$ is regular.
\end{proposition}
\begin{proof}

Let $\bfa \in \partial X$ be a point and $B\subset \partial X$ be a closed set. Since, 
$B$ is closed, there is $r>0$ such that 
\begin{equation} 
B \cap \calU(\bfa, r) = \emptyset. 
\end{equation} 
For a given $r>0$, let $r'>0$  be as in \lemref{Lem:Inclusion-Criterion}. We will show that, 
for every $\bfb \in B$, there is $r_\bfb$ such that 
\[
\calU(\bfa, r') \cap \calU(\bfb, r_\bfb) = \emptyset.
\] 
Then, $\bigcup_{\bfb \in B} \calU(\bfb, r_\bfb)$ contains an open neighborhood of $B$ disjoint 
from $\calU(\bfa, r')$.
 
Pick $\bfb \in B$ and let $\beta_0$ be the central element of $\bfb$. Assume for contradiction that  
\[
\calU(\bfa, r') \cap \calU(\bfb, R) \not = \emptyset
\] 
for every $R$. Then, there is a sequence 
\[
\bfc_i \in  \calU(\bfa, r') \cap \calU(\bfb, R_i) \qquad \text{ where }R_i \to \infty.
\] 
Let the geodesic ray $\gamma_i$ be the central element of $\bfc_i$. Up to taking a subsequence, $\gamma_i$
point-wise converge to a geodesic ray $\gamma$.  Since $\gamma_i$ can be redirected to $\beta_0$ at radius $R_i$, 
$\gamma$ can be quasi-redirected to $\beta_0$ and hence $\gamma \in \bfb$. 
This is a contradiction which shows $\calU(\bfa, r') \cap \calU(\bfb, r_\bfb) \not = \emptyset$ for some 
$r_\bfb$. 
\end{proof}

\begin{theorem} \label{Thm:Nice}
When $X$ is an ATM space, $\partial X$ is a compact, second countable and metrizable topological space.
\end{theorem}

\begin{proof}
As we saw in the beginning of the subsection, $X$ satisfies Assumptions 0, 1 and 2 and hence has a well-defined 
$\partial X$. We showed in \propref{secondcountable} that $\partial X$ is second countable and in 
\propref{regular} we showed that that $\partial X$ regular. Also,  $\preceq$ is symmetric by 
\corref{Cor:Symmetric} and hence Hausdorff by \lemref{sym-Haus}. 
Urysohn Metrization Theorem implies that $\partial X$ is metrizable. 

A metrizable topological space is compact if and only if it is sequentially compact. 
Let $\{\bfa_i\}$ be a sequence of equivalence classes in $\partial X$ and consider 
the associated geodesic representatives $\{ \alpha_0^i\}$. Since $X$ is proper, there exists a subsequence 
which we again index as  $\{ \alpha_0^i\}$ that converges point-wise to a geodesic $\beta_0$. 
We will show that $\bfa_i \to \bfb$, that is, 

For a given $r>0$, let $r'>0$ be as in \lemref{Lem:Inclusion-Criterion}. For $i$ large enough, 
\lemref{onestep} implies that $\alpha_0^i$ can be $f_\bfb(1,0)$--redirected to $\beta_0$. 
Again by Equation~\ref{actualtopology},  $f_\bfb(1,0) < F_\bfa(1,0)$ and together with \lemref{Lem:Inclusion-Criterion}, we have that $\bfa_i \in \calU(\bfb, r)$. Since this holds for every $r>0$,
we have $\bfa_i \to \bfb$. Therefore $\partial X$ is compact. 
\end{proof}

\section{A geometric characterization of the Bowditch boundary}
In this section we examine the case when $X$ is a Cayley graph of a relatively hyperbolic group pair $(G, \calP)$ 
where $G$ is a group and $\calP$ is a collection of mono-directional subgroups. We show that, in this case, the 
quasi-redirecting boundary is naturally homeomorphic to the Bowditch boundary of  $(G, \calP)$. 
That is, for this class of groups, the Bowditch boundary can be constructed purely from the geometry of $X$ 
without referencing the group structure or the dynamics of the group action. We begin by reviewing relative 
hyperbolic groups and the Bowditch boundary following \cite{Bowditch}. 

\begin{definition}
Fix a finite generating set $S$ once and for all and let $\Cay(G)$ denote the Cayley graph of $G$ with respect to this 
generating set. We refer to the groups $P \in \calP$ as \emph{peripheral} subgroups. Let $\calA$
be the set of subgraphs of of $\Cay(G)$ associates to cosets of groups in $\calP$. Namely, 
for $P \in \calP$ and $g \in G$, $A_{P, g}$ is the induced subgraph of $\Cay(G)$ with vertex set $gP$. 
We form the \emph{coned-off} Cayley graph, denoted $K(G)$ or simply $K$,  by adding a vertex $*p_A$ for each 
$A \in \calA$, and adding edges of length $\frac 12$ from $*p_A$ to each vertex of $A$. 
Since $\Cay(G)$ is a subgraph of $K$, for any two vertices $v, w \in \Cay(G)$, we have 
\begin{equation}\label{shorter}
d_{K} (v, w) \leq d_{\Cay(G)}.
\end{equation}
\end{definition}

\begin{definition} 
 A graph is \emph{fine} if for each integer $n$, every edge belongs to only finitely many simple cycles of length $n$. If the coned-off Cayley graph is hyperbolic and is fine, then $G$ is \emph{relatively hyperbolic} relative to $\calP$. 
 
 A key property of relative hyperbolic group is the \emph{Bounded Coset Penetration} \cite{Farb98} which we state now. An oriented path $\ell \in K$ is said to \emph{penetrate} $A \in \calA$ if it passes through the cone point $*p_A$ of $A$; 
its \emph{entering} and \emph{exiting} vertices are the vertices immediately before and after $*p_A$ on $\ell$. The path is \emph{without backtracking}
 if once it penetrates $A \in \calA$, it does not penetrate $A$ again.  If for each $q\geq 1$ there is a constant $a = a(q)$ such that if $\zeta$ and $\zeta'$ are $(q, 0)$--quasi-geodesics without backtracking in $K$ and with the same endpoints, then
 
\begin{enumerate}
\item if $\zeta$ penetrates some $A \in \calA$, but $\zeta'$ does not, then the distance between the entering and exiting vertices of $\zeta$ in $A$ is at most $a(q)$; and

\item if $\zeta$ and $\zeta'$ both penetrate $A \in \calA$, then the distance between the entering vertices of $\zeta$ and $\zeta'$ in $A$ is at most $a(q)$, and similarly for the exiting vertices.
\end{enumerate}

\end{definition}

Now we define the Bowditch boundary for relatively hyperbolic groups. Let $\partial K$ denote the Gromov boundary of $K$.  Let $V(K)$ denote the vertex set of $K$, 
 let $V_\infty K=  \{  *p_A, A \in \calA\}$ and let $\triangle K = V_\infty (K) \cup \partial K$. 
 
 \begin{definition}
 For $v, w \in (V(K) \cup \partial K)$, let $[v, w]_K$ denote a geodesic segment (or a geodesic ray) in $K$ 
 connecting $v$ to $w$. Given any $v \in(V(K) \cup \partial K)$ and  a finite set $W\subseteq V(K)$, we write
\[ 
m(v, W) = \Big\{ w \in \triangle K \ST   W \cap [v, w]_K  \subseteq \{v \} 
\quad \text{for every geodesic $[v, w]_K$} \Big\}.
\]
The Bowditch boundary $\partial_B G$ of the relative hyperbolic group $G$ is the set $\triangle K$
equipped with a topology that is generated by the neighborhoods of the form $m(v, W)$. 
\end{definition}

Every geodesic in $K$ can be associated to some geodesic in $\Cay(G)$.  Let $\ell$ be a path in $K$, a \emph{lift} of $\ell$, denoted $\overline{\ell}$, is a path formed from $\ell$ by replacing edges incident to vertices in $V_{\infty}(K)$ with a geodesic in $\Cay(G)$.

\begin{lemma}\label{ellgeodesic}
There exists a uniform bound $\delta_0$ such that, given any  geodesic line or segment $\ell \subset K$ where $|\ell| \geq 3$,  
there exists a geodesic line  or segment $\overline{\ell}_0$ in $\Cay(G)$ such that, when considered as a subset of $K$,  
$\overline{\ell}_0$ is contained in a $\delta_0$--neighborhood of $\ell$ in $K$. If an endpoint of $\ell$ is $*p_A$, then  
$\overline{\ell}_0$  can be chosen to start at any vertex in $N_{M_0}(A)$. Furthermore, every $(1,0)$--transition point of 
$\overline{\ell}_0$ is $\delta_0$--close  in $\Cay(G)$ to some vertex of $\ell$.
\end{lemma}
\begin{proof}
Since vertices in $V_\infty(K)$ are not adjacent to each other in $K$,  $|\ell|\geq 3$ implies $\ell$ contains at least one vertex of $\Cay(G)$. By \cite[Proposition 1.14]{sistopaper}, there exists bounded constants $\qq_1$ such. that every $\ell$ has a lift $\overline{\ell}$ that is a $\qq_1$--quasi-geodesic in $\Cay(G)$. If  $\overline{\ell}$ is a finite quasi-geodesic segment then by Assumption 0, there exists a geodesic connecting its end vertices and by \cite[Lemma 8.13]{Hru10} $\overline{\ell}$ is in a bounded neighborhood of $\ell$ in $K$. Otherwise,  $\overline{\ell}$ is a infinite quasi-geodesic line or ray, then there exists an infinite set of longer and longer geodesic segments $\{ \overline{\ell}_i \}$ in $\Cay(G)$ with endpoints on $\overline{\ell}$ such that the end points converges to both ends of $\overline{\ell}$. Since $|\ell| \geq 3$, $\ell$ contains at least one vertex $v$ not in $*p_A$ and thus by  Proposition~\ref{sis12} (2) there exists a transition point  $v_i \in  \overline{\ell}_i $ such that for all $i$
\[
d_{\Cay(G)}(v, v_i) \leq L.
\] 
Thus by ${\rm Arzel\grave{a}}$--Ascoli Theorem, up to a subsequence, the set  $\{ \overline{\ell}_i \}$ converges to geodesic ray or line which we denote
$\overline{\ell}_0$. By construction the projection of  $\overline{\ell}_0$ to $K$ is in a bounded neighborhood of $\ell$. The last claim follows from \cite[Lemma 8.13]{Hru10} all transition points of $\overline{\ell}_0$ is boundedly close to points of $\ell$.
\end{proof}

Lastly we recall the \emph{relative thin triangle} property \cite[Definition 3.11]{sis12}  and by \cite[Theorem 1.1]{sis12}, the condition holds for geodesic triangles in $\Cay(G)$. 

\begin{proposition} \label{Prop:thin}
There exists a constant $\delta_1$ such that the following holds. 
For point $x, y, z \in \Cay(G)$ consider a geodesic triangle $(x, y, z)$ and let $w$ be a 
$(1,0)$--transition point along $[x,y]$. Then there exists $w' \in [x, z] \cup [z, y]$ such that 
$d_{\Cay(G)}(w, w') \leq \delta_1$. 
\end{proposition}

By Osin-Sapir \cite[Theorem 9.1]{DS05}, $\Cay(G)$ is an asymptotically tree-graded space with respect to 
$\calA$. Assume for the remainder of this section that each $A \in \calA$ is mono-directional and we have that 
$\Cay(G)$ is an ATM space. Thus the quasi-redirecting boundary $\partial G$ exists. 
Recall from Proposition~\ref{Prop:Assump} and Proposition~\ref{Prop:Assump-A}, each equivalence class in 
$\partial G$ contains a central element that is a geodesic ray. 

\begin{definition}\label{defxi}
 Define a map 
\[
\xi: \partial G \to \partial_B G
\] 
as follows. Let $\bfa \in \partial G$ and  $\alpha_0 \in \bfa$ be the central element of $\bfa$. 
If $\alpha_0$ is not transient, then by Lemma~\ref{non-transient-class} there exists a set $A \in \calA$ such that a tail of $\alpha_0$ is in a bounded neighborhood of $A$. In this case we define \[\xi(\bfa) := *p_A.\]

Otherwise, $\alpha_0$ is transient. By the construction and hyperbolicity of $K$, $\alpha_0$ is an unbounded unparameterized quasi-geodesic in $K$
and hence converges to a point $\hat{\alpha_0}$ in $\partial K$. We define
\[\xi(\bfa) : = \hat{\alpha_0}. \]
\end{definition}

\begin{lemma}
The map $\xi \from  \partial G \to \partial_B G $ is a bijection. 
\end{lemma}

\begin{proof}
Suppose $\xi(\bfa) = \xi(\bfb)$ for $\bfa, \bfb \in \partial G$ and let $\alpha_0$ and $\beta_0$ be the central 
elements.  If $\xi(\bfa) = \xi(\bfb)$ is a vertex in $V_\infty(K)$, then $\alpha_0$ and $\beta_0$ are eventually 
contained in a bounded neighborhood of $A$ for some $A \in \calA$. Lemma~\ref{Prop:Assump-A} shows 
$\alpha_0 \sim \beta_0$ in $\partial G$ and hence $\bfa = \bfb$. 

Otherwise, $\xi(\bfa) = \xi(\bfb)$ is a point in $\partial K$. Then $\alpha_0$ and $\beta_0$ are
in a bounded Hausdorff distance from each other in $K$. But, by bounded coset penetration,  $\alpha_0$ and $\beta_0$
have to enter the neighborhood of each $A \in \calA$ near (in $\Cay(G)$--metric) the point $\go_A$. 
Therefore, $\alpha_0$ and $\beta_0$ come boundedly close to each (in $\Cay(G)$--metric)
infinitely often. Part (II) of \lemref{surgery} implies that $\alpha_0 \sim \beta_0$ and hence $\bfa = \bfb$. 

To see that the map is surjective, we also have two cases. Let $v \in V_\infty (K)$ be a point in the Bowditch boundary and let $A$ be the associated set in $\calA$. Let $\alpha$ be a quasi-geodesic ray that  connects $[\go, \go_A]$ with a geodesic ray starting at  $\go_A$ and lie entirely in $A$. By \cite[Lemma 4.19]{DS05} $\alpha$ is a bounded constant quasi-geodesic ray in the class of $\partial A$. Then it follows that $\xi([\alpha]) = v$. Otherwise, let $v$ be a point in $\partial K$. Since $K$ is hyperbolic, there exists an equivalence class of quasi-geodesic rays associated with $v$ and in fact there exists a geodesic 
representative in this class (for instance by ${\rm Arzel\grave{a}}$-Ascoli Theorem), which we refer to as $\alpha$. Since $\alpha$ is a geodesic ray in $K$, by \cite[Proposition 1.14]{sistopaper}, there 
exists a bounded constant quasi-geodesic ray  $\alpha'$ in $\Cay(G)$ that is a lift of $\alpha$. 
We claim that, for $\bfa=[\alpha']$, we have 
\[ \xi(\bfa) = v. \]
Indeed, the central element $\alpha_0$ of $\bfa$ is a geodesic in $\Cay(G)$, and an unparameterized  
quasi-geodesic in $K$. Thus it stays in a bounded neighborhood of $\alpha$ and hence 
converges to $v$. This finishes the proof. 
 \end{proof}

We now show that $\xi$ and $\xi^{-1}$ are both continuous. First we show that for every $v \in \Delta(K)$ 
and every finite subset $W \subset V(K)$, $m(v, W)$ is open in $\partial G$. It suffices to verify this for when 
$W$ has one element as a finite intersection of open sets is open.

\begin{lemma}
For every $\bfb \in \partial G$ and $p \in V(K)$ there exists $r>0$ such that  
\[
\xi(\calU(\bfb, r)) \subset  m(\xi(\bfb), p). 
\]
Therefore, $\xi$ is continuous. 
\end{lemma}

\begin{proof}
Let the geodesic ray $\beta_0$ be the central element of $\bfb$. 
\subsection*{Case I} Assume that  $\bfb$ is transient. 
Consider $\beta_0$ as a subset of $K$ and let $\pi_{\xi(\bfb)}(p)$ be the closest point projection 
of $p$ to $\beta_0$ in $K$ (see \figref{Fig:quad}). 
Since $K$ is hyperbolic, $\pi_{\xi(\bfb)}(p)$ has a bounded diameter in $K$. 
Since $\bfb$ is transient, $\beta_0$ has transition point that are arbitrarily far from $\go$. 
Choose $r>0$ such that, $(\beta_0)_r$ is a $(1,0)$--transition point of $\beta_0$ and 
\begin{equation}\label{case91}
d_K(\go, (\beta_0)_r) \gg d_K(\go, \pi_{\xi(\bfb)}(p))+D(9,0)+2\delta,
\end{equation}
where $\delta$ is the hyperbolicity constant of $K$, $D(9,0)$ is as in Corollary~\ref{RFT} and 
$d_K(\go, \pi_{\xi(\bfb)}(p))$ is the maximum distance in $K$ between any point in $\pi_{\xi(\bfb)}(p)$ to $\go$. 

Let $\bfa \in  \calU(\bfb, r)$ and let $\alpha_0$ be the central element in $\bfa$. Since $(\beta_0)_r$ is a transition 
point, there exists points $q \in \alpha_0$ such that  
 \[
 d(q, (\beta_0)_r) < D((9,0)),
 \]
 Thus $\Norm{q} \geq r - D((9,0))$. Since $K$ is hyperbolic, there exists either a geodesic $\ell$ in $K$ connecting
$\xi(\bfa)$ to $\xi(\bfb)$. The line $\ell$ is an edge in the ideal quadrilateral $((\beta_0)_r, \xi(\bfb), \xi(\bfa), q)$ (see Figure~\ref{Fig:quad})
hence it stays in a bounded neighborhood of 
\[
\beta_0|_{\geq r} \cup \alpha_0|_{\geq r} \cup [(\beta_0)_r, q].
\] 
Hence, $\ell$ is far from $p$ in $K$ and hence does not pass through $p$. Therefore, 
$\xi(\bfa) \in m(\xi(\bfb), p)$.

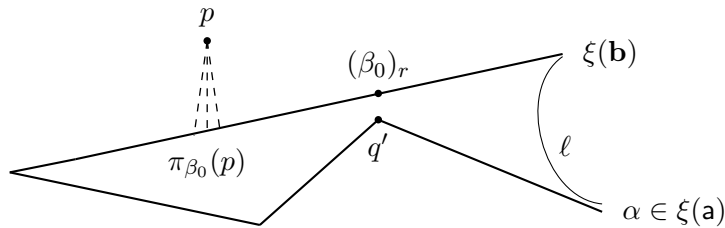
\begin{figure}[H]
\centering
\begin{tikzpicture}[scale=0.7]
\tikzstyle{vertex} =[circle,draw,fill=black,thick, inner sep=0pt, minimum size=2pt] 
\node[vertex] at (4.5,9)[label=above:$p$]{};

\draw[thick] (2,6.75) -- (11.25,8.75);
\draw[thick] (2,6.75) -- (0.75,6.5);
\draw [line width=0.5pt, dashed] (4.5,9) -- (4.25,7.2);
\draw [line width=0.5pt, dashed] (4.5,9) -- (4.5,7.25);
\draw [line width=0.5pt, dashed] (4.5,9) -- (4.75,7.25);
\node at (4.5,7.35) [label=below:$\pi_{\beta_0}(p)$]{};
\draw [thick](0.75,6.5) -- (5.5,5.5);
\draw[thick] (5.5,5.5) -- (7.75,7.5);
\node at  (11.25,8.75)[label=right:$\xi(\bfb)$]{};
\node[vertex] at  (7.75,8)[label=above:$(\beta_0)_r$]{};
\node[vertex] at  (7.75,7.5)[label=below:$q'$]{};
\draw[thick] (7.75,7.5) -- (12,5.75);
\node at (12,5.75)[label=right:$\alpha \in \xi(\bfa)$]{};
 \draw(12,5.9) to[bend left=70,](11.25,8.7){};
\node at (10.8, 7) [label=right:$\ell$]{};
\end{tikzpicture}
\caption{A transition point $(\beta_0)_r$ separates the point $p$ and any geodesic line that connects $\xi(\bfb)$ and $\xi(\bfa)$.}
\label{Fig:quad}
\end{figure}

\subsection*{Case II} Suppose otherwise that $\bfb$ is not transient. By Lemma~\ref{transient-class} 
there exists a unique set $A \in \calA$ such that $\xi(\bfb) = *p_A$.  Let $\beta_0$ be the central element of $\bfb$. Let 
\[
r \gg  2\big(\Norm{\go_A}+ \Norm{p} \big).
\]
Let $\bfa \in \calU(\bfb, r)$ and let $\alpha_0$ be the central element of $\bfa$. Then $\alpha_0$ can be 
$f_\bfb(1,0)$--redirect to $\beta_0$ at radius $r$. Let $\Ge \in A$ be the point near where $\alpha_0$ 
leaves the $M_0$--neighborhood of $A$. 
 
Consider any geodesic segment or ray $\ell$ in $K$ connecting $\xi(\bfa)$ to $*p_A$. By  \cite[Proposition 8.13]{Hru10}, 
$\ell$ enters $N_{\tau(f_\bfb(\qq))}(A)$ at a point that is boundedly close to $\Ge$. Since $*p_A$ is the final point 
in $\ell$, $*p_A$ does not appear in interior of $\ell$ and hence, for any other vertex $x$ in $\ell$, we have 
$\Norm{x} \geq \Norm{\Ge} - D(1,0)$. This implies $\Norm{x} \gg \Norm{p}$ and hence $\ell$ does not pass through $p$. 
Therefore, 
\[
 \bfa \in m(\xi(\bfb), p)
\]
and hence $\calU(\bfb, r) \subset m(\xi(\bfb), p)$. 
\end{proof}

\begin{lemma} 
For any $\bfb \in \partial G$ and $r>0$, there exists a finite set of vertices $W \subset V(K)$ such that 
for every $\bfa$ 
 \[ \xi(\bfa) \in m(\xi(\bfb), W) \qquad \Longrightarrow \qquad \bfa \in \calU(\bfb, r).\]
Therefore, $\xi^{-1}$ is continuous. 
\end{lemma}
\begin{proof}
Let the geodesic ray $\beta_0$ be the central element of $\bfb$.

\subsection*{Case I} Suppose that $\beta_0$ is 
a transient geodesic ray. Given any $r$, let $r'$ be as in \lemref{Lem:Inclusion-Criterion} and 
let $(\beta_0)_{r_1}$ be a transition point of $\beta_0$ where $r_1 \geq 2 r'$. Let $\delta_0$ be the constant from Lemma~\ref{ellgeodesic} and let $\delta_1$ be the constants of Proposition~\ref{Prop:thin} accordingly. Let
\begin{equation}\label{r1}
 r_2 > \delta_0+ \delta_1.
 \end{equation}
Let 
\[
W = B((\beta_0)_{r_1}, r_2) \subset \Cay(G) \subset K
\]  
be the ball of radius $r_2$ in $\Cay(G)$ centered at the vertex $\beta_{r_1}$. However, we consider 
$W$ as a subset of $K$. Note that Since $\Cay(G)$ is proper, $|B((\beta_0)_{r_1}, r_2) | < \infty$.

Consider any element $\bfa \in \partial_B G$ such that $\xi(\bfa) \in m(\xi(\bfb), W)$. That is, there exists a 
bi-infinite geodesic line $\ell$ in $K$ from $\xi(\bfa)$ to $\xi(\bfb)$ avoiding $W$.  Since $\bfb$ is transient $|\ell| \geq 3$, 
by Lemma~\ref{ellgeodesic} there exists  $\overline{\ell}_0$ whose projection to $K$ is boundedly close to $\ell$. If $\alpha_0$ 
is transient, then the geodesics $\alpha_0$ (the central element of $\bfa$) $\beta_0$ and $\overline{\ell}_0$ form an ideal 
geodesic triangle in $\Cay(G)$. If $\alpha_0$ is not transient, then there exists $A_\bfa$ such that $\alpha_0$ eventually stays 
in $A_\bfa$. Then by Lemma~\ref{ellgeodesic} we can chose the starting point of $\overline{\ell}_0$ to be any vertex in 
$\alpha_0 \cap N_{M_0}(A_\bfa)$, and we have that $\alpha_0$, $\beta_0$ and $\overline{\ell}_0$ also form a semi-ideal 
geodesic triangle. \propref{Prop:thin} implies that $(\beta_0)_{r_1}$ is $\delta_1$--close to either a transition point of either 
$\alpha_0$ or a transition point  in$\overline{\ell}_0$. 

If $(\beta_0)_{r_1}$ is $\delta_1$--close to any transition point  of $\overline{\ell}_0$, then by Lemma~\ref{ellgeodesic}, all transition points of $\overline{\ell}_0$ is $\delta_0$ close to points of $\ell$, we have that $(\beta_0)_{r_1}$ is $(\delta_1+\delta_0)$-close to points of $\ell$ in the metric of $\Cay(G)$. This contradicts the choice of $r_2$. Thus $(\beta_0)_{r_1}$ is only $\delta_1$--close to a transition point in $\alpha_0$. This implies that  
$\alpha_0$ can be $(3,0)$--redirected to $\beta_0$ at radius $r_1/2 \geq r'$. Observe that $(3,0) < f_\bfb((1,0)) +(0,1)$. Thus \lemref{Lem:Inclusion-Criterion}
implies that $\bfa \in \calU(\bfb, r)$. 

\subsection*{Case II} Suppose on the other hand $\bfb$ is not transient. 
Let $r_0$ be such that $\beta|_{\geq r_0}$ is in a bounded neighbourhood of $A$ and $\xi(\bfb) = *p_A$.  Let \[f_G(\qq) : = \max\{ (9q, Q), f_A(\qq) \forall  A \in \calA\}.\] Such a maximum function exists because all $A$'s are translations of a finite set of subgroups. For 
a give $r>0$ let $r'$ be as in \lemref{Lem:Inclusion-Criterion} and let 
\[
r'' : = R_\bfb( f_\bfb(1,0)+(0,1), f_\bfb(\qq_0), r').
\]
\[
r_1 =\max \{ 2r'', 2r_0\}.
\]
Let 
\[
R \gg  r_1+\delta_0+D(f_G(1,0))
\]
and let $W = B(\go, R)$ be the ball of radius $R$ in $\Cay(G)$ centered at $\go$. 
Again, $W$ contains finite number of vertices. We now think of $W$ as a subset of $K$. 

Let $\xi(\bfa) \in m(\xi(\bfb), W)$ with $\alpha_0$ as a central element. If $\xi(\bfa) = \xi(\bfb)$, the case is trivial. 
Thus we assume that $\alpha_0$ eventually leaves the mono-directional set $A$ at a point $x$ with a $(1,0)$--transition 
point for $\alpha_0$. 

Since $\xi(\bfa) \in m(\xi(\bfb), W)$, there exists a geodesic ray or segment $\ell$ in $K$ that connects 
$\xi(\bfa)$ to $\xi(\bfb)$ and is disjoint from $W$. Again since $\bfb \neq \bfa$, $|\ell| \geq 3$. Thus by Lemma~\ref{ellgeodesic} there exists geodesic line $\overline{\ell}_0$  whose projection to $K$ is boundedly close to $\ell$. In particular, we can choose for $\alpha_0$ to start at a vertex $x' \neq x$ in $A$. Since $\alpha_0$ and $\overline{\ell}_0$ are both geodesic lines connecting $\xi(\bfb)$ to $\xi(\bfa)$, by Lemma~\ref{sis12}, there exists a $(1,0)$-transition point  $p \in \overline{\ell}_0$ such that 
\[d (p, x) \leq D(f_\bfa(1,0)).\]
Furthermore, by by Lemma~\ref{ellgeodesic},
\[ d(p, \ell) \leq \delta_0.\]
Therefore,
\[d(x, \ell) \leq D(f_\bfa(1,0)) + \delta_0.\]
Since $\ell$ avoid $W$,  all points of $x$ has $\Cay(G)$-norm greater than $R$ and thus 
\[ \Norm{x} \geq R-D(f_\bfa(1,0)) + \delta_0 \geq R-D(f_G(1,0)) + \delta_0 \geq r_1\geq r''.\] 
Since $x \in A$ there exists a $\qq_0$--ray in $A$ which can be $f_\bfb(\qq_0)$ redirected to $\beta_0$ after $r''$.Thus $\alpha_0$ can be  $f_\bfb(\qq_0)$-redirected to $\beta_0$ at 
\[ r'' = R_\bfb( f_\bfb(1,0)+(0,1), f_\bfb(\qq_0), r'),\]
 thus $\alpha_0$ can be $(f_\bfb(1,0)+(0,1))$--redirected to $\beta_0$ at radius $r'$. Thus \lemref{Lem:Inclusion-Criterion} implies that $\bfa \in \calU(\bfb, r)$. And we are done. 
\end{proof}

The combination of all of the preceding results proves the homeomorphism:
\begin{theorem}\label{bowditch-app}
Let  $G$ be a relatively hyperbolic group with respect to subgroups $P_1, P_2,...P_k$. Assume that the Cayley graphs of the subgroups $P_i$'s are mono-directional sets, then  the quasi-redirecting boundary $\partial G$ is homeomorphic to $\partial_B G$.
\end{theorem}

\begin{proof}
Since the map $\xi \from \partial X \to \partial_B X$ is 1-1, onto, and both $\xi$ and $\xi^{-1}$ is continuous, we conclude that $\xi \from \partial G \to \partial_B G$ is a homeomorphism.
\end{proof}

\section{The hairy parking lot}\label{counterexample}
In this section we analyze the quasi-redirecting boundary of a metric space $X$ which demonstrates how
$\partial X$ could be non-Hausdorff when the space is not an ATM space.  The space $X$ we construct 
here is a proper CAT(0) metric space and it is a modification of an example first analyzed by 
Cashen \cite{cashen}. We show that the partial relation in $P(X)$ is not symmetric which means 
$\partial X$ is not Hausdorff. However, $\partial X$ is still compact. We also use to this example to justify 
our definition of the topology for $\partial X$ (see Subsection \ref{Sec:Out-Topology}) . 

Let $\EE^2$ be the Euclidean plane and $B(1)$ be the open ball of radius 1 around the origin. 
Let $Y = \EE^2 - B(1)$. In the polar coordinate in $\EE^2$, we can write 
\[
Y = \{ (\rho, \theta) \st \rho \geq 1\}.
\] 
Let $Z$ be the union of $Y$ and an infinite number of rays, attached at 
a net of points in $Y$ (say, the points in $Y$ with integer Cartesian coordinates). Then $Z$ resembles 
the plane with a hole in the center and infinitely many hairs attached. Let $X = \widetilde Z$ be the universal cover of $Z$.
We fix a lift of the point $(1,0)$ in $Z$ and denote it by $\go$. This defines a polar coordinate for $\widetilde Y \subset X$
where $\theta \in (-\infty, \infty)$, $\rho \geq 1$ and $\go = (1,0)$. 

Note that $X$ is CAT(0). We fix 3 geodesic rays in $X$ (note that the first coordinate is the radius and 
the second coordinate is the angle): 
\begin{align*} 
\alpha_+ \from \RR_+ \to X, && \alpha_+(t) &= (1,t)\\
\alpha_- \from \RR_+ \to X, && \alpha_-(t) &= (1,-t)\\
\zeta \from \RR_+ \to X, && \zeta(t) &= (t+1, 0)
\end{align*} 
Also, for every hair attached at the point $(\rho, \theta)$ (which we refer to as the $(\rho, \theta)$--hair) we pick a quasi-geodesic 
$\gamma_{\rho,\theta}$ exiting this hair, namely, 
\[
\gamma_{\rho, \theta} \from \RR_+ \to X, \qquad \gamma_{\rho, \theta}(t) =
\begin{cases} 
(1, t) & t \in [0, \theta] \\
(t-\theta + 1, \theta) & t \in [\theta, \theta + \rho-1]\\
\text{exits along the hair} & t \geq \theta+\rho -1
\end{cases}.
\]
This is a uniform quasi-geodesic ray and the geodesic ray exiting the $(\rho, \theta)$--hair stays
uniformly close to $\gamma_{\rho, \theta}$, however, $\gamma_{\rho, \theta}$ is easier to describe. 
Let $\bfa_\pm = [\alpha_\pm]$, $\bfz = [\zeta]$ and $\bfc_{\rho,\theta} = [\gamma_{\rho, \theta}]$. 

We claim that 
\[
P(X) = \big\{ \bfa_+, \bfa_-, \bfz \big\} \ \cup \ \big\{ \bfc_{r, \theta} \big\} _{\text{$(r,\theta)$--hairs}} . 
\]
Also, $\bfa_\pm \preceq \bfz$ and otherwise, no other classes are comparable. 
To see this, we note that all quasi-geodesic rays that exit the $(\rho, \theta)$--hair are in the same
class (since the tails of these rays coincide) and are not comparable to other quasi-geodesics (since 
the hairs have only one point of contact with the rest of the space). We only need to show that every 
quasi-geodesic ray that stays in $\widetilde Z$ can be quasi-redirected to $\zeta$. This can be 
done using a logarithmic spiral. We do this explicitly for the $\alpha_+$. For $T \geq 1$, define 
\begin{equation} \label{Eq:log-spiral} 
\alpha_T\from \RR_+ \to X, \quad \alpha_T(t)=
\begin{cases} 
\alpha_+(t) = (1, t) & t \in [0, T] \\
\big((t-T+1), T - \ln(t-T+1) \big) & t \in \left[T,  (e^T+ T -1)\right]\\
\big((t-T+1), 0\big) & t \in \big[(e^T + T -1), \infty\big) 
\end{cases}.
\end{equation} 
To see that $\alpha_T$ is a quasi-geodesic, note that, for $t \in \big[T, (e^T + T -1)\big]$, we have 
\begin{equation} \label{Eq:Norm}
\left\Vert \frac{d}{dt}\alpha_T(t)  \right\Vert \leq 
\left| \frac{d\rho}{dt} \right| + \rho  \left| \frac{d\theta}{dt} \right| 
= 1 + (t-T+1) \cdot \frac{1}{(t-T+1)}  \leq 2. 
\end{equation} 
For other values of $t$, $\Norm{\dot{\alpha_T}(t)} =1$.  
To see the lower bound, we use the fact that 
\[
d\big( (r, \theta), (r', \theta') \big) \geq \max\big(|r-r'|, |\theta- \theta'|\big) \geq \frac 12 \big( |r-r'|+  |\theta- \theta'| \big). 
\] 
For example, we can estimate the distance between $\go$ and $\alpha_T(t)$, for $t \in [T, (e^T + T -1)]$, by
\begin{equation} \label{Eq:Lower}
d (\go, \alpha_T(t)) \geq \frac 12  \left( (t-T)  + (T - \ln(t-T+1)) \right) 
= \frac 12( t- \ln(t-T+1)). 
\end{equation} 
Recall that, for $A > 0$, we have (the right hand side is the equation of a tangent line to the graph of $y=\ln(x)$) 
\[
\ln(x) \leq \frac{x-A}{A} + \ln(A). 
\]
Therefore, setting $x = (t-T+1)$ and $A=2$, we have 
\begin{equation} \label{Eq:ln} 
t- \ln(t-T+1)) \geq t - \left(\frac{t-T-1}{2}\right)  - \ln(2)
\geq \frac t2 - \ln(2). 
\end{equation} 
Combining \eqref{Eq:Lower} and \eqref{Eq:ln} (and rounding up the additive error) we get
\[
d (\go, \alpha_T(t)) \geq \frac t4 -1. 
\]
A similar calculation for other pairs of points shows that $\alpha_T$ is a $(4,1)$--ray. 

Since the family $\alpha_T$ gives uniform quasi-redirections from $\alpha_+$ to $\zeta$ 
at every radius, we can conclude $\bfa_+ \preceq \bfz$. The proof of $\bfa_- \preceq \bfz$ is similar.
However, as it was noted in \cite{cashen}, the geodesic $\alpha_+$ is Morse. Hence, 
any quasi-geodesic with end points on $\alpha_+$ stays in a bounded neighborhood of $\alpha_+$. 
Thus, $\zeta$ cannot be quasi-redirected to $\alpha_+$. The same also holds for $\alpha_-$. 
Hence, $\bfz$, $\bfa_+$ and $\bfa_-$ are distinct classes. 

The same argument as above shows that every $\gamma_{\rho, \theta}$ can be $(4,1)$--redirected to $\zeta$
at radius comparable to $(\rho+\theta)$. That is, for a sequence $\bfc_{\rho_n, \theta_n}$, we have 
\[
\bfc_{r_n, \theta_n} \to \bfz  \qquad \Longleftrightarrow \qquad  (\rho_n + \theta_n) \to \infty.
\]
In fact, if $(\rho+\theta) \geq r$ then $\bfc_{\rho, \theta} \in \calU(\zeta, r)$ and hence
$P(X) -  \calU(\zeta, r)$ is finite. In particular, $\partial X$ is compact.

\subsection{Is there an Out-Topology?} \label{Sec:Out-Topology}
Our goal throughout this paper has been to define a simple natural analogue of Gromov boundary. 
The idea that $P(X)$ should be considered as the set of points at infinity seems quite natural to us
since the directions that can be quasi-redirected to each other do not represent truly different
directions. We then defined a notion of a cone topology on these classes essentially by saying that classes that 
are close to $\bfa \in P(X)$ are those that can still be redirected to $\bfa$ at a large radius. 

However, a priori, one could define the cone topology in the opposite way, namely, classes that are close to $\bfa \in P(X)$ 
are those that $\bfa$ can be redirected to at a large radius. It seems like we have made an arbitrary 
choice here that needs justification. Why should a neighborhood be define by quasi-redirection towards 
$\bfa$ (the In-topology) as opposed to away from $\bfa$ (the Out-topology)? Is there a well-defined notion of Out-topology?

To be more precise, let $X$ be a metric space where the assumptions 0, 1 and 2 hold. 
We can attempt to define neighborhoods for the Out-topology as before: 
For $\bfa \in P(X)$, $r>0$ and $F \from [1,\infty) \times [0,\infty) \to [1,\infty) \times [0,\infty)$ , define 
\begin{align*}
 \Uout(\bfa, r, F) := \Big  \{  \Gb \in P(X) \cup X \, \ST 
&\text{ every $\qq$--ray $\alpha \in \bfa$ can be $F_\bfa(\qq)$--redirected } \\
 & \, \text{to central element $\beta_0 \in \Gb$ at radius $r$.}  \Big \}   
\end{align*}
In \secref{Sec:boundary}, the choice of $F_\bfa$ was not important as long as $F_\bfa$
was large enough due to \lemref{Lem:F} which is why $F$ is not included in the notation
$\calU(\bfa, r)$. However,  the analogue of \lemref{Lem:F}
does not hold for sets $\Uout$. 

To see this in the hairy parking lot example, we consider such neighborhoods around $\bfa_+$ 
for a given function $F$. It turns out, if we enlarge $F$ slightly, say set $F'= F+ (1,0)$, 
then $\Uout(\bfa_+, r, F)$ does not contain $\Uout(\bfa_+, R, F')$ no matter how large $R$ is. 

In fact, let us analyze what subset of $X$ can be reached via a $(q,Q)$--ray that matches $\alpha_+$
up to a radius $r$. Of course, all point $(\rho, \theta)$
where $\theta \geq r$ can be reached. But to make $\theta$ smaller, we have to move 
away from the set $\{\rho=1\}$ at a linear rate. We argue that the
most efficient way is along a log spiral as in \eqref{Eq:log-spiral}. Consider, for $A>0$
and $f \from \RR_+ \to \RR_+$, the ray 
\begin{equation} \label{Eq:log-spiral-2} 
\gamma \from \RR_+ \to X, \qquad \gamma(t)=
\begin{cases} 
\alpha_+(t) = (1, t) & t \in [0, r] \\
\big( A \cdot (t-r+1) , r- f(t-r) \big) & t \geq r
\end{cases}.
\end{equation} 
This path follows $\alpha_+$ up to a radius $r$ and then moves away from the set $\{\rho=1\}$
at a constant speed of $A$ while reducing $\theta$ by a function $f$. Let us analyze how fast $\gamma$
can unwind around $\alpha_+$ (that is, how large $f$ can be) and stay a quasi-geodesic. 
As in \eqref{Eq:Norm}, for $t \geq r$, we have
\[
\Norm{\dot{\gamma}(t)} \leq 
\left| \frac{d\rho}{dt} \right| + \rho  \left| \frac{d\theta}{dt} \right| 
= A + (t-r +1) f'(t-r). 
\]
If $\gamma$ is a $\qq$--ray, $\Norm{\dot{\gamma}(t)}$ has to be bounded with a bound depending on $\qq$. That is, for a constants $B_\qq$
depending on $\qq$, we have  
\[
f'(t-r) \leq \frac {B_\qq}{t-r+1} \quad \Longrightarrow \qquad 
f(t-r) \leq B_\qq \ln(t-r+1). \tag{$f(0) =0$} 
\]
Note that $B_\qq$ is increasing as a function of $\qq$. 

If $(\theta, \rho)$ is a point on $\gamma$, for $\theta \in (-\infty, r]$, we have 
\begin{align*}
\theta =  r- f(t-r) \geq r -  B_\qq \ln(t-r+1) 
&& \Longrightarrow && \ln(t-r+1) &\geq \frac{r- \theta }{B_\qq}  \\
&& \Longrightarrow  && (t -r+1) &\geq e^{\frac{r- \theta}{B_\qq}}. 
\end{align*}
and 
\begin{equation} \label{gamma}
\rho = A \cdot (t-r+1) \geq A \cdot e^{\frac{r- \theta}{B_\qq}} \geq  C_{r,\qq} \cdot e^{\frac{- \theta}{B_\qq}}.
\end{equation}
For some function $C_{r, \qq}$ depending on $r$ and $\qq$. 
Similarly, if $\gamma'$ is a $\qq'$--ray that matches $\alpha_+$ up to a radius $R$
for any $(\rho, \theta)$ on such a geodesic we have 
\begin{equation} \label{gamma'}
\rho \geq  C_{R,\qq'} \cdot e^{\frac{- \theta}{B_{\qq'}}}.
\end{equation}
If $\qq=(q,Q)$, $\qq'= (q', Q')$ and $q' >q$, we have $B_{\qq'}  > B_{\qq}$. Therefore, no matter how large 
$R$ is compared to $r$, as $\theta \to -\infty$, 
\[
C_{R,\qq'} \cdot  e^{\frac{- \theta}{B_{\qq'}}} \leq  C_{r,\qq} e^{\frac{- \theta}{B_\qq}}.
\]
That is, the set of point that could be on $\gamma'$ is not contained in the set of point
that could be on $\gamma$ and, for every $R>r$, we can find a point $(\rho, \theta)$ where 
\eqref{gamma} holds but \eqref{gamma'} does not.  If $(\rho, \theta)$ is any such point, then 
\[
\bfc_{\rho, \theta} \not \in \Uout(\bfa, r, F_\bfa)
\qquad \text{but} \qquad
\bfc_{\rho, \theta} \in \Uout(\bfa, R, F').
\]

This causes some complication with the definition of Out-topology. 
For the sets $\Uout$ to define a topology, in analogy with \lemref{Lemma:Ubeta}, we need to know that 
the following:  There is a family of functions $F_\param$ and, for every $r > 0$ and every $\bfb \in \Uout(\bfa, r\,)$,
there is $R>0$ such that 
\[
 \Uout(\bfb, R , F_\bfb) \subset \Uout(\bfa, r, F_\bfa).
\]
The argument above can be modified to show that this does not hold in general for 
arbitrarily large functions $F_\param$. For example, we can set $\bfa=\bfa_+$, and choose $\gamma_{1,R}$ to be a hair
attached to $\alpha_+$ at a distance $R$ from $\go$. Then 
$\gamma_{1,R} \in \Uout(\bfa_+, R)$, but if $F_{\bfc_{1,R}}$ is even slightly larger than 
$F_\bfa$, we have 
\[
 \Uout(\bfc_{1,R}, R , F_{\bfc_{1,R}}) \not \subset \Uout(\bfa, r, F_\bfa).
\]
That is, if we want to have a Out-Topology, we have to choose the functions  $F_\param$ in a very delicate way
and it is not clear if this is possible. And the Out-topology would change if you changed the functions 
$F_\bfa$. 

Another issue is that, even if it was possible to define an Out-topology, the restriction of the Out-topology
to natural subspaces such as the sublinearly Morse boundary or even the contracting boundary (see \cite{cashenmackay})
does not match the previously defined topologies on these spaces. For the sake of brevity, we leave
the details to reader except to mention that, similar calculations as above shows that, for every $F$, it is possible 
to find a sequence $\rho_n, \theta_n, R_n$ such that 
\[
R_n \to \infty, \quad \theta_n \to -\infty \quad\text{and}\quad  \bfc_{\rho_n, \theta_n} \in \Uout(\alpha_+, R_n , F). 
\]
However, if $\theta_n \to -\infty$, the sequence $\bfc_{\rho_n, \theta_n}$ does not converge to 
$\alpha_+$ in the usual topology of the Morse-boundary (following either \cite{cashenmackay} or \cite{QRT2}).

\section{The Croke-Kleiner Group}
In this section, we show that the hairy parking lot example (Section~\ref{counterexample}) is not an artificial example and indeed the same phenomena  happens in a finitely presented group, in particular, in right angled Artin groups. As a special case, 
we closely examine the Croke-Kleiner Group acting on the universal cover of the  Salvetti Complex (denoted, as usual, by $X$). 
Our main goal is to prove \thmref{Thm:CK}. The proof showcases the rich structure of space
of quasi-geodesic rays in the Croke-Kleiner Group. We also give a classification of the set $P(X)$
by examining the growth rate of the \emph{excursion} of geodesic rays (see \thmref{Thm:Classification}). 
In the hairy parking lot example, every direction is either sublinearly Morse (in fact Morse) or non-Hausdorff point. However, 
it turns out that, in the Croke-Kleiner Group, there exists unexpected equivalence classes of directions that  are rank-one and QI-invariant, but are not sublinearly Morse. In addition, these points are not the non-Hausdorff points of the space. This adds to the knowledge that rank-one is not a good predictor of QI-invariant, as a class of rank-one directions is exhibited to not be QI-invariant in  \cite{Qin16}. Lastly, this example shows that neither 
the Bass-Serre trees nor the hierarchically hyperbolic structures of such spaces is in general accurate in identifying QI-invariant directions at boundary at infinity.

In the interest of brevity, the exact calculation of the constants of quasi-geodesic rays constructed in this section 
are omitted since these calculations are standard and are similar to the arguments presented in the paper so far. 

\subsection*{The Croke-Kleiner Group}\label{Sec:CK}
For background on the Croke-Kleiner group, see \cite{crokekleiner, CK02} and \cite{Qin16}. Here we follow the notation from 
\cite{crokekleiner}. The Croke-Kleiner group is a group $G$
with the following presentation: 
 \[
 G = \big\langle \,  a, b, c, d \st [a, b], [b, c], [c, d]  \, \big\rangle.
 \] 
 Consider a tori complex as follows. 
 
 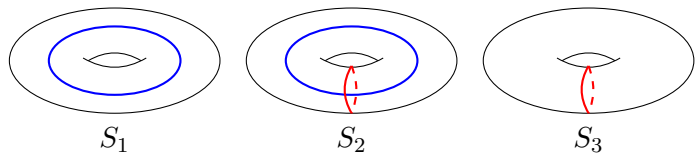
\begin{figure}[ht]
\begin{center}
\begin{tikzpicture}[scale=0.35]

\draw (-1,0) to [bend left] (1,0);
\draw (-1.2,0.1) to [bend right] (1.2,0.1);
\draw (0,0) ellipse (4 and 2);
\draw [thick,blue] (0,0) ellipse (2.5 and 1.3);

\draw (8,0) to [bend left] (10,0);
\draw (7.8,0.1) to [bend right] (10.2,0.1);
\draw (9,0) ellipse (4 and 2);
\draw [thick,blue] (9,0) ellipse (2.5 and 1.3);
\draw [thick,red,bend right=30] (9,-0.2) to (9,-2);
\draw [thick,red,dashed,bend left=20] (9,-0.2) to (9,-2);

\draw (17,0) to [bend left] (19,0);
\draw (16.8,0.1) to [bend right] (19.2,0.1);
\draw (18,0) ellipse (4 and 2);
\draw [thick,red,bend right=30] (18,-0.2) to (18,-2);
\draw [thick,red,dashed,bend left=20] (18,-0.2) to (18,-2);

\node at (0,-3) {$S_1$};
\node at (9,-3) {$S_2$};
\node at (18,-3) {$S_3$};

\end{tikzpicture}
\end{center}
\caption{the tori are glued along the curves isometrically .}
\label{}
\end{figure}
Let $\go$ be a 0-cell. Let $a, b, c, d$ be four oriented 1-cells,
 each an isometric copy of $[0,1]$, whose boundary is $\go$. Let $S_1$ be a torus that is an isometric copy of 
 $[0, 1] \times [0, 1]$, where the opposite sides are glued together to the loops $a$ and  $b$. Likewise, let $S_2$ be 
 a Euclidean square torus with 1-skeleton $b, c$ and $S_3$ be a Euclidean square torus with 1-skeleton $c, d$.  
 This tori complex is the \emph{Salvetti complex} associated with this group. The universal cover of the Salvetti complex is the main space we study 
 in this section and we denote  it by $X$. 
In this section, we show 

\begin{theorem}\label{Thm:CK}
The space $X$ satisfies Assumptions 0-2 and hence $\partial X$ is defined. 
The relation $\preceq$ is not symmetric and $P(X)$ has a unique maximal element. 
Therefore, $\partial X$ is not Hausdorff. 
\end{theorem} 

Note that Assumption 0 holds by \lemref{Lem:cocompact}. To check  Assumptions 1 and 2, we need to identify elements 
of $P(X)$, choose a geodesic representative $\alpha_0$ in each class $\bfa \in P(X)$ and show that every $\qq$--ray 
$\alpha\in \bfa$ can be $f_\bfa(\qq)$--redirected to $\alpha_0$ for some function $f_\bfa$ depending only on the class $\bfa$. 
 
 \subsection*{Blocks and the Bass-Serre tree}
A \emph{block} in $X$ is a convex infinite subset of $X$ that is a lift of either  $S_1 \cup S_2$, 
or $S_2 \cup S_3$. Thus a block is isometric to the universal cover of the Salvetti complex of either of the following groups 
\[ 
G_1 = \big\langle a, b, c \st [a, b], [b, c] \, \big\rangle \qquad \text{or} \qquad 
G_2 =\big \langle b, c, d \st [b, c], [c, d] \, \big\rangle. 
\]
In other words, each blocks comes with a co-compact action of a conjugate copy of either $G_1$ of $G_2$. 
A \emph{flat} in $X$ is a lift of $S_1, S_2$ or $S_3$. Given a pair of blocks, their intersection in $X$ is 
either empty or a \emph{flat} (in fact, always a lift of $S_2$). That is, a flat comes with a compact action of 
a conjugate of the group $ \big\langle b, c \st [b, c] \, \big\rangle = \ZZ^2$.  We refer to these as $bc$--flats. 

One can construct a graph where vertices are blocks and two vertices are 
connected by an edge if and only if two blocks intersect in a flat. The resulting graph, which we denote by $T$, 
is the Bass-Serre tree associated with amalgamated product decomposition of $G$:
\[
G = G_1 *_{ \langle \, b, c \, \st \, [b, c] \, \rangle} G_2.
\]
For $x \in X$, the \emph{$a$--axis} passing through $x$ is the bi-infinite geodesic ray passing through points
$x \cdot a^n$, for $n \in \ZZ$. The $b$--axis, $c$--axis and $d$--axis are similarly defined. 

Note that blocks have a product structure.  A $G_1$--block containing a point 
$x \in X$ is a Euclidean product of the tree generated by $\langle \, a, c \, \rangle$ and the $b$--axis 
centered at the point $x$. Similarly, a $G_2$--block containing $x \in X$ is a Euclidean product of the 
tree generated by $\langle \, b, d \, \rangle$ and the $c$--axis centered at the point $x$. 
We can use this product structure to quasi-redirect quasi-geodesics to each other. 
For example, the arguments in \propref{product} can be used to prove the following lemma. 

\begin{lemma} \label{Lem:block-redirect}
For every $\qq \in [1, \infty]\times [0, \infty)$ and $\rho>0$, there is a 
$\qq' \in [1, \infty]\times [0, \infty)$ such that the following holds. 
Let $B$ be a block, $R \geq (1+\rho) \cdot r >0$ be a pair of radii and $\alpha$ and $\beta$ be two 
$\qq$--rays. Assume $\alpha_r \in B$ and that $\beta|_{\geq R}$ starts at a point in $B$. Then, $\alpha$
can be $\qq'$--redirected to $\beta$ at radius $r$. 
\end{lemma} 

That is, we can transition from $\alpha|_r$ to $\beta|_{\geq R}$ as long as there is buffer between 
them that has a product structure and a thickness that is a linear function of $r$. 

\subsection*{The unique maximal element}
We now start analyzing the set $P(X)$. Let $\go$ also denote the point in $X$ associated with 
the identity element in $G$. Let $\zeta \from \RR_+ \to X$ be a geodesic ray in $X$ such that, 
for a positive integer $n$, 
\[
 \zeta(n) = \go \cdot b^n.
\] 
That is, $\zeta$ follows the positive $b$--axis passing through $\go$. 
We start by showing that $\bfz=[\zeta]$ is the unique maximal element in $P(X)$.
The choice of $b$ here is arbitrary. The geodesic $\zeta_c$ following the $c$--axis is in the same block 
at $\zeta$ and hence $\zeta_c \in \bfz$. Similarly, if $\zeta_a$ and $\zeta_d$ are geodesics following 
the $a$--axis and the $d$--axis respectively, we have $\zeta_a \in [\zeta]$ since $\zeta$ and 
$\zeta_a$ are in the same block and $\zeta_d \in [\zeta_c] = [\zeta]$ since $\zeta_c$ and $\zeta_d$ 
are in the same block. 
Recall that $\qq$--rays of $X$ are assumed to be a continuous and emanating from $\go$.  
 
 \begin{proposition} \label{Prop:zeta-is-top}
 Let $\alpha$ be a $\qq$--ray in $X$. Then $\alpha$ can be $\qq'$--redirected to $\zeta$ where
 $\qq'$ depends only on $\qq$. In particular,  $\alpha \preceq \zeta$. 
 \end{proposition} 
 
\begin{proof} 
After perturbing $\alpha$ by a bounded amount, we can assume that $\alpha$ lies in the $1$--skeleton of
the universal cover of the Salvetti complex. In fact, we assume, there is a sequence $s_n \in \{ a, b, c, d\}$, for $n \in \NN$, 
such that $\alpha(n) = s_1 \dots s_n$. Let $r>0$ be given. We construct a quasi-geodesic ray redirecting 
$\alpha$ to $\zeta$ at radius $r$. 
 
Using \lemref{Lem:block-redirect}, we can find a quasi-geodesic ray $\alpha'$, with constants depending only on 
$\qq$, where $\alpha_r = \alpha'_r$ 
and such that $\alpha'$ eventually follows some $b$--axis. That is, there is $T>0$, such that for $t \geq T$, 
we have $\alpha'(t+1) = \alpha'(t) \cdot b$. 

Again, we can assume that $\alpha'$ lies in the $1$--skeleton of the universal cover of the Salvetti complex.  
We now construct a spiral quasi-geodesic redirecting $\alpha'$ to $\zeta$. This is, in spirit, the same
idea as \eqnref{Eq:log-spiral-2}. Meaning, we try to backtrack along the segment $\alpha'[0,T]$,
however, to stay a quasi-geodesic, we move exponentially far away from $\alpha'$ in each block. 
To be more precise, let 
\[
\alpha'(T) =v_1 w_1 \dots v_k w_k,\quad
\text{where} \quad v_i \not \in G_2  \quad\text{and}\quad  w_i \not \in G_1.
\] 
It is possible that we have to start with a word in $G_2$ or end with a word in $G_1$, but in 
these cases the construction is similar. For $t \in [0,2T]$, we define 
$\gamma(t) = \alpha'(t)$. That is, 
\[
\gamma(2T) = v_1 w_1 \dots v_k w_k \cdot b^T.
\]
We start by moving along the $c$--axis to 
\[
\gamma(4T) = v_1 w_1 \dots v_k w_k \cdot b^T \cdot c^{2T} = 
v_1 w_1 \dots v_k\cdot c^{2T} \cdot w_k b^T. 
\]
This is still a quasi-geodesic. At the worst case moving along the $c$--axis cancels with 
$w_k$. But the presence of $b^T$ in the end ensures that, any point along this segment is at least 
a distance $T$ from $\alpha'[0,T]$. Now we can undo $w_k b^T$ to get:
\[
\gamma(5T+|w_k|) = v_1 w_1 \dots v_k \cdot c^{2T}. 
\]
This is also a quasi-geodesic since the distance between this segment and $\gamma[0, 2T]$
is at least $T$. We can now proceed in this way: first we add a large power of $b$ to get 
\[
\gamma(9T+|w_k|) = v_1 w_1 \dots w_{k-1} v_k \cdot c^{2T} \cdot b^{4T}= 
 v_1 w_1 \dots w_{k-1}  \cdot b^{4T} \cdot v_k c^{2T},
\]
and then we undo $v_k c^{2T}$ to get
\[
\gamma(11T+|w_k| +|v_k|) = v_1 w_1 \dots w_{k-1}  \cdot b^{4T}.
\]
As before, the large power of $b$ ensures that this path is still a quasi-geodesic. We can let 
$T_1 = 11T+|w_k| +|v_k|$ and proceed inductively. To establish the pattern, we repeat the above
procedure one more time. We will add $c^{T_1}$ and undo $w_{k-1} b^{4T}$ to reach
\[
\gamma(2T_1 + 4 T + |w_{k-1}|) =  v_1 w_1 \dots w_{k-2} v_{k-1}  \cdot c^{T_1}.
\]
We then add $b^{2T_1}$ and undo $v_{k-1}  \cdot c^{T_1}$ to reach 
\[
\gamma(5T_1 + 4 T + |w_{k-1}| + |v_{k-1}|) =  v_1 w_1 \dots w_{k-2}  \cdot b^{2T_1}.
\]
This is the same at multiplying $\gamma(T_1)$ on the right by 
\[
c^{T_1} \cdot b^{-4T} \cdot w_{k-1}^{-1} \cdot b^{2T_1} \cdot c^{-T_1} \cdot v_{k-1}^{-1}
\]
one letter at a time. Note that the powers of $b$ do not cancel since $w_k \not \in G_1$. 
And then we set \[T_2 = 5T_1 + 4 T + |w_{k-1}| + |v_{k-1}|\] and proceed as before. When all $v_i$
and $w_i$ disappear, we have $\gamma(T_k)$ is on $\zeta$, and we can continue along
$\zeta$ for $t \geq T_k$. 
\end{proof} 

\subsection*{Recurrent quasi-geodesic rays}
We define a map $\Psi_\alpha \from \RR_+ \to T$ (recall that $T$ is the Bass-Serre tree) as follows: 
Let $A_0 \in T$ be the $G_1$--blocks containing $\go$. Let the $u_1>0$ be the supremum of times
$t$ such that $\alpha(t) \in A_0$ and let $A_1 \not = A_0$ be the block $\alpha$ enters immediately after
it exists $A_0$. We now define $u_i$ and $A_i$ inductively. Let $u_i > u_{i-1}$ be the supremum of times 
$t$ such that $\alpha(t)$ is contained in the block $A_{i-1}$ and let $A_i \not = A_{i-1}$ be the block 
$\alpha$ enters immediately afterwards. Now define 
\[
\Psi_\alpha(t) = A_i \qquad\text{for}\qquad t \in [u_i, u_{i+1}).
\]  
This is a quasi-Lipschitz map where contestants depend on $\qq$. 

It is possible that $u_i = \infty$ for some $i$. This means that $\alpha$ visits the block $A_i$
infinitely many times. In this case, we say $\alpha$ is \emph{recurrent}. 
Otherwise, we say $\alpha$ is \emph{transient}. 

\begin{lemma} 
If $\alpha$ is recurrent, then $\alpha \in \bfz$. 
\end{lemma} 

\begin{proof} 
Let $A_i$ be the terminal block for $\alpha$. Let $\beta$ be a quasi-geodesic that 
eventually stays in $B$ following the $b$--axis. \lemref{Lem:block-redirect} implies that 
$\beta \preceq \alpha$. Also, we can use the argument of \propref{Prop:zeta-is-top}  to build 
a spiral from $\zeta$ to $\beta$.  That is, $\zeta \preceq \beta$. Hence, $\zeta \preceq \alpha$. 
But we already know $\alpha \preceq \zeta$ by \propref{Prop:zeta-is-top} . This finishes the proof. 
\end{proof} 

\subsection*{Geodesic representatives in each class}
Now assume $\alpha$ is not recurrent. Note that the sequence $A_i$ is an embedded path in
$T$, that is, a sequences of vertices where $A_i$ is adjacent to $A_{i+1}$ without
repeating. Since $T$ is a tree, this path is a geodesic in $T$ limiting to some end $\xi$ of $T$. 
By \cite[Corollary 5.26]{CK02}, there is at least one geodesic ray $\alpha_0 \in X$ whose itinerary 
is the sequence $A_i$. We show that $\alpha_0$ can always be redirected to $\alpha$. First, 
we recall the following lemma that follows from \cite[Proposition 3.10]{Qin16} and \cite[Section 4.1]{CK02}.

\begin{lemma} \cite{Qin16, CK02} \label{Lem:Sector} 
Let $x,y$ be two points in $A_{i+2}$, let $[\go, x']$ and $[\go, y']$ be subsegments of geodesic 
segments $[\go, x]$ and $[\go, y]$ such that $x', y' \in A_i$. Let $[\go, z]$ be the intersection
$[\go, x'] \cap [\go, y']$ (possibly $z = \go$). Then, there is an embedded Euclidean triangle
in $X$ where the edges are the geodesic segments $[z, x']$, $[z, y']$ and $[x',y']$. 
\end{lemma} 

\begin{lemma} \label{Lem:geodesic-to-alpha}
For $\alpha$ and $\alpha_0$ as above, we have $\alpha_0 \preceq \alpha$. In particular, all
geodesics with itinerary $A_i$ are in the same class. 
\end{lemma} 

\begin{figure}[ht]
\begin{tikzpicture}[scale=0.9]
 \tikzstyle{vertex} =[circle,draw,fill=black,thick, inner sep=0pt,minimum size=2pt] 
 
 \fill[black!15] (1,0) -- (4,1.5) -- (4,0) -- cycle;
 \node[vertex] (a) at (0,0)  [label=180:$\go$]   {};

 \node  at (11.6 , 2) {$\alpha|_{\geq R}$};
 \draw [thick] (0,0)--(11, 0) node [right] {$\alpha_0$};
 \node[vertex] (x) at (5, 2) [label=-90:$x$]  {};
 \node[vertex] (z) at (1,0) [label=-90:$z$]  {};
 \node[vertex] (x') at (4, 1.5) [label=90:$x'$]  {};
 \node[vertex] (y') at (4, 0) [label=-90:$y'$]  {};
 \node[vertex] (ar) at (2, 0) [label=-90:$(\alpha_0)_r$]  {};

 \draw [dashed, thick](1,0)--(5,2);
 \draw [dashed, thick] (x')--(y');
 \draw [red, thick] (2,0)--(3,0) -- (3,1) -- (4, 1.5);
 
  \pgfsetlinewidth{1pt}
  \pgfsetplottension{.75}
  \pgfplothandlercurveto
  \pgfplotstreamstart
  \pgfplotstreampoint{\pgfpoint{5 cm}{2cm}}  
  \pgfplotstreampoint{\pgfpoint{6cm}{2.8cm}}   
  \pgfplotstreampoint{\pgfpoint{7 cm}{2.2cm}}
  \pgfplotstreampoint{\pgfpoint{8 cm}{3.4cm}}
  \pgfplotstreampoint{\pgfpoint{9 cm}{2.2cm}}
  \pgfplotstreampoint{\pgfpoint{10 cm}{3.2cm}}
  \pgfplotstreampoint{\pgfpoint{11cm}{2cm}}
  \pgfplotstreamend
  \pgfusepath{stroke}  
\end{tikzpicture}
\caption{$\alpha_0$ can be redirected to $\alpha$ at radius $r$.}
\label{Fig:triangle}
\end{figure}
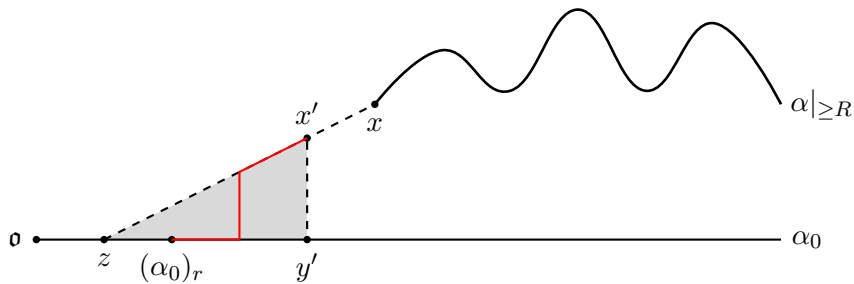

\begin{proof} 
Pick $r>0$. Let $(\alpha_0)_r$ be in $A_i$. We  choose $R>2r$ and $j> i+2$ such that $\alpha|_{\geq R}$
starts at a point $x \in A_j$ and there is a point $y'$ on $\alpha_0 \cap A_{j-2}$ with $\Norm{y'} \geq 2r$. 
It follows from Part (I) of \lemref{surgery} that the concatenation $\gamma=[\go, x] \cup \alpha|{\geq R}$ is a 
$(3q, Q)$--quasi-geodesic. Also, by \lemref{Lem:Sector}, for a point $x' \in [\go, x] \cap A_{j-2}$, 
there is a flat triangle containing points $z, x'$ and $y'$. If $\Norm{z} \geq r$, then $\gamma$ already 
gives a redirection from $\alpha_0$ to $\alpha$ at radius $r$. Otherwise, the Euclidean triangle 
can be used, similar to \thmref{product}, to redirect $\alpha_0$ to $\gamma$ at radius $r$. 
(See \figref{Fig:triangle}). This finishes the proof. 
\end{proof} 

\begin{proposition} \label{Prop:zeta-or-geodesic}
For every $\qq$, there is $\qq'$ such that the following holds. 
Assume $\alpha$ is a $\qq$-ray that is transient with itinerary $A_i$ and let $\alpha_0$ be a geodesic 
with the same itinerary. Then either $\alpha \in \bfz$ or $\alpha$ can be $\qq'$--redirected to $\alpha_0$.
\end{proposition} 

\begin{proof} 
We use Part (II) of \lemref{surgery}. If there is a sequence of radii $r_i \to \infty$ such that 
$d_X((\alpha_0)_{r_i}, \alpha) \leq r_i/2$, then $\gamma$ can be $(9q, Q)$--redirected to $\alpha_0$
and we are done. Otherwise, there is $r_0$ such that, for every $r \geq r_0$, 
\begin{equation} \label{Eq:r2}
d_X((\alpha_0)_r, \alpha) \geq r/2.
\end{equation}
We show that, in this case, $\zeta$ can be redirected to $\alpha$. Since $\alpha$ can always be redirected to 
$\zeta$, this implies that $\alpha \in \bfz$.

By \cite[Section 4.1]{CK02}, the itinerary of $\alpha_0$ determines the combinatorial path. 
Using the language of \cite{CK02}, let $x_k$ be the points on the $bc$-flat shared by $A_k$ and $A_{k+1}$
that determines the quarter-planes on the walls. Let $w_k \in G$ be the word such that $x_k \cdot w_k = x_{k+1}$. 
Then the concatenation of words $w_k$ is a quasi-geodesic rays $\alpha_1$ with uniform constants $\qq_0$ and with 
the same itinerary as $\alpha_0$. That is, we can write the $\qq_0$--ray  $\alpha_1$ as an infinite sequence 
\[
w_1 w_2 w_3 \dots
\qquad\text{where}\qquad w_{2i+1} \not \in G_2  \qquad\text{and}\qquad w_{2i} \not \in G_1,
\]
or perhaps we start with $G_2$ instead of $G_1$. 
Since we assume in this paper without loss of generality that quasi-geodesics are continuous, it follows from the template structure of $X$ that the entry point of $\alpha$ into $A_k$ is aligned 
with $x_k$ either along the $b$--axis of the $c$--axis. That is, there are powers $p_k \in \ZZ$ such that 
\[
\alpha(u_k) = x_{k} \cdot b^{p_k} \qquad\text{for even $k$ and}\qquad 
\alpha(u_k) = x_{k} \cdot c^{p_k} \qquad\text{for odd $k$,}
\]
or, again, with opposite parity. 

Let $\alpha^{A_i}_1$ be the point where  $\alpha_1$ enters $A_i$, note that if a $\qq$-ray satisfies  Equation \eqref{Eq:r2},  $p_k$ is comparable and differ by a uniform multiplicative constant to the
$d(\alpha^{A_k}_1, \alpha(u_k))$. Thus we observe that $p_k$ (powers of $b$ or $c$) comparable and differ by a uniform multiplicative constant to $u_k$. Recall also that 
from $\alpha(u_k)$ to $\alpha(u_{k+1})$, the path $\alpha$ has to replace the powers of $b$ with the powers of $c$ (or vice versa) 
and  also it has to travel along $w_k$.  Therefore, there is a constant $\rho_0>0$ 
depending on $\qq$ such that, 
for $k>0$ large enough, 
\begin{equation} \label{Eq:uk}
u_{k+1} - u_k \geq \rho_0 \cdot (u_k + |w_k|).
\end{equation} 

Let $r>0$ be given. We attempt to build a spiral quasi-geodesic $\gamma$ chasing after $\alpha$
following the outline of the proof of \propref{Prop:zeta-is-top} . 
We build the quasi-geodesic $\gamma$ inductively. For $t \in [0,r]$, define 
$\gamma(t) = \zeta(t) = b^t$. Then we proceed exactly as in the proof of \propref{Prop:zeta-is-top}  trying 
to get to the block $A_i$ as fast possible while staying a quasi-geodesic. To accomplish this, the length of 
time $\gamma$ 
spends in each $A_i$ has to grow exponentially. However, we do not have to double the length 
each time; the powers of $b$ and $c$ in the proof of \propref{Prop:zeta-is-top}  can grow as a 
slow exponential function. And how slow these powers grow will determine the quasi-geodesic constants 
of $\gamma$.

That is, for every $\rho>0$, there is $\qq_\rho$, such that the amount of time $\gamma$
stays in each $A_i$ grows by a factor of $(1+\rho)$ only, and then we can proceed along
the word $w_i$ with a speed of $1/\rho$.  Here $\qq_\rho \to \infty$ as $\rho \to 0$. However, for a fixed $\rho$, 
we can build a spiral that visits the block $A_k$ at time $t_k$ where
\begin{equation} \label{Eq:tk}
t_0 = r \qquad\text{and}\qquad t_{k+1} \leq (1+\rho) \cdot t_k + \rho \cdot |w_k|. 
\end{equation} 
If we choose $\rho < \rho_0$, Equations \eqref{Eq:uk} and \eqref{Eq:tk} imply the sequence $t_k$ grows more 
slowly than the sequence $u_k$. That is, for $k$ large enough 
\[
\Norm{\alpha(u_{k+1})} \geq 2 \Norm{\gamma(t_k)}. 
\]
Now \lemref{Lem:block-redirect} implies that $\gamma$ can be redirected to $\alpha$. Since this is true for every $r$, 
$\zeta$ can be redirected to $\alpha$ and, in view of \propref{Prop:zeta-is-top} , $\alpha \in \bfz$. 
\end{proof} 

\begin{corollary} \label{Cor:Geod-Rep} 
For every quasi-geodesic ray $\alpha$, the class $[\alpha]$ contains a geodesic representative and 
$\alpha$ can be uniformly redirected to this geodesic. 
\end{corollary} 

\begin{proof} 
Every quasi-geodesic can be uniformly quasi-redirected to $\zeta$. Hence, the Corollary holds if $\alpha \in \bfz$.
Otherwise,  by \lemref{Lem:geodesic-to-alpha}, $\alpha_0 \preceq \alpha$ and by \propref{Prop:zeta-or-geodesic} 
$\alpha$ can be uniformly quasi-redirected to $\alpha_0$. That is, $[\alpha_0] = [\alpha]$ and the corollary holds. 
\end{proof} 

\subsection*{Excursion}
For a transient quasi-geodesic ray $\alpha$, whether or not $\zeta$ can be quasi-redirected to $\alpha$ 
depends the excursion function of $\alpha$, namely, the amount of progress $\alpha$ makes in each block.
We now make this precise. We work with the combinatorial path introduces in the previous section. 
That is, we assume $\alpha_1$ is an infinite sequence 
\[
w_1 w_2 w_3 \dots
\qquad\text{where}\qquad w_{2i+1} \not \in G_2   \qquad\text{and}\qquad w_{2i} \not \in G_2.
\]
We say the excursion is \emph{sublinear with respect to the distance in $X$} if 
\[
\lim_{i \to \infty}\frac{ |w_k|}{\sum_{i=1}^{k-1} |w_i|} =0.
\]
In \cite[Theorem A.12]{QRT1}, these class of rays were studied and it was proven than 
they are sublinearly Morse. As we saw in \secref{Sec:kappa}, every element of a sublinearly Morse
class $[\alpha_1]$ has to fellow travel $\alpha_1$ sublinearly. Therefore, for such quasi-geodesics, 
$[\alpha_1] \not = \bfz$. 

\begin{proof}[Proof of \thmref{Thm:CK}]
Assumption 0 follows from \lemref{Lem:cocompact} and Assumption 1 and 2 follow from 
\corref{Cor:Geod-Rep}. By \propref{Prop:zeta-is-top}, $\bfz$ is the unique maximal element in $P(X)$. 
And since the sublinearly Morse direction are different from $\bfz$, $P(X)$ has more than one point. 
That is $\preceq$ is not symmetric and, by \thmref{sym-Haus}, $\partial X$ is not Hausdorff. 
\end{proof} 

However, it turns out that sublinear excursion with respect to the distance in $X$ does not
give a characterization of directions that are different from $\bfz$. Indeed, we say the excursion of 
$\alpha_1$ is \emph{sub-exponential with respect to the distance in $T$} if 
\[
\lim_{i \to \infty}\frac{\log |w_i|}{i} =0.
\]

\begin{theorem} \label{Thm:Classification} 
For a transient quasi-geodesic ray $\alpha_1$, $[\alpha_1] \not = \bfz$ if and only if 
the excursion of $\alpha_1$ is sub-exponential with respect to the distance in $T$. 
\end{theorem} 

\begin{proof} 
First we show that if the excursion of $\alpha_1$ is sub-exponential with respect to the distance in $T$
then $[\alpha_1] \not = \bfz$. We need to show that $\zeta$ cannot be quasi-redirected to $\alpha_1$. 
That is, for every $\qq$, there is an $r>0$ such that there is no $\qq$--ray $\gamma$
with $\gamma|_r = \zeta|_r$ that is eventually equal to $\alpha_1$. 
 
Assume, for contradiction that such $\gamma$ exists for every $r$.  
The proof is similar to \propref{Prop:zeta-or-geodesic} with $\gamma$ in this proof playing the role of 
$\alpha$ in the previous one. As before, let $x_k$ be the points on the $bc$--flat 
shared by $A_k$ and $A_{k+1}$ that determines the quarter-planes on the walls. In particular 
$x_{k+1} = x_k \cdot w_k$. Let $t_k$ be the first time $\gamma$ enters the $bc$--flat
containing $x_k$ and let $\ell_k = d(\gamma(t_k), x_k)$. Then 
$\gamma(t_k) \cdot  x_k^{-1}$ is either a power $b$ or a power of $c$ depending on parity. 
That is, if the difference is a power of $b$ for $x_k$ then $\gamma(t_{k+1}) \cdot  x_{k+1}^{-1}$ 
is a power of $c$ and vice versa. Hence, during the interval $[t_k, t_{k+1}]$, the path $\gamma$
has to add $\ell_{k+1}$ powers of $c$, undo $\ell_k$ powers of $b$ and travel along $w_k$. 
Since $\gamma$ is a $\qq$--ray, there is $\rho_\qq>0$ depending on $\qq$ such that 
\begin{equation} \label{Eq:Low}
t_0 = r \qquad\text{and}\qquad 
t_{k+1}-t_k \geq \rho_\qq \cdot (\ell_k + \ell_{k+1} + |w_k|). 
\end{equation} 
Another way to travel to $\gamma(t_k)$ is to go along the segments $w_i$ to $x_k$ and then
a distance of $\ell_k$ in the $bc$--plane containing $x_k$. Again, since $\gamma$ is 
a $\qq$--ray, we have 
\begin{equation} \label{Eq:High} 
\left(\ell_k + \sum_{i=1}^{k-1} |w_i| \right) \geq \rho_\qq \cdot t_k. 
\end{equation} 
Let 
\[
\rho_1 = \frac{\rho_\qq}2 \qquad\text{and}\qquad 0 < \rho_0 < \rho_1. 
\]
We will show by induction that, for $r$ large enough, 
\begin{equation}  \label{Eq:Estimates} 
t_k \geq r \cdot (1 + \rho_1)^k \qquad\text{and}\qquad
\ell_k \geq r \cdot \frac{\rho_q}2 \cdot (1 + \rho_1)^k. 
\end{equation} 
The base case follows immediately from \eqref{Eq:Low} and \eqref{Eq:High}. 
Assuming these inequalities for $k$, we have (from \eqref{Eq:Low}) that 
\[
t_{k+1} \geq t_k + \rho_\qq \cdot \ell_k \geq
 r \cdot (1 + \rho_1)^k  + r \cdot \frac{\rho_q^2}2 \cdot (1 + \rho_1)^k =  r \cdot (1 + \rho_1)^{k+1}. 
\]
On the other hand, since the excursion of $\alpha_1$ is sub-exponential with respect to the distance in $T$,
there is $C$ depending on $\rho_0$ such that $|w_i| \leq C \cdot (1+\rho_0)^i$. Hence, 
\[
\sum_{i=1}^{k-1} |w_i|  \leq C \sum_{i=1}^{k-1} \frac{(1+\rho_0)^k - 1)}{(1+\rho_0) -1} 
\leq \frac{C}{\rho_0} \cdot (1+ \rho_0)^k.
\]
Choose $r$ large enough such that 
\[
\frac{r \cdot \rho_\qq}2 \geq \frac{C}{\rho_0}. 
\]
Then, by \eqref{Eq:High}, we have 
\[
\ell_{k+1} \geq \rho_q \cdot t_{k+1} - \sum_{i=1}^{k} |w_i|   \geq 
\rho_q \cdot   r \cdot (1 + \rho_1)^{k+1} -\frac{C}{\rho_0} \cdot (1+ \rho_0)^{k+1} 
\geq r \cdot \frac{\rho_q}2 \cdot (1 + \rho_1)^{k+1}. 
\]
This finishes the proof of \eqref{Eq:Estimates}. Since $\rho_0 < \rho_1$, this implies that
$\gamma$ arrives in $A_i$ long after $\alpha_1$ has left $A_i$ and the distance between $\gamma$
and $\alpha_1$ goes to infinity (instead of zero). Thus, there does not exit a $\qq$--ray 
redirecting $\zeta$ to $\alpha_1$. But this holds for every $\qq$. Hence, $\alpha_1 \not \in \bfz$. 

We now prove the other direction. Assume that the excursion of $\alpha_1$ is not sub-exponential with respect to 
the distance in $T$. This implies that 
\begin{equation} \label{Eq:quantifiers} 
\exists \rho>0 \quad \forall r>0 \quad \exists \, k >0 \quad \text{such that} \quad
|w_k| \geq r  \cdot (1+\rho)^k. 
\end{equation} 
Otherwise, 
\[
\forall \rho>0 \quad \exists r>0 \quad \forall k >0 \quad \text{we have} \quad
|w_k| \leq r \cdot \rho \cdot (1+\rho)^k,
\]
which implies 
\[
\forall \rho>0 \quad \lim_{k \to \infty}\frac{\log |w_k|}{k} \leq \rho 
\qquad \Longrightarrow \qquad  \lim_{k \to \infty}\frac{\log |w_k|}{k}=0.
\]

Let $\rho$ be as in \eqref{Eq:quantifiers} and $r>0$ be given. We construct a $\qq$--ray 
$\gamma$ (where $\qq$ depends on $\rho$) that quasi-redirect $\zeta$ to $\alpha_1$
at radius $r$. Let $k$ be the first index where the inequality in \eqref{Eq:quantifiers} holds. That is, 
\[
|w_k| \geq r \cdot (1+\rho)^k
\qquad\text{and}\qquad |w_i| < r \cdot (1+\rho)^i, 
\qquad\text{for}\qquad 1 \leq i <k. 
\]
Set $t_0 = r$ and, for $i=1, \dots, (k-1)$ set 
\[
t_i = r \cdot (1 + \rho) ^i
\qquad\text{and}\qquad
\ell_i = r \cdot \rho \cdot (1 + \rho)^i
\]
Then, assuming $\rho$ is small enough such that 
\[
\rho + \rho(1+\rho) + 1 \leq 2, 
\] 
we have 
\begin{align*} 
t_i + \frac{\rho}2(\ell_i + \ell_{i+1} + |w_i| ) 
& \leq  r \cdot (1 + \rho) ^i  + r\cdot \frac{\rho}2 \Big( \rho \cdot (1 + \rho)^i + \rho \cdot (1 + \rho)^{i+1} +  (1+\rho)^i\Big)  \\
& \leq  r \cdot (1 +  \rho) ^i  + r\cdot \rho \cdot (1 + 4 \rho)^i  =  t_{i+1}. 
\end{align*} 
Therefore, for an appropriate $\qq$ depending on $\rho$, there is enough time for a $\qq$--ray $\gamma$ that starts a distance
$\ell_i$ from $x_i$ to reach a point that is distance $\ell_{i+1}$ from $x_{i_l}$ while maintaining a distance $\ell_i$
from the $\alpha_1$. Since $\ell_i \geq \rho\, t_i$, this path is indeed a quasi-geodesic. Note that 
$\rho/2$ is playing the role of $\rho_\qq$ in \eqref{Eq:Low}. 

We have shown that there exists a quasi-geodesic ray $\gamma$ where $\gamma|_r = \zeta|_r$
such that $\gamma$ that reached $A_k$ at time $t_k$ and a distance from $x_k$ is 
\[
\ell_k = r \cdot \rho \cdot (1+\rho)^k < r \cdot (1+\rho)^k = |w_k|. 
\]
Hence, by \lemref{Lem:block-redirect}, we can connect $\gamma$ to $x_{k+1}$ and continue along $\alpha_1$
while staying a quasi-geodesic. This gives a redirection from $\zeta$ to $\alpha_1$ at radius $r$. 

Since this holds for every $r$, we can conclude that if the excursion of $\alpha_1$ is not sub-exponential with respect 
to the distance in $T$ then $\alpha_1 \in \bfz$. 
\end{proof} 

\subsection*{Description of $P(X)$}
Let us review what we know about $P(X)$ so far. To every transient quasi-geodesic ray $\alpha$, we can associated
an itinerary $A_i$, which is an infinite embedded path in $T$ exiting an end $\xi$. Given such $\xi$, 
there is a preferred quasi-geodesic $\alpha_1$ exiting $\xi$ that passes through the points $x_k$. 
We set $w_k = x_{k+1} \cdot x_k^{-1}$ and refer to $|w_k|$ as the excursion of $\alpha_1$ in the block $A_i$. 
Then $[\alpha_1]$ is different from $\bfz$ if and only if the excursion of $\alpha_1$ is sub-exponential with 
respect to the distance in $T$. That is, every class in $P(X)$ is either $\bfz$ or $[\alpha_1]$ for such $\alpha_1$. 
To finish the description of $P(X)$ we need to show all these classes are different. 

\begin{lemma} \label{Lem:different-or-zeta}
Let $\alpha$ and $\beta$ be transient quasi-geodesic rays with different itineraries. If $\beta \preceq \alpha$
then $\zeta \preceq \alpha$. 
\end{lemma} 

\begin{proof} 
Let $A_i$ be the itinerary for $\alpha$, $B_i$ be the itinerary for $\beta$ and let $k$ be the largest index where
$A_k = B_k$. Let $r>0$ be given. Consider a quasi-geodesic ray $\gamma$ constructed as before
such that $\gamma|_r = \zeta|_r$ and $\gamma$ follows the itinerary $A_i$. Let $t_k$ be the first time
$\gamma(t_k) \in A_k$ and $\ell_k - d(\gamma(t_k), x_k)$. Now choose $R \gg \ell_k$ and consider 
a quasi-geodesic $\beta'$ quasi-redirecting $\beta$ to $\alpha$ at radius $R$. Then $\beta'$ arrives at and leaves 
$A_k$ much later than $\gamma$. Hence, by \lemref{Lem:block-redirect}, we can redirect $\gamma$ to 
$\beta'$, that is, construct a quasi-geodesic ray $\gamma'$ where $\gamma[0, t_k\ = \gamma'[0, t_k]$
and $\gamma'$ is eventually equal to $\beta'$. Since $\beta'$ is eventually equal to $\alpha$, $\gamma'$
quasi-redirects $\zeta$ to $\alpha$ at radius $r$. This can be done for every $r$ with uniform constants. Hence
$\zeta \preceq \alpha$. 
\end{proof} 

\begin{corollary} 
If $\alpha_1$ and $\beta_1$ are transient quasi-geodesic rays whose excursion is sub-exponential with
respect to distance in $T$, then $[\alpha_1] \not = [\beta_1]$. 
\end{corollary} 

\begin{proof} 
If $\alpha_ \preceq \beta_1$, then by \lemref{Lem:different-or-zeta}, $\zeta \prec \alpha_1$ which implies 
$[\alpha_1] = \bfz$. But this contradicts \thmref{Thm:Classification}. 
\end{proof} 

Therefore, we can think of $P(X)$ as a quotient of $\partial T$ where all the direction 
whose excursion is not sub-exponential with respect to distance in $T$ are collapsed into one point $\zeta$. 

\subsection*{An enlargement of the sublinearly Morse boundary}
Finally, we check that not all directions are sublinearly Morse. Choose a sequence $\rho_n >0$ such that 
$\rho_n \to 0$ as $n \to \infty$. 
Fix a constant $C\geq1$ and, for each $n$, choose a power $k_n$ large enough such that 
\begin{equation} \label{Eq:rho_n} 
(1+\rho_n)^{k_n} \geq C \, k_n + \sum_{m=1}^{n-1}(1+\rho_m)^{k_m}. 
\end{equation} 
Then construct a quasi-geodesic $\alpha_1$ as before where 
\[
|w_i| = \begin{cases}
(1+\rho_n)^{k_n} & \text{if $i = k_n$ for some $n$}\\
C & \text{otherwise}
\end{cases}.
\]
\eqnref{Eq:rho_n} implies that, for every $n$, 
\[
\frac{|w_{k_n}|}{\sum_{i=1}^{k_n-1} |w_i|} \geq 1,
\]
meaning, the excursion is not sublinear with respect to distance in $X$. Hence $\alpha_1$ is not
sublinearly Morse. However, since $\rho_n \to 0$, we have 
\[
\frac{\log |w_i|}{i} \leq \rho_n \qquad\text{for $i \geq k_n$}
\qquad\Longrightarrow\qquad 
\lim_{i\to \infty} \frac{\log |w_i|}{i} =0.
\]
That is, the excursion is sub-exponential with respect to distance in $T$.  These seem to be interesting new classes
of quasi-geodesic rays that warrant more study.

\end{document}